\documentclass[a4paper,12pt]{amsart}

\usepackage[margin=1in, headsep=18pt]{geometry}

\usepackage{amsmath,amssymb,amsfonts,amsthm,amscd,amsopn,amstext}
\usepackage{mathtools}
\usepackage{verbatim}

\usepackage[mathscr]{euscript}
\usepackage[bbgreekl]{mathbbol}

\usepackage{tikz}
\usepackage{tikz-cd}
\usetikzlibrary{babel, matrix, shapes, arrows, calc, 3d, decorations, decorations.pathmorphing, through}

\tikzset{ext/.style={circle, draw, inner sep=1pt}, int/.style={circle, draw, fill, inner sep=1pt}, nil/.style={inner sep=1pt}}
\tikzset{exte/.style={circle, draw, inner sep=3pt}, inte/.style={circle, draw, fill, inner sep=3pt}}
\tikzset{diagram/.style={matrix of math nodes, row sep=3em, column sep=2.5em, text height=1.5ex, text depth=0.25ex}}
\tikzset{diagram2/.style={matrix of math nodes, row sep=0.5em, column sep=0.5em, text height=1.5ex, text depth=0.25ex}}
\tikzset{every picture/.append style={baseline=-.65ex}}
\tikzset{de/.style={-latex}} 
\tikzset{ed/.style={latex-}} 

\usepackage{enumitem}
\usepackage{todonotes}
\usepackage{color}

\usepackage{hyperref}

\newcommand{\ldot}{{\:\raisebox{1.5pt}{\selectfont\text{\circle*{1.5}}}}}
\newcommand{\udot}{{\:\raisebox{4pt}{\selectfont\text{\circle*{1.5}}}}}
\newcommand{\ttt}{\text{-}}

\let\le\leqslant
\let\ge\geqslant
\let\leq\leqslant
\let\geq\geqslant

\newcommand{\bN}{\mathbb N}
\newcommand{\bZ}{\mathbb Z}

\newcommand{\bR}{\mathbb R}
\newcommand{\bC}{\mathbb C}
\newcommand{\bF}{\mathbb F}

\newcommand{\rmH}{\mathrm{H}}

\newcommand{\gr}{\mathsf{gr}}

\newcommand{\calO}{\mathcal{O}}
\newcommand{\calP}{\mathcal{P}}

\newcommand{\calF}{{\mathcal{F}}}

\newcommand{\fg}{{\mathfrak g}}
\newcommand{\fgl}{{\mathfrak gl}}
\newcommand{\fh}{{\mathfrak h}}

\newcommand{\GL}{{\mathsf GL}}
\newcommand{\sDiff}{{\mathsf{Dif}}}

\newcommand{\witt}{\mathfrak{witt}}
\newcommand{\vir}{\mathfrak{vir}}
\newcommand{\Res}{\operatorname{Res}}
\newcommand{\ad}{\operatorname{ad}}
\newcommand{\sym}{\mathsf{sym}}
\newcommand{\Conf}{\mathsf{Conf}}

\newcommand{\cA}{{\mathcal A}}

\newcommand{\cE}{{\mathcal E}}

\newcommand{\caH}{{\mathcal H}}

\newcommand{\cT}{{\mathcal T}}

\newcommand{\cV}{{\mathscr V}}
\newcommand{\cW}{{\mathcal W}}

\newcommand{\calT}{{\mathscr{T}}}

\newcommand{\sL}{\mathsf{L}}

\newcommand{\pol}{\mathsf{pol}}
\newcommand{\Span}{\mathsf{span}}

\newcommand{\CoInd}{\mathsf{CoInd}}

\numberwithin{equation}{section}

\newtheorem{theorem}[equation]{Theorem}
\newtheorem*{theorem*}{Theorem}
\newtheorem{proposition}[equation]{Proposition}
\newtheorem*{proposition*}{Proposition}

\newtheorem*{statement*}{Statement}
\newtheorem{lemma}[equation]{Lemma}
\newtheorem*{lemma*}{Lemma}
\newtheorem{corollary}[equation]{Corollary}
\newtheorem*{corollary*}{Corollary}
\newtheorem{conjecture}[equation]{Conjecture}
\newtheorem{definition}[equation]{Definition}
\newtheorem*{definition*}{Definition}

\newtheorem{remark}[equation]{Remark}
\newtheorem*{remark*}{Remark}

\newtheorem*{example*}{Example}

\newtheorem{theoremA}{Theorem}

\newtheorem{conjectureA}[theoremA]{Conjecture}

\providecommand{\eprint}[2][]{\href{http://arxiv.org/abs/#2}{arXiv:#2}}

\begin{document}


\title[Noetherianity of tensor modules and homology]{On Noetherianity of tensor modules and homology of subalgebras of vector fields}

\author[B.\,Feigin]{Boris Feigin}
\address{Boris Feigin: \newline
Einstein Institute of Mathematics, Hebrew University of Jerusalem, Givat Ram, Jerusalem 9190401, Israel}
\email{borfeigin@gmail.com}

\author[A.\,Kanel-Belov]{Alexei Kanel-Belov}
\address{Alexei Kanel-Belov: \newline
Department of Mathematics, Bar-Ilan University, Gelbart Research Institute for Mathematical Sciences, Max and Anna Webb St, Ramat Gan, 5290002, Israel
\newline
Moscow Institute of Physics and Technology (MIPT), 9 Institutskiy Per., Dolgoprudny, Moscow Oblast, Russia, 141701}
\email{kanelster@gmail.com}

\author[A.\,Khoroshkin]{Anton Khoroshkin}
\address{Anton Khoroshkin: \newline
Department of Mathematics, University of Haifa, Mount Carmel, 3498838, Haifa, Israel}
\email{khoroshkin@gmail.com}

\date{}

\maketitle

\begin{abstract}		
	We show that finite tensor products of modules of tensor fields are Noetherian modules over any Lie subalgebra of finite codimension in the Lie algebra of polynomial vector fields on $\mathbb{R}^n$. Moreover, the Hilbert series of their graded submodules are rational functions whose denominators are products of factors of the form $(1-t^m)$.
	As a corollary, we prove the conjecture of I.\,M. Gelfand, announced at the ICM in Nice in 1970, on the finite-dimensionality of the continuous cohomology of graded Lie subalgebras of finite codimension in the Lie algebra of formal vector fields $\cW_n$.
	We also prove that tensor products of modules of densities are Noetherian over the Witt Lie algebra of polynomial vector fields on the punctured line. Moreover, for a weight $\lambda\neq0,1$ every nonzero submodule of the symmetric power $S^r(\cT_\lambda)$ contains $\Delta^{2d}S^r(\cT_\lambda)$, where $\Delta$ is the Vandermonde determinant. As a consequence, we show that a conjecture of Etingof implies the conjecture of Petukhov and Sierra on simplicity of the quotients $U(\vir)/(C-c)$ of the enveloping algebra of the Virasoro algebra.
\end{abstract}

{\small \tableofcontents }

\setcounter{section}{-1}
\section{Introduction}
\label{sec::intro}
This paper answers some questions that remained open since the golden age of Gelfand--Fuchs cohomology.
We deal with the Lie algebra of formal vector fields $\cW_n$ on $\bR^n$ and its subalgebra of polynomial vector fields $\cW_n^{\pol}$. With each $\GL_n$-representation $V_\lambda$ we can associate a vector bundle $E_{\lambda}\to \bR^n$, obtained from the tangent bundle of $\bR^n$ via the representation $V_\lambda$.
We denote by $\cT_{\lambda}$ the $\cW_n$-module (resp.\ $\cW_n^{\pol}$-module) of (polynomial) sections of this bundle, and call it \emph{the module of tensor fields}.
Let us denote by $\calT$ the abelian category of $\sL_1(n)$-modules generated by submodules of the tensor products $\cT_{\lambda_1}\otimes\ldots\otimes\cT_{\lambda_r}$. That is, $M\in\calT$ if and only if $M$ is isomorphic to a subquotient of a finite direct sum of tensor products of modules of tensor fields.

The main result of this paper that follows from the main Theorem~\ref{thm::Noether} can be formulated as follows:
\begin{theoremA}(Corollary~\ref{cor::g::noether})\label{thmA::Noether}
	For any collection of $\GL_n$-modules $V_{\lambda_1},\ldots,V_{\lambda_r}$ and any Lie subalgebra $\fg\subset \cW_n^{\pol}$ of finite codimension, the tensor product $\cT_{\lambda_1}\otimes\ldots\otimes\cT_{\lambda_r}$ is a Noetherian $\fg$-module.
	In particular, the category $\calT$ is locally Noetherian.
	
	Moreover, for any $\fg$-submodule $M\subset \cT_{\lambda_1}\otimes\ldots\otimes\cT_{\lambda_r}$ there exists a finite set $S\subset M$ such that the $2nr$-th component of the PBW filtration of $U(\fg)$ generates $M$:
	$$
	\calF^{2nr}_{PBW}U(\fg) \langle S\rangle = \sum_{j=0}^{2nr}\underbrace{\fg\circ\ldots\circ\fg}_{j}  \langle S \rangle = M.
	$$
\end{theoremA}
As a corollary (via Theorem~\ref{thm::finite::homology}), we prove the longstanding conjecture of Gelfand, stated at the International Congress of Mathematicians in 1970 (see~\cite{Gelfand_Nice}), in the following form:
\begin{theoremA}(Corollary~\ref{cor::Hk::Ld})
	For any given $p$, $d$, and $n$, the continuous cohomology $\rmH^{p}_{c}(\sL_d(n);\bR)$ of the Lie subalgebra $\sL_d(n)\subset \cW_n$ with constant coefficients is finite-dimensional.
\end{theoremA}
Here, the Lie subalgebra $\sL_d(n)$ consists of vector fields $\sum_{i=1}^{n} f_i\frac{\partial}{\partial x_i}$ such that $f_i$ are functions all of whose partial derivatives of order at most $d$ vanish at the origin. The Lie algebras $\sL_d(n)$ are known to be graded, finitely generated Lie algebras of polynomial growth, which are pronilpotent for $d\geq1$. The dimension of the $k$-th graded component of $\sL_d(n)$ depends polynomially on $k$:
$$\sL_d(n) = \prod_{k\geq d} (\sL_d(n))_{(k)} \text{ with } (\sL_d(n))_{(k)}\simeq S^{k+1}(\Bbbk^n)\otimes (\Bbbk^n)^* \ \Rightarrow \ \dim(\sL_d(n))_{(k)} =n \binom{n+k}{k+1}.
$$
Therefore, the universal enveloping algebra $U(\sL_d(n))$ is an associative algebra of intermediate growth. In particular, it fails to admit a finite Gr\"obner basis (see e.g.~\cite{Ufnarovski} for details).
Moreover, Sierra and Walton proved in~\cite{SierraWalton} (by an indirect computation) that the universal enveloping algebra of $\sL_1(1)$ is not Noetherian. It is not hard to show that this implies that the universal enveloping algebras $U(\sL_d(n))$ are also not Noetherian, for all values of $d$ and $n$.
On the other hand, Petukhov and Sierra proved in~\cite[Theorem~4.2]{Petukhov} the Noetherianity of the module of the second symmetric power $S^2(\sL_1(1))$ over the positive Witt algebra $\sL_1(1)$.
The Lie algebra $\sL_1$ is the positive part of the \emph{Witt algebra} $\witt=\Bbbk[z,z^{-1}]\partial_z$ of polynomial vector fields on the punctured line (see Section~\ref{sec::Witt::def}), and the second part of the paper (Section~\ref{sec::Witt}) is devoted to $\witt$ itself.
We write $\cT_\lambda=\Bbbk[z,z^{-1}]\partial^{-\lambda}$ for the module of densities of weight $\lambda$ over $\witt$; note that the adjoint representation is $\witt\simeq\cT_{-1}$.
The strategy used for $\sL_1$ works equally well for $\witt$:
\begin{theoremA}(Theorem~\ref{thm::Noether::Witt})\label{thmA::Witt::Noether}
	For any collection of weights $\lambda_1,\ldots,\lambda_r\in\Bbbk$, the tensor product $\cT_{\lambda_1}\otimes\ldots\otimes\cT_{\lambda_r}$ is a Noetherian $\witt$-module. Moreover, every $\witt$-submodule $M\subset\cT_{\lambda_1}\otimes\ldots\otimes\cT_{\lambda_r}$ is of the form $M=\calF^{2r}_{PBW}U(\witt)\cdot S$ for a finite subset $S\subset M$.
\end{theoremA}
Since $S^r(\sL_1)\subset(\cW_1^{\pol})^{\otimes r}$ and $S^r(\witt)\subset\witt^{\otimes r}$, and the adjoint representations $\cW_1^{\pol}$ and $\witt$ are modules of vector fields $\cT_{-1}$, Theorems~\ref{thmA::Noether} and~\ref{thmA::Witt::Noether} imply that all symmetric powers $S^r(\sL_1)$ and $S^r(\witt)$ are Noetherian. This was conjectured by Petukhov and Sierra, see~\cite[Remark~4.9]{Petukhov} and~\cite[slide~27]{Sierra::AMS2020}. They also conjectured that the two-sided ideals of $U(\sL_1)$ and of $U(\witt)$ satisfy the ascending chain condition, see~\cite[Conjecture~1.3]{Petukhov} and~\cite[slide~19]{Sierra::AMS2020}.
Over $\witt$ the modules of symmetric tensors turn out to be much more rigid than over $\sL_1$. The module $S^r(\cT_\lambda)$ is identified with the space of symmetric Laurent polynomials in $z_1,\ldots,z_r$ (see Subsection~\ref{sec::Witt::U}); let $\Delta=\prod_{i<j}(z_i-z_j)$.
\begin{theoremA}(Theorem~\ref{thm::Witt::submod})
	\label{thmA::Witt::submod}
	Let $\lambda\neq0,1$. For every nonzero $\witt$-submodule $M\subset S^r(\cT_\lambda)$ there exists $d$ such that $\Delta^{2d}S^r(\cT_\lambda)\subset M$. In particular, this holds for the symmetric powers $S^r(\witt)$ of the adjoint representation $\witt=\cT_{-1}$.
\end{theoremA}
The assumption $\lambda\neq0,1$ is necessary: $S^r(\cT_0)$ contains the constants, and $\cT_1$ contains the exact forms. For $\lambda=-1$ the statement is closely related to a conjecture of Petukhov and Sierra, which asserts that the quotient $U(\vir)/(C-c)$ of the universal enveloping algebra of the Virasoro algebra $\vir$ is simple for $c\neq0$ (Conjecture~\ref{conj::PS}; see~\cite{Sierra::Madrid2023}).
We show in Theorem~\ref{thm::Etingof::PS} that this conjecture would follow from a conjecture of Etingof (Conjecture~\ref{conj::Etingof}; private communication), which asserts that the extension of $S^r(\witt)$ by $S^{r-1}(\witt)$ coming from the Virasoro cocycle remains non-split after the restriction to the submodule $\Delta^{2d}S^r(\witt)$.
Theorem~\ref{thmA::Witt::submod} is the main ingredient of that reduction.

At more or less the same time as the first version of this paper was posted on arXiv, Sam, Snowden, and Tosteson proved in~\cite[Theorems~1.3 and~1.4]{SST} that finitely generated \emph{polynomial representations} of the Lie algebra $\sL_0$, that is, subquotients of direct sums of the tensor powers $\cT_0^{\otimes r}$ of the module of functions $\cT_0=\Bbbk[x_1,\ldots,x_n]$, are Noetherian and have rational Hilbert series.
It is worth mentioning that, in order to prove Noetherianity, Sam, Snowden, and Tosteson used methods in~\cite{SST} that differ completely from ours.
Namely, they use representation stability to show Noetherianity for the case $r=\infty$, and then carefully examine what happens to projectives upon restricting to a finite value of $r$.
We go from the other side: we first show in detail why the tensor products we work with are Noetherian for the case $n=1$, and then generalize to all $n$ by splitting the variables.
For $n=1$ the proof is an induction on the number of tensor factors. Its key point is that the image of the universal enveloping algebra $U(\sL_1)$ in the algebra of differential operators contains the multiplications by all symmetric polynomials divisible by a fixed product of the linear forms $z_i$ and $z_i-z_j$ (for the variables of weight $0$ these multiplications come composed with de Rham differentials). Together with the Artin--Rees lemma, this reduces the problem to modules supported on the hyperplanes $\{z_i=0\}$ and $\{z_i=z_j\}$, which are tensor products with fewer factors. As a by-product, we obtain an explicit bound on the PBW degree of the elements of the universal enveloping algebra needed to generate a submodule from finitely many of its elements.
Our proof also gives the rationality of the Hilbert series, which was proved for polynomial representations of $\cW_n^{\pol}$ and of $\cW_n^+$ by Sam, Snowden, and Tosteson in~\cite[Theorem~1.4]{SST}:
\begin{theoremA}(Corollary~\ref{cor::Hilb::Wn})
	Let $d\geq0$, put $D:=\max(d,1)$, and let $M=M'/M''$, where $M''\subset M'$ are graded $\sL_d^{\pol}(n)$-submodules of a finite direct sum of tensor products of at most $r$ modules of tensor fields, with the induced grading. Then the Hilbert series $H_M(t)$ is a rational function whose denominator is a product of factors of the form $(1-t^m)$; more precisely, $H_M(t)\cdot\prod_{k=1}^{r}(1-t^{Dk})^n\in\bZ[t]$.
\end{theoremA}

We also want to emphasize that finite-dimensionality of the cohomology of $\cW_n$ was the starting point for a larger program of computing the appropriate cohomology of the Lie algebras $Vect(X)$ of smooth vector fields on a smooth manifold $X$ (see e.g.~\cite{Guichardet}).
Note that $Vect(X)$ is an infinite-dimensional topological Lie algebra, meaning that its Lie bracket is a continuous map.
The ordinary Lie cohomology of $Vect(X)$ is not known, and we are still far from understanding any such answer for topological Lie algebras of this kind. In particular, there are even significant subtleties concerning the exactness of the tensor product operation for topological Lie algebras (see~\cite{Positselski}).
However, there was significant progress in the 1970s and 1980s in computing the so-called \emph{continuous cohomology} $\rmH^{\udot}_{c}(\ttt)$ -- the cohomology of the subcomplex of continuous cochains of the Chevalley--Eilenberg complex of $Vect(X)$, consisting of continuous skew-symmetric multilinear maps from $Vect(X)$.
The first known computation of this kind is due to Gelfand and Fuchs~\cite{GF-0}.
They described the continuous cohomology of the Lie algebra $\cW_n$ of formal vector fields on $\bR^n$, and demonstrated applications of this theory to the characteristic classes of foliations.
Later, they proved the finite-dimensionality of $H_c^{k}(Vect(X);\bR)$ for all $X$ and all $k$, in~\cite{GF-1,GF-2}. The conjectural description of $H_c^{\udot}(Vect(X);\bR)$ was given independently by R.\,Bott and D.\,Fuchs in 1972, and first proved by Haefliger~\cite{Haefliger}; another proof was given somewhat later by Bott and Segal~\cite{Bott_Segal}.
Recently, B.\,Hennion and M.\,Kapranov extended the Gelfand--Fuchs computations to the case of algebraic vector fields $\calT_X$ on an affine algebraic variety $X$~\cite{Kapranov}.
In both the smooth and the algebraic settings, one can define the $Vect(X)$-modules (resp.\ $\calT_X$-modules) $\cT_\lambda$ of sections of the vector bundle $E_\lambda:=T_{X}\times_X V_\lambda$ associated with the $\GL_n$-module $V_\lambda$. Thus, all of the questions we consider here can equally be posed for a general affine algebraic variety.

We conclude the Introduction with Conjecture~\ref{conj::vect}, particular cases of which follow from the results of this paper.
\begin{conjectureA}
	\label{conj::vect}	
	Let $X$ be a smooth affine algebraic variety. Let $\fg$ be a Lie subalgebra of finite codimension in the Lie algebra of algebraic (polynomial) vector fields on $X$. Then:
	\begin{itemize}
		\item 
		For any collection of finite-dimensional $\GL_n$-modules $V_{\lambda_1},\ldots,V_{\lambda_r}$, the tensor product $\cT_{\lambda_1}\otimes\ldots\otimes\cT_{\lambda_r}$ is a Noetherian $\fg$-module.
		\item
		For any integer $p\geq 0$, the $p$-th Lie algebra homology $\rmH_p(\fg;\Bbbk)$ with constant coefficients is a finite-dimensional vector space.
	\end{itemize}	
\end{conjectureA}

\subsection{Structure of the paper}
The paper is organized as follows.
\\
We recall the necessary notation for the Lie algebras of (formal) vector fields $\cW_n$, its graded subalgebras $\sL_d$, and the module of tensor fields $\cT_{\lambda}$, in Section~\ref{sec::notation::Wn}. We also outline several formulas for the case $n=1$ that are used later on, including the action of $\sL_1$ on tensor products of densities by differential operators and the restrictions to the degenerate hyperplanes (Subsections~\ref{sec::notation::W1} and~\ref{sec::twisted}), and introduce the Witt and Virasoro algebras $\witt$ and $\vir$ together with the modules of densities over $\witt$ (Section~\ref{sec::Witt::def}).
Section~\ref{sec::Noether::T_lambda} is devoted to the $\sL_1$-modules of tensor densities, and Section~\ref{sec::Noether::Wn} to the reduction of the aforementioned Theorem~\ref{thm::Noether} to this case. The main technical result of Section~\ref{sec::Noether::T_lambda} is Theorem~\ref{thm::Noether::L_1}: tensor products of modules of densities with arbitrary weights are Noetherian over $\sL_1$, and every submodule is generated by finitely many of its elements under the action of the elements of $U(\sL_1)$ of PBW degree at most $2r$. As a by-product of the proof, we show that the Hilbert series of graded submodules are rational functions (Theorem~\ref{thm::Hilb} and Corollary~\ref{cor::Hilb::Wn}).
Section~\ref{sec::Witt} is devoted to the Witt algebra. We explain how the proof of Noetherianity is adapted to $\witt$ (Subsection~\ref{sec::Witt::Noether}), describe the image of the universal enveloping algebra $U(\witt)$ in the algebra of differential operators (Subsection~\ref{sec::Witt::U}), prove that for $\lambda\neq0,1$ every nonzero submodule of $S^r(\cT_\lambda)$ contains $\Delta^{2d}S^r(\cT_\lambda)$ (Subsection~\ref{sec::Witt::submod}).
We then give two different applications of Theorem~\ref{thm::Noether} in Section~\ref{sec::applications}.
First, in Section~\ref{sec::Gelfand} we explain the proof of Gelfand's conjecture and some of its generalizations. Second, in Section~\ref{sec::Specht} we show an immediate application to the spaces of commutative polynomials invariant under the simultaneous substitutions $x_i\mapsto p(x_i)$ of all variables: such spaces have the finite basis property, a commutative analogue of the Specht property.
Finally, we discuss the conjectures of Petukhov--Sierra and Etingof (Subsection~\ref{sec::Witt::conj}) on the simplicity of the quotients of the universal enveloping of the polynomial Virasoro algebra $\vir$.

\subsection*{Acknowledgement}
A.\,Kh. would like to thank P.\,Tosteson for pointing out an error in the first suggested proof of the Noetherianity property announced in 2022 and to A.Petukhov to pointing out an error in the second version announced in 2026. A.Kh.  acknowledge V.\,Dotsenko, M.\,Kapranov and D.\,Piontkovsky for stimulating discussions and is also grateful to P.\,Etingof for stimulating discussions, and for pointing out Conjecture~\ref{conj::PS} and the question answered in Theorem~\ref{thmA::Witt::submod}.
A large part of this text was written while A.\,Kh. was visiting Hebrew University in the summer of 2022; A.\,Kh. would like to thank D.\,Kazhdan, Hebrew University, and ERSC grant 669655 for providing excellent working conditions and a stimulating atmosphere.
The research of A.K.-B. was supported by RSF grant  25-21-00884.

{\it{Statement on the use of AI:}}
AI tools (Claude Opus and Sonnet 3.5/4) were used to proofread and polish the language of the final manuscript. In particular, the suggestion to apply the Artin–Rees lemma rather than the Gröbner basis machinery used in the previous version was generated by Claude.

\section{Vector Fields and Tensor Modules}
The main results of this paper remain true for arbitrary fields of characteristic $0$.
However, for most applications, one only needs either $\Bbbk=\bR$ or $\Bbbk=\bC$.

\subsection{General Framework}
\label{sec::notation::Wn}
Denote by $\cW_n$ the Lie algebra of formal vector fields on the $n$-dimensional plane, and by $\cW_n^{\pol}$ its Lie subalgebra of polynomial vector fields. Let us choose coordinates $\langle x_1,\ldots,x_n\rangle$ on the $n$-dimensional plane. We write the elements of $\cW_n$ (respectively $\cW_n^{\pol}$) as sums $\sum_{i=1}^{n}f_i\partial_i$, where $f_i\in \Bbbk[[x_1,\ldots,x_n]]$ (resp. $f_i\in\Bbbk[x_1,\ldots,x_n]$). To each matrix $A=(a_{ij})$ of order $n$ we assign the linear vector field $\sum_{ij} a_{ij} x_i\partial_j$. This assignment defines an isomorphism between the matrix Lie algebra $\fgl_n$ and the Lie subalgebra of linear vector fields. Denoting the tautological $n$-dimensional representation $\mathsf{span}(x_1,\ldots,x_n)$ by $\Bbbk^n$, we obtain the following decomposition of $\fgl_n$-modules:
\begin{equation}
	\label{eq::Wn::decomposition}
	\cW_n \simeq \prod_{k=0}^{\infty} S^k(\Bbbk^n)\otimes(\Bbbk^n)^*; \ \ \cW_n^{\pol} \simeq \oplus_{k=0}^{\infty} S^k(\Bbbk^n)\otimes(\Bbbk^n)^*
\end{equation}
The Lie algebra $\cW_n$ is graded by $\deg(x^{\alpha}\partial_i):=|\alpha|-1$: its component of degree $k\geq-1$ is $(\cW_n)_k:=S^{k+1}(\Bbbk^n)\otimes(\Bbbk^n)^*$, and $[(\cW_n)_k,(\cW_n)_m]\subset(\cW_n)_{k+m}$.
For each $d\geq 0$ we have the graded subalgebras
$$\sL_d(n):=\prod_{k\geq d}(\cW_n)_k=\prod_{k=d+1}^{\infty} S^k(\Bbbk^n)\otimes(\Bbbk^n)^*, \qquad \sL_d^{\pol}(n):=\bigoplus_{k\geq d}(\cW_n)_k$$
of $\cW_n$ and $\cW_n^{\pol}$, respectively; they consist of the vector fields $\sum_if_i\partial_i$ such that all partial derivatives of the $f_i$ of order at most $d$ vanish at the origin. Note that $\sL_d(n)$ is pronilpotent for $d\geq 1$, and $\sL_0(n)/\sL_1(n)\simeq\Bbbk^n\otimes(\Bbbk^n)^*\simeq\fgl_n$.

With each irreducible $\fgl_n$-representation $V_{\lambda}$ we associate an irreducible $\sL_0(n)$-module $V_\lambda$, on which the subalgebra $\sL_1(n)\subset \sL_0(n)$ acts trivially. The coinduced module
\begin{equation}
	\CoInd_{\sL_0(n)}^{\cW_n} V_{\lambda} = \mathrm{Hom}_{U(\sL_0(n))}(U(\cW_n), V_{\lambda})
	\simeq \prod_{k=0}^{\infty} S^k(\Bbbk^n)\otimes V_{\lambda}=\Bbbk[[x_1,\ldots,x_n]]\otimes V_\lambda
\end{equation}
of the Lie algebra $\cW_n$ is the module of formal tensor fields. The direct sum of its homogeneous components
\begin{equation}
	\label{eq::T_lambda::pol}
	\cT_{\lambda}:=\bigoplus_{k=0}^{\infty} S^k(\Bbbk^n)\otimes V_{\lambda}=\Bbbk[x_1,\ldots,x_n]\otimes V_\lambda
\end{equation}
is a $\cW_n^{\pol}$-submodule: it consists of the polynomial sections of the vector bundle $E_\lambda$. Throughout the paper, $\cT_\lambda$ denotes this $\cW_n^{\pol}$-module, which we call the \emph{module of tensor fields}.

\subsection{The one-dimensional case in detail}
\label{sec::notation::W1}
Let us examine the Lie algebra $\cW_1$ of vector fields on the line in more detail. This case is considerably simpler than that of general $n$, since all graded components are one-dimensional. To avoid confusion, we use the variable $z$ for the coordinate on the line. For $u\in\Bbbk[z]$ we write $e_u:=u\,\partial_z$. The Lie algebra $\cW_1^{\pol}$ admits the \emph{monomial} basis $e_k:=e_{z^{k+1}}=z^{k+1}\partial_z$, $k\geq-1$, which satisfies the commutation relations
$$[e_u,e_v]=e_{uv'-u'v},\qquad [e_k,e_m] = (m-k)e_{k+m}.$$
Thus $e_k$ spans the component $(\cW_1)_k$ of Subsection~\ref{sec::notation::Wn}. In the one-dimensional case we work with polynomial vector fields and write
$$\sL_d:=\sL_d^{\pol}(1)=\Span\{e_k\mid k\geq d\}=\{e_u\mid u\in z^{d+1}\Bbbk[z]\},\qquad d\geq0;$$
in particular, $\sL_1=\{e_u\mid u\in z^2\Bbbk[z]\}$.
Each irreducible $\fgl_1$-representation is one-dimensional, $\Bbbk_{\lambda}$, where $\lambda\in\Bbbk$ is a character. The module $\cT_{\lambda}$ can be written symbolically as $\Bbbk[z]\partial^{-\lambda}$, with the following action of $\cW_1^{\pol}$:
$$e_u\cdot g\,\partial^{-\lambda}:= (u g' + \lambda u' g)\,\partial^{-\lambda}, \text{ where }u,g\in\Bbbk[z].$$
Note that, as a vector space, $\cT_{\lambda}$ is isomorphic to the polynomial ring $\Bbbk[z]$, and the symbol $\partial^{-\lambda}$ only affects the extra term in the action. We are interested in the following flat deformation of these modules over the Lie subalgebra $\sL_0$:
\begin{equation}
	\label{eq::T_lmu}
	\cT_{\lambda,\mu}:= \Bbbk[z]\, z^{\mu}\partial^{-\lambda}, \ \lambda,\mu\in\Bbbk
\end{equation}
with action given by
\begin{equation}
	\label{eq::L_1::action}	
	\begin{array}{c}
		e_u \cdot g\, z^{\mu} \partial^{-\lambda} := \bigl(u g' + \mu \frac{u}{z}\, g + \lambda u' g\bigr)z^{\mu} \partial^{-\lambda}; \\ 	
		\text{ in particular: } e_k \cdot z^m (z^\mu\partial^{-\lambda}):= (m+\mu + (k+1)\lambda) z^{m+k} (z^\mu\partial^{-\lambda}). 
	\end{array} 
\end{equation} 
Setting $\mu=0$ recovers the $\cW_1^{\pol}$-module $\cT_{\lambda,0}\simeq \cT_{\lambda}$; however, for $\mu\neq 0$ the action of $e_{-1}$ on $\cT_{\lambda,\mu}$ is not well defined. Nonetheless, for every natural number $k$ we have the following embeddings of modules:
$$
\cT_{\lambda,k} \hookrightarrow \cT_{\lambda} \text{ and } \cT_{\lambda,\mu+k} \hookrightarrow \cT_{\lambda,\mu}.
$$
With each pair of collections of numbers $\bar\lambda=(\lambda_1,\ldots,\lambda_r)$ and $\bar\mu=(\mu_1,\ldots,\mu_r)$ we associate the formal notation $z^{\bar{\mu}}:=\prod_{i=1}^{r}z_i^{\mu_i}$ and
$\partial^{-\bar{\lambda}}:=\prod_{i=1}^{r} \partial_i^{-\lambda_i}$, together with the modules
\begin{equation}
	\label{eq::T_blmu}
	\begin{array}{c}
		{
			\cT_{\bar\lambda}^r:= 
			\cT_{\lambda_1} \otimes \ldots \otimes \cT_{\lambda_r}
			\simeq \Bbbk[z_1,\ldots,z_r] \partial^{-\bar{\lambda}}.
		}\\
		\cT_{\bar{\lambda},\bar{\mu}}^r:= \cT_{\lambda_1,\mu_1} \otimes \ldots \otimes \cT_{\lambda_r,\mu_r}
		\simeq \Bbbk[z_1,\ldots,z_r] z^{\bar{\mu}}\partial^{-\bar{\lambda}}.
	\end{array}
\end{equation}
We will sometimes omit the upper index $r$, which records the number of tensor factors.
\begin{remark}
	\label{rmk::grading}	
	The infinite-dimensional Lie algebra $\sL_1$ is graded, with one-dimensional graded components: $\deg(e_k)=k$. The modules $\cT_{\lambda}$ and $\cT_{\lambda,\mu}$, as well as the tensor products $\cT^r_{\bar{\lambda},\bar{\mu}}$, are graded with respect to the total $z$-degree; for example,
	$$\deg(z^{\bar{a}} z^{\bar{\mu}}\partial^{-\bar{\lambda}}):= a_1+\ldots+a_r.$$ 
	Moreover, all graded components are finite-dimensional.
\end{remark}

We emphasize that, as graded vector spaces, the tensor products $\cT_{\bar\lambda}^r$ and $\cT_{\bar{\lambda},\bar{\mu}}^r$ are isomorphic to the polynomial ring $A:=\Bbbk[z_1,\ldots,z_r]$, where the variable $z_i$ corresponds to the $i$-th tensor factor $\cT_{\lambda_i,\mu_i}$, and the action of $e_k$
in the \emph{monomial} basis $\{z^{\bar{a}}z^{\bar{\mu}}\partial^{-\bar{\lambda}}|\bar{a}=(a_1,\ldots,a_r)\in \bN^{r}\}$ is given by
\begin{equation}
	\label{eq::L::action}
	e_k \cdot z^{\bar{a}} z^{\bar{\mu}}\partial^{-\bar{\lambda}}:= \left(\sum_{i=1}^{r} (a_i+\mu_i + (k+1)\lambda_i) z_i^k\right) z^{\bar{a}} z^{\bar{\mu}}\partial^{-\bar{\lambda}}.
\end{equation}
More generally, the element $e_u\in\sL_0$ acts on $A$ by the differential operator
\begin{equation}
	\label{eq::L_u::q}
	L_u:=\sum_{i=1}^r\ell^{(i)}_u,\qquad \ell^{(i)}_u:=u(z_i)\,\partial_i+\lambda_i\,u'(z_i)+\mu_i\,\frac{u(z_i)}{z_i},\qquad \partial_i:=\frac{\partial}{\partial z_i}.
\end{equation}
For $u\in z\Bbbk[z]$ all coefficients of $\ell^{(i)}_u$ are polynomials. For an arbitrary $u\in\Bbbk[z]$, formula~\eqref{eq::L_u::q} still defines a differential operator with coefficients in $\Bbbk[z_1^{\pm1},\ldots,z_r^{\pm1}]$, and we use this convention in the computations of Section~\ref{sec::Noether::T_lambda}. The operators $\ell^{(i)}_u$ and $\ell^{(j)}_v$ commute for $i\neq j$. For a function $g$ we denote by $m_g$ the operator of multiplication by $g$. We have
\begin{equation}
	\label{eq::L_u::Leibniz}
	L_u=X_u+m_{\omega(u)},\qquad X_u:=\sum_{i=1}^ru(z_i)\,\partial_i,\qquad \omega(u):=\sum_{i=1}^r\Bigl(\lambda_iu'(z_i)+\mu_i\frac{u(z_i)}{z_i}\Bigr),
\end{equation}
where $X_u$ is a derivation; hence $L_u(fg)=X_u(f)\,g+f\,L_u(g)$.

\subsection{Twisted modules and degenerate hyperplanes}
\label{sec::twisted}
Let us also introduce another family of graded $\sL_1$-modules $\cT_{\bar\lambda,\bar\mu;\bar\nu}^r$, depending on a third collection of parameters $\bar\nu= \{\nu_{ij} \mid 1\leq i<j\leq r\}$. As a graded vector space, $\cT^r_{\bar\lambda,\bar\mu;\bar\nu}$ is isomorphic to the polynomial ring $A$, whose elements can be written in the form
$$
\left\{g(z_1,\ldots,z_r)\Bigl[\prod_{1\leq i<j\leq r}(z_i-z_j)^{\nu_{ij}}\Bigr] z^{\bar\mu} \partial^{-\bar\lambda} \ \middle| \ g(z_1,\ldots,z_r)\in \Bbbk[z_1,\ldots,z_r]\right\}.
$$
For $u\in\Bbbk[z]$ we put
\begin{equation}
	\label{eq::c_ell}
	c_{z_i}(u):=\frac{u(z_i)}{z_i},\qquad c_{z_i-z_j}(u):=\frac{u(z_i)-u(z_j)}{z_i-z_j}=c_{z_j-z_i}(u),
\end{equation}
so that $\omega(u)=\sum_{i=1}^r\bigl(\lambda_iu'(z_i)+\mu_i\,c_{z_i}(u)\bigr)$ in~\eqref{eq::L_u::Leibniz}.
The action of the vector fields is given by a sum of differential and divided-difference operators: under the identification of $\cT^r_{\bar\lambda,\bar\mu;\bar\nu}$ with $A$ given by the coefficient $g$, the element $e_u\in\sL_1$ acts by the operator
\begin{equation}
	\label{eq::L::action::3}
	L_u+\sum_{i<j}\nu_{ij}\,m_{c_{z_i-z_j}(u)},
\end{equation}
where $L_u$ is the operator~\eqref{eq::L_u::q}.
Note that the modules $\cT_{\bar\lambda,\bar\mu;\bar\nu}^{r}$ are well defined for all values $\lambda_i,\mu_i,\nu_{ij}\in\Bbbk$. If the $\nu_{ij}$ (respectively, the $\mu_i$) are nonnegative integers, then the formal expressions $(z_i-z_j)^{\nu_{ij}}$ (respectively, $z_i^{\mu_i}$) are honest polynomials, and we have the following embeddings:
$$
\cT_{\bar\lambda,\bar\mu;\bar\nu}^r \stackrel{\nu_{ij}\in\bZ_{\geq0}}{ \subset } \cT^r_{\bar\lambda,\bar\mu} \stackrel{\mu_{i}\in\bZ_{\geq0}}{ \subset } \cT^{r}_{\bar\lambda}.
$$
Substituting $z_1=z_2=z$ into the polynomials yields the following surjective morphisms of modules:
\begin{gather}
	{
		\label{eq::substitution::maps}	
		\begin{array}{c}
			\pi_{+}: \cT_{\lambda_1,\mu_1}\otimes \cT_{\lambda_2,\mu_2} \twoheadrightarrow \cT_{\lambda_1+\lambda_2,\mu_1+\mu_2}, \\
			\pi_{+}: g(z_1,z_2)z_1^{\mu_1} z_2^{\mu_2}\partial_1^{-\lambda_1}\partial_2^{-\lambda_2} \mapsto g(z,z)z^{\mu_1+\mu_2}\partial^{-\lambda_1-\lambda_2}
		\end{array}		
	}
	\\
	{
		\label{eq::T_lmn}
		\begin{array}{c}	
			\pi_{+}': \cT^2_{(\lambda_1,\lambda_2),(\mu_1,\mu_2);\nu} 
			\twoheadrightarrow \cT_{\lambda_1+\lambda_2+\nu,\mu_1+\mu_2} \\
			\pi_{+}': g(z_1,z_2)(z_1-z_2)^{\nu} z_1^{\mu_1} z_2^{\mu_2}\partial_1^{-\lambda_1}\partial_2^{-\lambda_2} \mapsto g(z,z) z^{\mu_1+\mu_2} \partial^{-\lambda_1-\lambda_2-\nu}
		\end{array}
	}
\end{gather}
The kernels of these maps are easily described: they consist of polynomials divisible by $(z_1-z_2)$. We therefore obtain the following short exact sequences of $\sL_1$-modules:
$$
0\rightarrow \cT^2_{(\lambda_1,\lambda_2),(\mu_1,\mu_2);\nu+1} \rightarrow \cT^2_{(\lambda_1,\lambda_2),(\mu_1,\mu_2);\nu} \rightarrow
\cT_{\lambda_1+\lambda_2+\nu,\mu_1+\mu_2} \rightarrow 0.
$$

The following definition and lemma generalize these short exact sequences; they explain how the induction on the number of variables in Section~\ref{sec::Noether::T_lambda} works.

\begin{definition}
	\label{def::class::V}
	For $r\geq0$ let $\cV_r$ be the class of all $\sL_1$-modules $\cT^r_{\bar\lambda,\bar\mu}$ with $\bar\lambda,\bar\mu\in\Bbbk^r$ (for $r=0$ we get $\cT^0=\Bbbk$ with the trivial action). A \emph{degenerate form} of $\cT^r_{\bar\lambda,\bar\mu}$ is one of the linear forms $z_i$, $1\leq i\leq r$, and $z_i-z_j$, $i\neq j$. The zero sets $\{z_i=0\}$ and $\{z_i=z_j\}$ of the degenerate forms are called \emph{degenerate hyperplanes}.
\end{definition}

\begin{lemma}
	\label{lem::strata}
	Let $\cT\in\cV_r$, let $D$ be a product of degenerate forms of $\cT$ (possibly $D=1$), and let $\ell$ be a degenerate form of $\cT$. Then $D\cT\supset D\ell \cT$ are $\sL_1$-submodules of $\cT$, and the quotient $D\cT/D\ell \cT$ is isomorphic to a submodule of a module of $\cV_{r-1}$.
\end{lemma}
\begin{proof}
	If $u\in z^2\Bbbk[z]$, then $c_\ell(u)\in A$ for every degenerate form $\ell$, see~\eqref{eq::c_ell}. Since $\ell^{(i)}_u(z_ig)=u(z_i)g+z_i\,\ell^{(i)}_ug$ and $(\ell^{(i)}_u+\ell^{(j)}_u)((z_i-z_j)g)=(u(z_i)-u(z_j))g+(z_i-z_j)(\ell^{(i)}_u+\ell^{(j)}_u)g$, while all other operators $\ell^{(k)}_u$ commute with the multiplication by $\ell$, we have for all $g\in A$:
	\begin{equation}
		\label{eq::strata::twist}
		L_u(\ell g)=\ell\,\bigl(L_ug+c_\ell(u)\,g\bigr).
	\end{equation}
	Iterating~\eqref{eq::strata::twist} over the factors of $D=\ell_1\cdots\ell_s$, we get $L_u(Dg)=D\bigl(L_ug+c_D(u)g\bigr)$ with $c_D:=\sum_{t=1}^sc_{\ell_t}$. Therefore $D\cT$ is a submodule of $\cT$, and the map $g\mapsto Dg$ identifies $D\cT$ with the space $A$ on which $e_u$ acts by the operator $L_u+m_{c_D(u)}$. If $D=\pm\prod_iz_i^{n_i}\prod_{i<j}(z_i-z_j)^{\nu_{ij}}$, this is the module $\cT^r_{\bar\lambda,\bar\mu+\bar n;\bar\nu}$ of~\eqref{eq::L::action::3}. Under this identification $D\ell \cT$ corresponds to $\ell A$; hence $D\ell \cT$ is a submodule as well, and $D\cT/D\ell \cT$ is isomorphic to $A/\ell A$ with the action induced by the operators $L_u+m_{c_D(u)}$.

	\emph{The case $\ell=z_i$.} The restriction $g\mapsto g|_{z_i=0}$ identifies $A/z_iA$ with the ring of polynomials in the variables $z_k$, $k\neq i$. Since $u(0)=u'(0)=0$ and $u(z)/z$ vanishes at $z=0$ for $u\in z^2\Bbbk[z]$, the operator $\ell^{(i)}_u$ induces zero on $A/z_iA$. Moreover, $c_{z_i}(u)|_{z_i=0}=0$ and $c_{z_i-z_k}(u)|_{z_i=0}=c_{z_k}(u)$, while the operators $\ell^{(k)}_u$ and the functions $c_{z_k}(u)$, $c_{z_k-z_l}(u)$ with $k,l\neq i$ do not change. Hence $A/z_iA$ is the module $\cT^{r-1}_{\bar\lambda',\bar\mu';\bar\nu'}$ in the variables $z_k$, $k\neq i$, with
	$$\lambda'_k=\lambda_k,\qquad \mu'_k=\mu_k+n_k+\nu_{ik},\qquad \nu'_{kl}=\nu_{kl}\qquad (k,l\neq i).$$

	\emph{The case $\ell=z_i-z_j$.} The restriction $g\mapsto g|_{z_j=z_i}$ identifies $A/(z_i-z_j)A$ with the ring of polynomials in the variables $z_k$, $k\neq j$. Since $\partial_i(g|_{z_j=z_i})=(\partial_ig+\partial_jg)|_{z_j=z_i}$, we get
	$$
	\bigl((\ell^{(i)}_u+\ell^{(j)}_u)g\bigr)\big|_{z_j=z_i}=\Bigl(u(z_i)\partial_i+(\lambda_i+\lambda_j)u'(z_i)+(\mu_i+\mu_j)\frac{u(z_i)}{z_i}\Bigr)\bigl(g|_{z_j=z_i}\bigr).
	$$
	Moreover, $c_{z_i-z_j}(u)|_{z_j=z_i}=u'(z_i)$, $c_{z_j}(u)|_{z_j=z_i}=c_{z_i}(u)$ and $c_{z_j-z_k}(u)|_{z_j=z_i}=c_{z_i-z_k}(u)$. Hence $A/(z_i-z_j)A$ is the module $\cT^{r-1}_{\bar\lambda',\bar\mu';\bar\nu'}$ in the variables $z_k$, $k\neq j$, where the merged variable $z_i$ carries the weights
	$$\lambda'_i=\lambda_i+\lambda_j+\nu_{ij},\qquad \mu'_i=\mu_i+\mu_j+n_i+n_j,$$
	and $\lambda'_k=\lambda_k$, $\mu'_k=\mu_k+n_k$, $\nu'_{ik}=\nu_{ik}+\nu_{jk}$, $\nu'_{kl}=\nu_{kl}$ for $k,l\neq i,j$.

	In both cases all twists $\nu'_{kl}$ are nonnegative integers. Applying~\eqref{eq::strata::twist} to the module $\cT^{r-1}_{\bar\lambda',\bar\mu'}\in\cV_{r-1}$, we see that the multiplication by $\prod_{k<l}(z_k-z_l)^{\nu'_{kl}}$ is an injective homomorphism $\cT^{r-1}_{\bar\lambda',\bar\mu';\bar\nu'}\hookrightarrow\cT^{r-1}_{\bar\lambda',\bar\mu'}$.
\end{proof}

\subsection{The Witt and Virasoro algebras}
\label{sec::Witt::def}
In Section~\ref{sec::Witt} we work with the Lie algebra of polynomial vector fields on the punctured line, which we call the \emph{polynomial Witt algebra}:
$$
\witt:=\Bbbk[z,z^{-1}]\,\partial_z .
$$
The word ``polynomial'' means that only finite Laurent polynomials are allowed, as opposed to the completed Witt algebra $\Bbbk((z))\partial_z$.
As in Subsection~\ref{sec::notation::W1}, for $u\in\Bbbk[z^{\pm1}]$ we write $e_u:=u\,\partial_z$, and $e_k:=e_{z^{k+1}}=z^{k+1}\partial_z$ for $k\in\bZ$. Then
$$
[e_u,e_v]=(uv'-u'v)\,\partial_z, \qquad [e_k,e_m]=(m-k)\,e_{k+m}\quad (k,m\in\bZ),
$$
that is, $\witt$ has the same commutation relations as in Subsection~\ref{sec::notation::W1}, but now for all integers $k,m$.
The Lie algebra $\witt$ is graded by $\deg e_k=k$; it contains the Lie algebra $\cW_1^{\pol}=\langle e_k \mid k\geq -1\rangle$ of polynomial vector fields on the line and its graded subalgebras $\sL_d=\langle e_k\mid k\geq d\rangle$ of Subsection~\ref{sec::notation::W1}; the Lie algebra $\sL_1$ is known as the \emph{positive Witt algebra}.

The \emph{Virasoro Lie algebra} $\vir$ is the central extension of $\witt$ by the Gelfand--Fuchs cocycle
$$
\psi(e_u,e_v):=\frac{1}{12}\Res_{z=0}\,u'''v\,dz, \qquad \psi(e_k,e_m)=\frac{k^3-k}{12}\,\delta_{k+m,0}.
$$
That is, $\vir=\witt\oplus\Bbbk C$, where $C$ is central and
$$
[e_u,e_v]_{\vir}=[e_u,e_v]+\psi(e_u,e_v)\,C, \qquad\text{i.e.}\qquad [e_k,e_m]_{\vir}=(m-k)\,e_{k+m}+\frac{k^3-k}{12}\,\delta_{k+m,0}\,C.
$$
Its cohomology class generates $H^2(\witt;\Bbbk)$. Note that the cocycle $\psi$ vanishes on the subalgebra $\cW_1^{\pol}$: a nonzero value $\psi(e_k,e_{-k})$ requires $|k|\geq2$. Thus the central extension is invisible on polynomial vector fields on the line, and only appears for the Laurent polynomial Witt algebra.

For every $\lambda\in\Bbbk$ the \emph{module of densities} of weight $\lambda$ over $\witt$ is the space of Laurent polynomials
$$
\cT_{\lambda}:=\Bbbk[z,z^{-1}]\,\partial^{-\lambda}, \qquad e_u\cdot g\,\partial^{-\lambda}:=(ug'+\lambda u'g)\,\partial^{-\lambda},\quad u,g\in\Bbbk[z^{\pm1}],
$$
or, in the monomial basis, $e_k\cdot z^m\partial^{-\lambda}=(m+(k+1)\lambda)\,z^{m+k}\partial^{-\lambda}$ for all $k,m\in\bZ$.
These are the formulas~\eqref{eq::L_1::action} with $\mu=0$, now defined for all integers $k$ and $m$.
The polynomial module $\Bbbk[z]\partial^{-\lambda}$ of Subsection~\ref{sec::notation::W1} is a $\cW_1^{\pol}$-submodule of $\cT_{\lambda}$, and we keep the notation $\cT_\lambda$ for the $\witt$-module in Section~\ref{sec::Witt}.
For $\lambda=0$ we obtain functions, for $\lambda=1$ we obtain differential forms, and for $\lambda=-1$ we obtain vector fields, so that the adjoint representation is the module of densities of weight $-1$:
$$\witt\simeq\cT_{-1}, \qquad u\,\partial_z\mapsto u\,\partial^{1}.$$
The tensor products $\cT^r_{\bar\lambda}:=\cT_{\lambda_1}\otimes\ldots\otimes\cT_{\lambda_r}$ are identified, as in~\eqref{eq::T_blmu}, with the spaces of Laurent polynomials $\Bbbk[z_1^{\pm1},\ldots,z_r^{\pm1}]\partial^{-\bar\lambda}$. The Lie algebra $\witt$ acts on them by the operators $L_u=\sum_{i=1}^r\bigl(u(z_i)\partial_i+\lambda_iu'(z_i)\bigr)$ of~\eqref{eq::L_u::q} with $\bar\mu=0$, now for all $u\in\Bbbk[z^{\pm1}]$, and formula~\eqref{eq::L::action} (with $\bar\mu=0$) holds for all $k\in\bZ$ and $\bar a\in\bZ^r$.
Note that, in contrast with Remark~\ref{rmk::grading}, the graded components of $\cT^r_{\bar\lambda}$ over $\witt$ are infinite-dimensional for $r\geq2$.

\section{Noetherianity of $\sL_1$-modules of tensor modules}
\label{sec::Noether::T_lambda}

In this section we prove that tensor products of the $\sL_1$-modules of densities $\cT_{\lambda,\mu}$ are Noetherian for arbitrary weights, with an explicit bound on the PBW degree (Theorem~\ref{thm::Noether::L_1}), and we deduce the rationality of the Hilbert series of their graded submodules (Theorem~\ref{thm::Hilb}). In Section~\ref{sec::Noether::Wn} these results are extended to the Lie algebras $\sL_d(n)$ and to the Lie subalgebras of finite codimension in $\cW_n^{\pol}$.

\subsection{Noetherianity with a given PBW bound and strategy of the proof}
\label{sec::L_1}

For a Lie algebra $\fg$, a $\fg$-module $M$ and a subset $S\subset M$ we denote by
\begin{multline*}
\calF^{p}_{PBW}U(\fg)\cdot S = \sum_{j=0}^{p}\underbrace{\fg\circ\ldots\circ\fg}_{j}\,\langle S\rangle = \\
= \Span\bigl\{x_1x_2\cdots x_j\cdot s\ \bigm|\ 0\leq j\leq p,\ x_1,\ldots,x_j\in\fg,\ s\in S\bigr\}
\end{multline*}
the subspace obtained from $S$ by the action of the $p$-th term of the PBW filtration of the universal enveloping algebra $U(\fg)$.

\begin{definition}
	\label{def::property::P}
	A module $M$ over a Lie algebra $\fg$ is said to have \emph{property $\calP_p$} if every $\fg$-submodule $N\subset M$ is of the form $N=\calF^{p}_{PBW}U(\fg)\cdot S$ for a finite subset $S\subset N$. We say that the class $\cV_r$ of Definition~\ref{def::class::V} has property $\calP_p$ if every module of $\cV_r$ has it.
\end{definition}

\begin{lemma}
	\label{lem::closure}
	Let $M$ be a $\fg$-module.
	\begin{enumerate}[label=(\alph*), itemsep=0pt, topsep=0pt]
		\item If $M$ has property $\calP_p$, then $M$ is Noetherian, and $M$ has property $\calP_{p'}$ for every $p'\geq p$.
		\item If $M$ has property $\calP_p$, then so do all submodules and subquotients of $M$.
		\item Let $M'\subset M$ be a submodule. If $M'$ and $M/M'$ have property $\calP_p$, then $M$ has property $\calP_p$. In particular, if $M$ admits a finite filtration by submodules whose successive subquotients have property $\calP_p$, then $M$ has property $\calP_p$.
		\item Let $\fh\subset\fg$ be a Lie subalgebra. If $M$ has property $\calP_p$ as an $\fh$-module, then $M$ has property $\calP_p$ as a $\fg$-module.
	\end{enumerate}
\end{lemma}
\begin{proof}
	(a) All submodules of $M$ are finitely generated, and $N=\calF^{p}_{PBW}U(\fg)\cdot S\subset\calF^{p'}_{PBW}U(\fg)\cdot S\subset N$.
	
	(b) A submodule of a submodule of $M$ is a submodule of $M$. A submodule of a quotient $M/M''$ has the form $N/M''$ with $M''\subset N\subset M$; if $N=\calF^{p}_{PBW}U(\fg)\cdot S$, then $N/M''=\calF^{p}_{PBW}U(\fg)\cdot\bar S$, where $\bar S$ is the image of $S$. The last assertion is obvious.
	
	(c) Let $N\subset M$ be a submodule. By assumption, $N\cap M'=\calF^{p}_{PBW}U(\fg)\cdot S_1$ for a finite subset $S_1\subset N\cap M'$, and the image $N/(N\cap M')$ of $N$ in $M/M'$ is equal to $\calF^{p}_{PBW}U(\fg)\cdot\bar S_2$ for a finite subset $\bar S_2$, which we lift to a finite subset $S_2\subset N$. For $y\in N$ write the image of $y$ as $\sum_t\xi_t\bar s_t$ with $\xi_t\in\calF^{p}_{PBW}U(\fg)$ and $\bar s_t\in\bar S_2$. Then $y-\sum_t\xi_ts_t\in N\cap M'$, and hence $N=\calF^{p}_{PBW}U(\fg)\cdot(S_1\cup S_2)$. The second assertion follows by induction on the length of the filtration.

	(d) Every $\fg$-submodule $N\subset M$ is an $\fh$-submodule; hence $N=\calF^{p}_{PBW}U(\fh)\cdot S\subset\calF^{p}_{PBW}U(\fg)\cdot S\subset N$ for a finite subset $S\subset N$.
\end{proof}

In this section we focus on the Lie algebra $\sL_1=\{e_u\mid u\in z^2\Bbbk[z]\}$ of Subsection~\ref{sec::notation::W1} and study tensor products of finitely many $\sL_1$-modules $\cT_{\lambda_i,\mu_i}$,
which we denote by $\cT^{r}_{\bar{\lambda},\bar{\mu}}$ as in~\eqref{eq::T_blmu}.
For subsequent applications we are primarily interested in the case $\mu_i=0$. However, the proof below works for arbitrary values $\lambda_i,\mu_i\in\Bbbk$.
We use the notation of Subsections~\ref{sec::notation::W1} and~\ref{sec::twisted}: a module $\cT^r_{\bar\lambda,\bar\mu}$ of the class $\cV_r$ (Definition~\ref{def::class::V}) is identified with the polynomial ring $A=\Bbbk[z_1,\ldots,z_r]$, on which $e_u\in\sL_1$ acts by the operator $L_u=X_u+m_{\omega(u)}$ of~\eqref{eq::L_u::q} and~\eqref{eq::L_u::Leibniz}.
The main result of this section is the following theorem.

\begin{theorem}
	\label{thm::Noether::L_1}
	Let $r\geq0$ and let $\bar\lambda,\bar\mu\in\Bbbk^r$ be arbitrary. Then the $\sL_1$-module $\cT^r_{\bar\lambda,\bar\mu}$ has property $\calP_{2r}$ (Definition~\ref{def::property::P}). In particular, $\cT^r_{\bar\lambda,\bar\mu}$ is a Noetherian $\sL_1$-module.
\end{theorem}

The proof of Theorem~\ref{thm::Noether::L_1} is an induction on the number $r$ of tensor factors, that is, on the number of variables of $\cT^r_{\bar\lambda,\bar\mu}\simeq\Bbbk[z_1,\ldots,z_r]$. It consists of three steps and will be summarized in \S\ref{sec::L_1::proof}.
\begin{enumerate}[label=(\arabic*), itemsep=0pt, topsep=0pt]
	\item \emph{Multiplication operators of bounded PBW degree.} Every element of $U(\sL_1)$ acts on $\Bbbk[z_1,\ldots,z_r]$ by a differential operator. Using special elements $\Psi(F)\in U(\sL_1)$ attached to polynomials $F$ in auxiliary variables (see also Subsection~\ref{sec::Witt::U}), we show that the image of $\calF^{2r}_{PBW}U(\sL_1)$ contains the operators of multiplication by $hf$, where $f$ is an arbitrary symmetric polynomial and $h$ is a fixed product of linear forms vanishing on the hyperplanes $\{z_i=0\}$ and $\{z_i=z_j\}$ (Lemma~\ref{lem::key::good}). This is true if all weights $\lambda_i$ are different from $0$ and $1$. The weight $1$ is reduced to the weight $0$ by the de Rham differential, and for the variables of weight $0$, where no multiplication operators exist (Remark~\ref{rmk::no::mult}), the multiplication operators come composed with the de Rham differentials $\Phi_b$ in these variables.
	\item \emph{Comparison of $R$-spans.} Let $R$ be the ring of symmetric polynomials, and let $M$ be a submodule. If there are no variables of weight $0$, Step~(1) shows that $h\cdot RM$ is obtained from finitely many elements of $M$ by the action of $\calF^{2r}_{PBW}U(\sL_1)$. In general, we show that the $R$-modules $RM$ and $R\,\Phi_B(M)$ differ only by a module supported on the hyperplanes $\{z_i=0\}$ and $\{z_i=z_j\}$ (Lemmas~\ref{lem::support} and~\ref{cor::generic::rank}).
	\item \emph{The Artin--Rees lemma and induction.} By the Artin--Rees lemma, the elements of $M$ divisible by a sufficiently high power of $h$ are obtained from finitely many elements of $M$ by the action of $\calF^{2r}_{PBW}U(\sL_1)$. The quotient of $M$ by this ``deep'' part embeds into a module glued from finitely many modules supported on the hyperplanes $\{z_i=0\}$ and $\{z_i=z_j\}$. The latter are submodules of tensor products of $r-1$ modules of densities, and we can use the induction hypothesis.
\end{enumerate}
The induction requires arbitrary weights $\bar\lambda,\bar\mu$, even if one starts with $\bar\mu=0$: the restriction to a hyperplane $\{z_i=z_j\}$ merges two variables and adds their weights (compare with~\eqref{eq::T_lmn}), and the restriction to $\{z_i=0\}$ shifts the parameters $\mu_k$.
Step~(1) is carried out in Subsection~\ref{sec::L_1::operators} for the weights different from $0$ and $1$, and in Subsection~\ref{sec::L_1::deRham} for the weights $0$ and $1$. Step~(2) is Subsection~\ref{sec::L_1::spans}, and Step~(3) is Subsection~\ref{sec::L_1::proof}. In Subsection~\ref{sec::L_1::Hilb} we deduce the rationality of the Hilbert series.

\subsection{Multiplication operators of bounded PBW degree for $\lambda_i\neq0,1$}
\label{sec::L_1::operators}

In this subsection we show that, if all weights $\lambda_i$ are different from $0$ and $1$, then the image of $\calF^{2r}_{PBW}U(\sL_1)$ in the algebra of differential operators contains the multiplications by all symmetric polynomials divisible by a fixed product of degenerate forms (Lemma~\ref{lem::key::good}). The construction uses special elements of $U(\sL_1)$ attached to polynomials in auxiliary variables which are divisible by high powers of Vandermonde-type products.

For $n\geq1$ let $x_1,\ldots,x_n$ be auxiliary variables, and let $\Bbbk[x_1,\ldots,x_n]_{\geq2}$ be the ideal of $\Bbbk[x_1,\ldots,x_n]$ generated by $x_1^2\cdots x_n^2$; this is the linear span of the products $u_1(x_1)\cdots u_n(x_n)$ with $u_a\in z^2\Bbbk[z]$. Consider the linear map
$$
\Psi:\Bbbk[x_1,\ldots,x_n]_{\geq2}\to U(\sL_1),\qquad \Psi\bigl(u_1(x_1)u_2(x_2)\cdots u_n(x_n)\bigr):=e_{u_1}e_{u_2}\cdots e_{u_n},
$$
that is, $\Psi(x_1^{k_1+1}\cdots x_n^{k_n+1})=e_{k_1}\cdots e_{k_n}$ for $k_1,\ldots,k_n\geq1$. Clearly, $\Psi(F)\in\calF^{n}_{PBW}U(\sL_1)$.
For a map $\iota:[n]\to[r]$, where $[m]:=\{1,\ldots,m\}$, and $F=u_1(x_1)\cdots u_n(x_n)$ with arbitrary $u_a\in\Bbbk[z]$ define the differential operator
$$
\mathrm{Op}_\iota(F):=\prod_{i=1}^r\Bigl(\ \overrightarrow{\prod_{a\in\iota^{-1}(i)}}\ell^{(i)}_{u_a}\Bigr),
$$
where the inner products are ordered by increasing $a$ (the order of the outer product is irrelevant, since the factors commute). The expression is multilinear in $(u_1,\ldots,u_n)$, so $\mathrm{Op}_\iota$ extends by linearity to all polynomials $F\in\Bbbk[x_1,\ldots,x_n]$. The set $\iota^{-1}(i)$ is called the \emph{fibre} of $\iota$ over $i$.

\begin{lemma}
	\label{lem::Psi}
	Let $F\in\Bbbk[x_1,\ldots,x_n]$ and $\iota:[n]\to[r]$.
	\begin{enumerate}[label=(\alph*), itemsep=0pt, topsep=0pt]
		\item If $F\in\Bbbk[x_1,\ldots,x_n]_{\geq2}$, then $\Psi(F)$ acts on $\cT^r_{\bar\lambda,\bar\mu}$ by the operator $\sum_{\iota:[n]\to[r]}\mathrm{Op}_\iota(F)$.
		\item There is a finite collection of data $(c,g,d,j)$, depending on $\iota$, $\bar\lambda$, $\bar\mu$ but not on $F$, where $c\in\Bbbk$, $g$ is a monomial in the variables $z_i^{-1}$, $d\in\bZ_{\geq0}^n$ and $j\in\bZ_{\geq0}^r$, such that
		$$\mathrm{Op}_\iota(F)=\sum_{(c,g,d,j)}c\;m_{g\cdot(\partial_x^{d}F)|_{x_a=z_{\iota(a)}}}\circ\prod_{i=1}^r\partial_i^{j_i},$$
		and $\sum_{a\in\iota^{-1}(i)}d_a+j_i\leq|\iota^{-1}(i)|$ for every $i$ in each term. Here $\partial_x^d=\prod_a\partial_{x_a}^{d_a}$, and $|_{x_a=z_{\iota(a)}}$ denotes the substitution of $z_{\iota(a)}$ for $x_a$, $a=1,\ldots,n$.
		\item If $F$ is divisible by $(x_a-x_b)^{n+1}$ and $\iota(a)=\iota(b)$, then $\mathrm{Op}_\iota(F)=0$.
	\end{enumerate}
\end{lemma}
\begin{proof}
	(a) For $F=\prod_au_a(x_a)$ the element $\Psi(F)=e_{u_1}\cdots e_{u_n}$ acts by
	$$L_{u_1}\cdots L_{u_n}=\sum_{\iota:[n]\to[r]}\ell^{(\iota(1))}_{u_1}\ell^{(\iota(2))}_{u_2}\cdots\ell^{(\iota(n))}_{u_n}.$$
	Since operators in different variables commute, each summand can be regrouped into $\mathrm{Op}_\iota(F)$ without changing the relative order of the factors in the same variable. The general case follows by linearity.

	(b) It suffices to consider one fibre, that is, an ordered product $\ell_{u_1}\cdots\ell_{u_m}$ of operators $\ell_u=u\partial+\lambda u'+\mu u/z$ in one variable $z$. Expanding the product, each factor $\ell_{u_a}$ contributes either $u_a\partial$, or $\lambda u_a'$, or $\mu u_az^{-1}$. Moving all derivatives $\partial$ to the right by the Leibniz rule, each $\partial$ coming from the $a$-th factor either differentiates one of the functions $u_b$ with $b>a$, or differentiates a power of $z^{-1}$, or survives. Hence $\ell_{u_1}\cdots\ell_{u_m}$ is a sum of terms $c\,g\,u_1^{(d_1)}\cdots u_m^{(d_m)}\partial^{j}$, with universal constants $c$, monomials $g$ in $z^{-1}$, and $d_1+\ldots+d_m+j\leq m$. Since $\prod_au_a^{(d_a)}(z_{\iota(a)})=(\partial_x^dF)|_{x_a=z_{\iota(a)}}$ for $F=\prod_au_a(x_a)$, assertion (b) follows by multiplying over the fibres and by linearity in $F$.

	(c) In each term of (b) we have $|d|\leq\sum_i|\iota^{-1}(i)|=n$. Hence $\partial^d_xF$ is divisible by a positive power of $x_a-x_b$, and vanishes after the substitution $x_a=x_b=z_{\iota(a)}$.
\end{proof}

For the next lemma we need the explicit form of the product of two operators $\ell_u=u\partial+\lambda u'+\mu u/z$ in one variable:
\begin{multline}
	\label{eq::two::factors}
	\ell_\alpha\ell_\beta=\alpha\beta\,\partial^2+\Bigl((1+\lambda)\alpha\beta'+\lambda\alpha'\beta+\frac{2\mu}{z}\alpha\beta\Bigr)\partial
	\\
	+\lambda\,\alpha\beta''+\lambda^2\alpha'\beta'+\frac{\mu(1+\lambda)}{z}\alpha\beta'+\frac{\lambda\mu}{z}\alpha'\beta+\frac{\mu(\mu-1)}{z^2}\alpha\beta .
\end{multline}
Note that only the two terms $\lambda\,\alpha\beta''$ and $\lambda^2\alpha'\beta'$ carry two derivatives on the functions $\alpha,\beta$.

\begin{lemma}
	\label{lem::onevar}
	Let $F\in\Bbbk[x_1,\ldots,x_n]$ and $\iota:[n]\to[r]$. Suppose that the fibre of $\iota$ over $i\in[r]$ is $\{a,b\}$ with $a<b$, and that $F=(x_a-x_b)^2E$ for a polynomial $E$. Let $\iota'$ be the restriction of $\iota$ to $[n]\setminus\{a,b\}$. Then
	$$\mathrm{Op}_\iota(F)=2\lambda_i(1-\lambda_i)\cdot\mathrm{Op}_{\iota'}\bigl(E|_{x_a=x_b=z_i}\bigr),$$
	where the right-hand side is defined by $\Bbbk[z_i]$-linearity (the operator $\mathrm{Op}_{\iota'}$ does not involve the variable $z_i$). In particular, $\mathrm{Op}_\iota(F)=0$ if $\lambda_i\in\{0,1\}$.
\end{lemma}
\begin{proof}
	Let first $F=\alpha(x_a)\beta(x_b)F''$, where $F''$ is a product of functions of the other variables. Then $\mathrm{Op}_\iota(F)=\ell^{(i)}_\alpha\ell^{(i)}_\beta\circ\mathrm{Op}_{\iota'}(F'')$. Expand $\ell^{(i)}_\alpha\ell^{(i)}_\beta$ by~\eqref{eq::two::factors}; every term has the form $c\,g\,\alpha^{(d_a)}\beta^{(d_b)}\partial_i^j$. By linearity, for an arbitrary $F$ we get
	$$\mathrm{Op}_\iota(F)=\sum c\;m_g\circ\mathrm{Op}_{\iota'}\Bigl(\bigl(\partial_{x_a}^{d_a}\partial_{x_b}^{d_b}F\bigr)\big|_{x_a=x_b=z_i}\Bigr)\circ\partial_i^{j},$$
	the sum running over the terms of~\eqref{eq::two::factors}. If $F=(x_a-x_b)^2E$, the terms with $d_a+d_b\leq1$ vanish, since $\partial_{x_a}^{d_a}\partial_{x_b}^{d_b}F$ is divisible by $x_a-x_b$. The two remaining terms, $\lambda_i\alpha\beta''$ and $\lambda_i^2\alpha'\beta'$, give
	$$\lambda_i\bigl(\partial_{x_b}^2F\bigr)\big|_{x_a=x_b=z_i}+\lambda_i^2\bigl(\partial_{x_a}\partial_{x_b}F\bigr)\big|_{x_a=x_b=z_i}=\bigl(2\lambda_i-2\lambda_i^2\bigr)E|_{x_a=x_b=z_i}$$
	inside $\mathrm{Op}_{\iota'}$, because all derivatives must fall on the factor $(x_a-x_b)^2$. This proves the lemma.
\end{proof}

\begin{lemma}
	\label{lem::key::good}
	Let $\cT^r_{\bar\lambda,\bar\mu}\in\cV_r$, and suppose that $\lambda_i\notin\{0,1\}$ for all $i$. Let the symmetric group $\Gamma:=S_r$ act on $A$ by the automorphisms $\gamma(z_i)=z_{\gamma(i)}$, and let $R:=A^\Gamma$ be the ring of symmetric polynomials. Fix an even integer $q>2r$ and put
	$$h:=\prod_{i=1}^rz_i^4\cdot\prod_{i<i'}(z_i-z_{i'})^{4q}.$$
	Then for every $f\in R$ the operator $m_{hf}$ of multiplication by $hf$ on $\cT^r_{\bar\lambda,\bar\mu}$ is the action of an element of $\calF^{2r}_{PBW}U(\sL_1)$.
\end{lemma}
\begin{proof}
	Consider $2r$ auxiliary variables $x_i,y_i$, $i=1,\ldots,r$, ordered in such a way that $x_i$ precedes $y_i$; this order is used in the definition of $\Psi$ and $\mathrm{Op}_\iota$. The pairs $\{x_i,y_i\}$ are called \emph{clusters}, and the variables $x_i$ their \emph{representatives}. For a polynomial $H$ in the representatives put
	$$F_H:=\prod_{i=1}^r(x_i-y_i)^2\cdot\prod_{v}v^2\cdot\prod_{\{v,v'\}}(v-v')^q\cdot H,$$
	where $v$ runs over all $2r$ auxiliary variables and $\{v,v'\}$ runs over all unordered pairs of auxiliary variables from different clusters. The last factor is a power of a Vandermonde-type product, and it separates the clusters, see below. We have $F_H\in\Bbbk[x,y]_{\geq2}$, so that $\Psi(F_H)\in\calF^{2r}_{PBW}U(\sL_1)$, and by Lemma~\ref{lem::Psi}(a), $\Psi(F_H)$ acts by $\sum_\iota\mathrm{Op}_\iota(F_H)$, the sum running over all maps $\iota$ from the set of auxiliary variables to $[r]$.

 First, let us determine the maps $\iota$ that contribute. Since $q\geq 2r+1$, Lemma~\ref{lem::Psi}(c) shows that $\mathrm{Op}_\iota(F_H)=0$ unless $\iota$ separates any two auxiliary variables from different clusters. For such $\iota$, the $r$ representatives are mapped to pairwise distinct elements, that is, bijectively onto $[r]$. The variable $y_i$ lies in a cluster different from those of all representatives except $x_i$; hence $\iota(y_i)=\iota(x_i)$, and the fibres of $\iota$ are exactly the clusters. Thus $\iota=\iota_\gamma:=\gamma\circ\iota_0$ for some $\gamma\in S_r$, where $\iota_0(x_i)=\iota_0(y_i)=i$.

	Second, let us compute the contribution of $\iota_\gamma$. Applying Lemma~\ref{lem::onevar} successively to the $r$ pairs, we get
	$$\mathrm{Op}_{\iota_\gamma}(F_H)=\prod_{i=1}^r2\lambda_{\gamma(i)}(1-\lambda_{\gamma(i)})\cdot m_{C_\gamma}=\kappa\cdot m_{C_\gamma},$$
	where $\kappa:=\prod_{i=1}^r2\lambda_i(1-\lambda_i)\neq0$, since $\lambda_i\notin\{0,1\}$, and
	$$C_\gamma:=\Bigl(\frac{F_H}{\prod_{i=1}^r(x_i-y_i)^2}\Bigr)\Big|_{x_i=y_i=z_{\gamma(i)}}=\gamma(C_1).$$ For $\gamma=1$ the four pairs of variables between the clusters of $i\neq i'$ contribute $(z_i-z_{i'})^{4q}$ (the signs disappear, since $q$ is even), and the factor $\prod_vv^2$ gives $\prod_iz_i^4$. Therefore $C_1=h\cdot H(z)$, where $H(z)$ is obtained from $H$ by the substitution $x_i\mapsto z_i$. Since $q$ is even, $h$ is symmetric, so that $C_\gamma=h\cdot\gamma(H(z))$, and $\Psi(F_H)$ acts on $\cT^r_{\bar\lambda,\bar\mu}$ by $\kappa\cdot m_{h\sum_{\gamma\in S_r}\gamma(H(z))}$. Now let $f\in R$, and let $H$ be obtained from $f/r!$ by the substitution $z_i\mapsto x_i$. Then $\sum_\gamma\gamma(H(z))=f$, and the element $\kappa^{-1}\Psi(F_H)\in\calF^{2r}_{PBW}U(\sL_1)$ acts on $\cT^r_{\bar\lambda,\bar\mu}$ by $m_{hf}$.
\end{proof}

\subsection{De Rham differentials for $\lambda_i=0,1$}
\label{sec::L_1::deRham}

By Lemma~\ref{lem::onevar}, the construction of Subsection~\ref{sec::L_1::operators} gives zero as soon as a pair of auxiliary variables is mapped to a variable of weight $0$ or $1$; in fact, for these weights there are no nonconstant multiplication operators at all (Remark~\ref{rmk::no::mult}). In this subsection we introduce the operators $\Phi_B$ built from de Rham differentials, which replace the multiplication operators in these variables; in Subsection~\ref{sec::L_1::spans} they give the reduction from $RM$ to $R\,\Phi_B(M)$.

We start with the following simple observation about the weights $0$ and $1$.

\begin{lemma}
	\label{lem::Euler}
	Let $\mu\in\Bbbk$ and let $\Phi:=z\partial+\mu$, acting on $\Bbbk[z]$.
	\begin{enumerate}[label=(\alph*), itemsep=0pt, topsep=0pt]
		\item For every $u\in\Bbbk[z]$ the operator $\ell_u=u\partial+\mu\frac{u}{z}$ of $e_u$ on $\cT_{0,\mu}$ factors as $\ell_u=m_{u/z}\circ\Phi$. Moreover, $[\Phi,m_g]=m_{zg'}$.
		\item The map $\Phi:\cT_{0,\mu}\to\cT_{1,\mu-1}$ is a homomorphism of $\sL_1$-modules, and $\Phi(z^m)=(m+\mu)z^m$. In particular, $\Phi$ is an isomorphism if $\mu\notin\bZ_{\leq0}$.
		\item The operators of $e_u$ on $\cT_{1,\mu-1}$ and on $\cT_{0,\mu}$ differ by a multiplication operator: $\ell'_u=\ell_u+m_{u'-u/z}$.
	\end{enumerate}
\end{lemma}
\begin{proof}
	Parts (a) and (c) are immediate from~\eqref{eq::L_1::action}. By (a), $\Phi\circ m_{u/z}=m_{u/z}\circ\Phi+m_{z(u/z)'}=u\partial+\mu\frac uz+u'-\frac uz=\ell'_u$, the operator of $e_u$ on $\cT_{1,\mu-1}$. Hence $\Phi\circ\ell_u=\Phi\circ m_{u/z}\circ\Phi=\ell'_u\circ\Phi$, which proves (b); the formula for $\Phi(z^m)$ is clear.
\end{proof}
In terms of the realization~\eqref{eq::T_lmu} of the modules $\cT_{\lambda,\mu}$, the operator $\Phi$ is the de Rham differential $g\,z^{\mu}\mapsto d(g\,z^{\mu})=\bigl((z\partial+\mu)g\bigr)z^{\mu-1}dz$ from functions to $1$-forms; this explains why it commutes with the action of vector fields.

\begin{remark}
	\label{rmk::no::mult}
	The factor $2\lambda_i(1-\lambda_i)$ in Lemma~\ref{lem::onevar} vanishes for $\lambda_i\in\{0,1\}$. This is not an artefact of the method: for $\lambda\in\{0,1\}$ the image of $U(\sL_1)$ in $\mathrm{End}(\cT_{\lambda,\mu})$ contains no multiplication operators except the constants. Indeed, put $\theta:=z\partial$. By~\eqref{eq::L_1::action}, $e_k$ acts on $\cT_{0,\mu}$ by $z^k(\theta+\mu)$, and on $\cT_{1,\mu}$ by $(\theta+\mu+1)z^k=z^k(\theta+\mu+k+1)$. Hence every element of the image has the form $c+\sum_{n\geq1}z^n(\theta+\mu)q_n(\theta)$ for $\lambda=0$, and $c+\sum_{n\geq1}z^n(\theta+\mu+n+1)q_n(\theta)$ for $\lambda=1$, with $c\in\Bbbk$ and polynomials $q_n$; such an operator is a multiplication operator only if all $q_n$ vanish. This is the reason why the operators $\Phi_b$ appear in Lemma~\ref{lem::key::mult} below.
\end{remark}

\begin{definition}
	\label{def::normalized}
	A module $\cT^r_{\bar\lambda,\bar\mu}\in\cV_r$ is called \emph{normalized} if every variable $z_i$ is either \emph{good}, that is $\lambda_i\notin\{0,1\}$, or \emph{bad}, that is $\lambda_i=0$ and $\mu_i\notin\bZ_{\leq0}$.
\end{definition}

For a normalized module $\cT^r_{\bar\lambda,\bar\mu}\in\cV_r$ we use the following notation:
\begin{itemize}[itemsep=0pt, topsep=0pt]
	\item $G\subset[r]$ and $B\subset[r]$ are the sets of (indices of) good and bad variables, so that $[r]=G\sqcup B$;
	\item $\Gamma:=S_B\times S_G$ is the group of permutations of $\{1,\ldots,r\}$ preserving the sets $G$ and $B$. It acts on $A$ (and on its field of fractions) by the automorphisms $\gamma(z_i)=z_{\gamma(i)}$. We put $R:=A^\Gamma$; if all variables are good, these are the group $\Gamma$ and the ring $R$ of Lemma~\ref{lem::key::good};
	\item $\Phi_b:=z_b\partial_b+\mu_b$ for $b\in B$, and $\Phi_B:=\prod_{b\in B}\Phi_b$ (these operators commute).
\end{itemize}
Let us collect several simple facts. The ring $R$ is the ring of polynomials in the elementary symmetric polynomials of the sets of variables $\{z_i\mid i\in G\}$ and $\{z_b\mid b\in B\}$, and $A$ is a finitely generated (in fact, free) $R$-module. In particular, $R$ is Noetherian and every $R$-submodule of $A$ is finitely generated. The derivations $X_u$ are symmetric in all variables and therefore commute with $\Gamma$; in particular, $X_u(R)\subset R$. Finally, $\Phi_b(z_1^{a_1}\cdots z_r^{a_r})=(a_b+\mu_b)z_1^{a_1}\cdots z_r^{a_r}$; hence, since $\mu_b\notin\bZ_{\leq0}$, the operators $\Phi_b$ and $\Phi_B$ are bijective on $A$.

\begin{lemma}
	\label{lem::normalization}
	Let $p\geq0$ and assume that $\cV_{r-1}$ has property $\calP_p$. If every normalized module of $\cV_r$ has property $\calP_p$, then $\cV_r$ has property $\calP_p$.
\end{lemma}
\begin{proof}
	Let $\cT=\cT^r_{\bar\lambda,\bar\mu}\in\cV_r$. We modify $\cT$ by the following two moves.

	First, let us explain how to raise $\mu_i$. For $k\geq0$ the submodule $z_i^k\cT\subset \cT$ is isomorphic, by~\eqref{eq::strata::twist}, to the module obtained from $\cT$ by replacing $\mu_i$ with $\mu_i+k$. The quotient $\cT/z_i^k\cT$ has the finite filtration $\cT\supset z_i\cT\supset\ldots\supset z_i^k\cT$, whose subquotients are isomorphic to submodules of modules of $\cV_{r-1}$ by Lemma~\ref{lem::strata}. By Lemma~\ref{lem::closure} and the assumption on $\cV_{r-1}$, if the modified module has property $\calP_p$, then so does $\cT$.

	Second, let us explain how to replace the weight $1$ by the weight $0$. If $\lambda_i=1$ and $\mu_i+1\notin\bZ_{\leq0}$, then $\Phi:\cT_{0,\mu_i+1}\to\cT_{1,\mu_i}$ is an isomorphism by Lemma~\ref{lem::Euler}(b). Applying it to the $i$-th tensor factor, we see that $\cT$ is isomorphic to the module obtained from $\cT$ by replacing $(\lambda_i,\mu_i)$ with $(0,\mu_i+1)$.

	Now, for every $i$ with $\lambda_i=1$ we first raise $\mu_i$ in such a way that $\mu_i\notin\bZ_{<0}$, and then replace the weight $1$ by $0$; the new variable has $\lambda_i=0$ and $\mu_i+1\notin\bZ_{\leq0}$. Next, for every $i$ with $\lambda_i=0$ and $\mu_i\in\bZ_{\leq0}$ we raise $\mu_i$ to $1$. The resulting module is normalized, and $\cT$ has property $\calP_p$ if the resulting module does.
\end{proof}

\begin{lemma}
	\label{lem::singletons}
	Let $F\in\Bbbk[x_1,\ldots,x_n]$, and let $\iota:[n]\to[r]$ be an injective map such that $\lambda_{\iota(a)}=0$ for all $a$. Then
	$$\mathrm{Op}_\iota(F)=m_C\circ\prod_{a=1}^{n}\Phi_{\iota(a)},\qquad C:=\frac{F|_{x_a=z_{\iota(a)}}}{z_{\iota(1)}\cdots z_{\iota(n)}},\qquad \Phi_i:=z_i\partial_i+\mu_i .$$
\end{lemma}
\begin{proof}
	For $F=\prod_au_a(x_a)$, Lemma~\ref{lem::Euler}(a) gives $\mathrm{Op}_\iota(F)=\prod_a\ell^{(\iota(a))}_{u_a}=\prod_a\bigl(m_{u_a(z_{\iota(a)})/z_{\iota(a)}}\circ\Phi_{\iota(a)}\bigr)$, and the factors commute, as they act on different variables. The general case follows by linearity.
\end{proof}

The following lemma generalises Lemma~\ref{lem::key::good} for the case of bad variables.

\begin{lemma}
	\label{lem::key::mult}
	Let $\cT^r_{\bar\lambda,\bar\mu}\in\cV_r$ be normalized. There exists a nonzero $\Gamma$-invariant polynomial $h\in R$, which is a product of degenerate forms of $\cT^r_{\bar\lambda,\bar\mu}$, with the following property: for every $f\in R$ the operator $m_{hf}\circ\Phi_B$ on $\cT^r_{\bar\lambda,\bar\mu}$ is the action of an element of $\calF^{p_0}_{PBW}U(\sL_1)$, where
	$$p_0:=2|G|+|B|\leq 2r.$$
\end{lemma}
\begin{proof}
	We modify the proof of Lemma~\ref{lem::key::good}. In addition to the clusters $\{x_i,y_i\}$, $i\in G$, we use the \emph{singleton} clusters $\{s_b\}$, $b\in B$, whose representatives are the variables $s_b$; thus there are $p_0=2|G|+|B|$ auxiliary variables, and we fix an even integer $q>p_0$. The polynomial $F_H$ is defined by the same formula as in the proof of Lemma~\ref{lem::key::good}, where now $v$ runs over all $p_0$ auxiliary variables and $H$ is a polynomial in the representatives $x_i$, $s_b$; then $\Psi(F_H)\in\calF^{p_0}_{PBW}U(\sL_1)$.

	As before, only the maps $\iota$ separating the clusters contribute: the representatives are mapped bijectively onto $[r]$, and $\iota(y_i)=\iota(x_i)$. If a pair $\{x_i,y_i\}$ is mapped to a bad variable, then $\mathrm{Op}_\iota(F_H)=0$ by Lemma~\ref{lem::onevar}, since $2\lambda(1-\lambda)=0$ for $\lambda=0$. Hence only the maps $\iota_\gamma:=\gamma\circ\iota_0$ with $\gamma\in\Gamma=S_B\times S_G$ contribute, where $\iota_0(x_i)=\iota_0(y_i)=i$ and $\iota_0(s_b)=b$. Applying Lemma~\ref{lem::onevar} to the pairs and Lemma~\ref{lem::singletons} to the singletons (with coefficients in $\Bbbk[z_i\mid i\in G]$), we get
	$$\mathrm{Op}_{\iota_\gamma}(F_H)=\kappa\cdot m_{C_\gamma}\circ\Phi_B,\qquad C_\gamma:=\Bigl(\frac{F_H}{\prod_{i\in G}(x_i-y_i)^2\prod_{b\in B}s_b}\Bigr)\Big|_{x_i=y_i=z_{\gamma(i)},\ s_b=z_{\gamma(b)}},$$
	where $\kappa:=\prod_{i\in G}2\lambda_i(1-\lambda_i)\neq0$ and $C_\gamma=\gamma(C_1)$; here we use that the product $\prod_{b\in B}\Phi_{\gamma(b)}=\Phi_B$ does not depend on $\gamma$. As in the proof of Lemma~\ref{lem::key::good}, $C_1=h\cdot H(z)$, where $H(z)$ is obtained from $H$ by the substitution $x_i\mapsto z_i$, $s_b\mapsto z_b$, and
	$$h:=\prod_{i\in G}z_i^4\prod_{b\in B}z_b\prod_{\substack{i<i'\\ i,i'\in G}}(z_i-z_{i'})^{4q}\prod_{\substack{i\in G\\ b\in B}}(z_i-z_b)^{2q}\prod_{\substack{b<b'\\ b,b'\in B}}(z_b-z_{b'})^{q};$$
	it is invariant under $\Gamma$, since $q$ is even. Let $f\in R$, and let $H$ be obtained from $f/|\Gamma|$ by the substitution $z_i\mapsto x_i$, $z_b\mapsto s_b$. Summing over $\gamma\in\Gamma$, we see that $\kappa^{-1}\Psi(F_H)\in\calF^{p_0}_{PBW}U(\sL_1)$ acts on $\cT^r_{\bar\lambda,\bar\mu}$ by $m_{hf}\circ\Phi_B$.
\end{proof}

In the next subsection we show that the $R$-modules $RM$ and $R\,\Phi_B(M)$ differ only by a module supported on the degenerate hyperplanes. Together with Lemma~\ref{lem::key::mult}, this produces many elements of a submodule $M$ which are obtained from finitely many of them by the action of operators of bounded PBW degree.

\subsection{Comparison of $R$-spans}
\label{sec::L_1::spans}

In this subsection $\cT^r_{\bar\lambda,\bar\mu}\in\cV_r$ is normalized, and we use the notation introduced after Definition~\ref{def::normalized}. For a subset $Y\subset A$ we denote by $RY\subset A$ the $R$-submodule generated by $Y$. The letter $u$ always denotes an element of $z^2\Bbbk[z]$, so that $e_u\in\sL_1$, and $L_u=X_u+m_{\omega(u)}$ is the operator of $e_u$ on $A$, see~\eqref{eq::L_u::Leibniz}. Since $L_u(fy)=X_u(f)y+f\,L_u(y)$ and $X_u(R)\subset R$, the $R$-span $RM$ of a submodule $M\subset \cT^r_{\bar\lambda,\bar\mu}$ is stable under all operators $L_u$.

\begin{lemma}
	\label{lem::support}
	Put $h_0:=\prod_{\gamma\in\Gamma}\gamma(\delta)\in R$, where $\delta:=\prod_{i=1}^rz_i\cdot\prod_{i<j}(z_i-z_j)$ is the product of all degenerate forms. Let $E$ be a finitely generated $R$-module equipped with $\Bbbk$-linear operators $L_u:E\to E$, one for every $e_u\in\sL_1$, satisfying the Leibniz rule
	$$L_u(fx)=X_u(f)\,x+f\,L_u(x)\qquad\text{for } f\in R,\ x\in E.$$
	If $E$ is a torsion $R$-module, then $h_0^mE=0$ for some $m\geq0$.
\end{lemma}
\begin{proof}
	We may assume $E\neq0$. Let $J:=\mathrm{Ann}_R(E)$. If $x_1,\ldots,x_k$ generate $E$ and $0\neq f_t\in R$ kills $x_t$, then $f_1\cdots f_k\in J$; hence $J\neq0$, since $R$ is a domain. For $f\in J$ and $x\in E$ we have $X_u(f)x=L_u(fx)-f\,L_u(x)=0$, so that $X_u(J)\subset J$ for all $u$.

	Let $A^\circ:=A[h_0^{-1}]$ and $R^\circ:=R[h_0^{-1}]$; since $h_0$ is $\Gamma$-invariant, $R^\circ=(A^\circ)^\Gamma$. In $A^\circ$ all the degenerate forms are invertible, because $\delta$ divides $h_0$. The derivations $X_u$ extend to $A^\circ$, and the ideal $JA^\circ\subset A^\circ$ is stable under them by the Leibniz rule.

	First, let us show that the ideal $JA^\circ$ is stable under all partial derivatives $\partial_i$. Consider $u_k:=z^{k+2}$, $k=0,\ldots,r-1$. Then $X_{u_k}=\sum_{i=1}^rz_i^{k+2}\partial_i$, and the determinant of the matrix $\bigl(z_i^{k+2}\bigr)_{k,i}$ equals $\prod_iz_i^2\cdot\prod_{i<j}(z_j-z_i)$, which is invertible in $A^\circ$. Hence every $\partial_i$ is an $A^\circ$-linear combination of the derivations $X_{u_k}$, and therefore it preserves the ideal $JA^\circ$.

	Second, let us show that $JA^\circ=A^\circ$. The ideal $JA^\circ$ contains a nonzero element; multiplying it by a power of $h_0$ we obtain a nonzero polynomial in $JA^\circ\cap\Bbbk[z_1,\ldots,z_r]$. Let $P$ be such a polynomial of minimal total degree. If $P$ were not constant, then, since $\mathrm{char}\,\Bbbk=0$, some derivative $\partial_iP$ would be a nonzero polynomial in $JA^\circ$ of smaller degree. Hence $P$ is a nonzero constant, and $JA^\circ=A^\circ$.

	Finally, let us return to $R$. Write $1=\sum_ja_jf_j$ with $a_j\in A^\circ$ and $f_j\in J$. Applying the Reynolds operator $\rho:=\frac1{|\Gamma|}\sum_{\gamma\in\Gamma}\gamma$ and using $\gamma(f_j)=f_j$, we get $1=\sum_j\rho(a_j)f_j$ with $\rho(a_j)\in R^\circ$. Multiplying by a sufficiently high power of $h_0$, we obtain $h_0^m=\sum_j\bigl(h_0^m\rho(a_j)\bigr)f_j\in J$, since $h_0^m\rho(a_j)\in R$ for $m$ large enough. Hence $h_0^mE=0$.
\end{proof}

Let us recall the data we work with in the rest of this subsection. The tuples $\bar\lambda,\bar\mu$ are normalized (Definition~\ref{def::normalized}). After renumbering the variables we may assume that the bad variables come first, $B=\{1,\ldots,|B|\}$ and $G=\{|B|+1,\ldots,r\}$, so that
$$\lambda_1=\ldots=\lambda_{|B|}=0,\qquad\mu_1,\ldots,\mu_{|B|}\notin\bZ_{\leq0},\qquad\lambda_{|B|+1},\ldots,\lambda_r\notin\{0,1\}.$$
The group $\Gamma=S_B\times S_G$ permutes the bad variables among themselves and the good variables among themselves, $R=A^\Gamma$, and $\Phi_B=\prod_{b\in B}(z_b\partial_b+\mu_b)$. The parameters $\lambda_i,\mu_i$ need not be constant on $B$ and on $G$; therefore, in general, $\Gamma$ does not act on $\cT^r_{\bar\lambda,\bar\mu}$ by automorphisms of $\sL_1$-modules, and only the derivations $X_u$ commute with $\Gamma$.

Let $\cT'\in\cV_r$ be the module obtained from $\cT^r_{\bar\lambda,\bar\mu}$ by replacing every bad tensor factor $\cT_{0,\mu_b}$ by $\cT_{1,\mu_b-1}$. By Lemma~\ref{lem::Euler}(b), applied in every bad variable, $\Phi_B:\cT^r_{\bar\lambda,\bar\mu}\to\cT'$ is an isomorphism of $\sL_1$-modules. By Lemma~\ref{lem::Euler}(c), the operator $L'_u$ of $e_u\in\sL_1$ on $\cT'$ (identified with $A$) differs from $L_u$ by a multiplication operator:
\begin{equation}
	\label{eq::e::prime}
	L'_u=L_u+m_{c_u},\qquad c_u:=\sum_{b\in B}\Bigl(u'(z_b)-\frac{u(z_b)}{z_b}\Bigr)\in R .
\end{equation}
Let $\bF:=\Bbbk(z_1,\ldots,z_r)$ be the field of fractions of $A$, and consider the subfield of $\Gamma$-invariant rational functions:
$$
\mathbb{K}:=[\bF]^{\Gamma} = \mathsf{Frac}(R).
$$

\begin{lemma}
	\label{cor::generic::rank}
	For every submodule $M\subset \cT^r_{\bar\lambda,\bar\mu}$ there exists $m\geq0$ such that $h_0^m\,RM\subset R\,\Phi_B(M)$.
\end{lemma}
\begin{proof}
	Put $Q:=R\,\Phi_B(M)$. The $R$-module $RM$ is stable under the operators $L_u$. The $R$-module $Q$ is stable under the operators $L'_u$, since $\Phi_B(M)\subset\cT'$ is a submodule, and hence under $L_u=L'_u-m_{c_u}$, since $c_u\in R$. Therefore $E:=(RM+Q)/Q$ is an $R$-module with operators $L_u$ satisfying the Leibniz rule of Lemma~\ref{lem::support}; it is finitely generated, because $A$ is a Noetherian $R$-module. By Lemma~\ref{lem::support}, it suffices to show that $E$ is a torsion module. This holds if $M\subset\mathbb{K}Q$ inside $\bF$: then $RM\subset\mathbb{K}Q$, and for every $x\in RM$ there is $0\neq f\in R$ with $fx\in Q$.

	First, let us show that the subspace $\mathbb{K}Q\subset\bF$ is stable under $\Phi_B$. Let $h$ be as in Lemma~\ref{lem::key::mult}, and let $\xi\in U(\sL_1)$ be an element acting on $\cT^r_{\bar\lambda,\bar\mu}$ by $m_h\circ\Phi_B$ (Lemma~\ref{lem::key::mult} with $f=1$). Since $\Phi_B:\cT^r_{\bar\lambda,\bar\mu}\to\cT'$ is an isomorphism of $\sL_1$-modules, $\xi$ acts on $\cT'$ by the operator
	$$\Phi_B\circ(m_h\circ\Phi_B)\circ\Phi_B^{-1}=\Phi_B\circ m_h .$$
	The operator of $\xi$ on $\cT'$ is a polynomial in the operators $L'_u$, which act on $\bF$ by the same formulas. This operator and $\Phi_B\circ m_h$ are differential operators with polynomial coefficients; since such an operator is determined by its values on $A$, they coincide on $\bF$ as well. The derivations $X_u$ commute with $\Gamma$, so that $X_u(\mathbb{K})\subset\mathbb{K}$. By the Leibniz rule $L'_u(gx)=X_u(g)\,x+g\,L'_u(x)$, the subspace $\mathbb{K}Q=\mathbb{K}\,\Phi_B(M)$ is stable under all operators $L'_u$, and hence under $\Phi_B\circ m_h$. Since $h^{-1}\in\mathbb{K}$, for every $x\in\mathbb{K}Q$ we get $\Phi_B(x)=(\Phi_B\circ m_h)(h^{-1}x)\in\mathbb{K}Q$.

	Second, let us show that $M\subset\mathbb{K}Q$. The operator $\Phi_B$ is diagonal in the monomial basis of $A$: $\Phi_B(z_1^{a_1}\cdots z_r^{a_r})=\chi(\bar a)\,z_1^{a_1}\cdots z_r^{a_r}$ with $\chi(\bar a):=\prod_{b\in B}(a_b+\mu_b)$, and $\chi(\bar a)\neq0$, since $\mu_b\notin\bZ_{\leq0}$. Let $y\in M$, and let $c_1,\ldots,c_n$ be the pairwise distinct numbers $\chi(\bar a)$ for the monomials occurring in $y$. The polynomial $P(t):=1-\prod_{i=1}^n(1-t/c_i)$ satisfies $P(c_i)=1$ for all $i$ and $P(0)=0$. Hence $P(\Phi_B)\,y=y$, and $P(t)=t\,P_1(t)$ for a polynomial $P_1$. Therefore
	$$y=P_1(\Phi_B)\bigl(\Phi_B(y)\bigr)\in\mathbb{K}Q,$$
	because $\Phi_B(y)\in\Phi_B(M)\subset\mathbb{K}Q$ and $\mathbb{K}Q$ is stable under $\Phi_B$.
\end{proof}

\subsection{The Artin--Rees lemma and the proof of Theorem~\ref{thm::Noether::L_1}}
\label{sec::L_1::proof}

We use the Artin--Rees lemma only for principal ideals; in this case it has the following short proof (for the general case see, e.g.,~\cite[Corollary~10.10]{AM}).

\begin{lemma}[Artin--Rees]
	\label{lem::Artin::Rees}
	Let $R$ be a Noetherian ring, $A$ a finitely generated $R$-module, $N\subset A$ a submodule, and $h\in R$. Then there exists $c\geq0$ such that $N\cap h^{c+1}A\subset hN$.
\end{lemma}
\begin{proof}
	For $k\geq0$ let $N_k:=\{a\in A\mid h^ka\in N\}$. Since $N$ is an $R$-submodule, $N_k\subset N_{k+1}$, and since $A$ is a Noetherian $R$-module, $N_c=N_{c+1}$ for some $c$. If $x=h^{c+1}a\in N$ with $a\in A$, then $a\in N_{c+1}=N_c$, that is, $h^ca\in N$, and therefore $x=h\cdot h^ca\in hN$.
\end{proof}

The following proposition is the main step of the proof of Theorem~\ref{thm::Noether::L_1}: far from the degenerate hyperplanes, every submodule coincides with an $R$-module, which is generated in bounded PBW degree.

\begin{proposition}
	\label{prp::deep}
	Let $\cT^r_{\bar\lambda,\bar\mu}\in\cV_r$ be normalized, let $M\subset \cT^r_{\bar\lambda,\bar\mu}$ be a submodule, and let $x_1,\ldots,x_s\in M$ be such that $\Phi_B(x_1),\ldots,\Phi_B(x_s)$ generate the $R$-module $R\,\Phi_B(M)$. There exist a nonzero $\Gamma$-invariant polynomial $h_1\in R$, which is a product of degenerate forms, such that
	$$h_1\cdot RM\subset\calF^{p_0}_{PBW}U(\sL_1)\cdot\{x_1,\ldots,x_s\}\subset M\subset RM,$$
	 and an integer $c\geq0$ such that
	$$M\cap h_1^{c+1}\cT^r_{\bar\lambda,\bar\mu}=RM\cap h_1^{c+1}A\subset\calF^{p_0}_{PBW}U(\sL_1)\cdot\{x_1,\ldots,x_s\}.$$
	In particular, the part $M\cap h_1^{c+1}\cT^r_{\bar\lambda,\bar\mu}$ of $M$ is an $R$-submodule of $A$. If $M$ is graded, then $h_1$ is homogeneous, and $RM$ and $M\cap h_1^{c+1}\cT^r_{\bar\lambda,\bar\mu}$ are graded.
\end{proposition}
\begin{proof}
	The $R$-module $Q:=R\,\Phi_B(M)\subset A$ is finitely generated, so the elements $x_t$ exist. Let $h$ be as in Lemma~\ref{lem::key::mult}. For $f_1,\ldots,f_s\in R$, Lemma~\ref{lem::key::mult} provides elements $\xi_t\in\calF^{p_0}_{PBW}U(\sL_1)$ acting on $\cT^r_{\bar\lambda,\bar\mu}$ by $m_{hf_t}\circ\Phi_B$; hence $\sum_th f_t\Phi_B(x_t)=\sum_t\xi_tx_t$, and
	$$hQ\subset\calF^{p_0}_{PBW}U(\sL_1)\cdot\{x_1,\ldots,x_s\}\subset M .$$
	By Lemma~\ref{cor::generic::rank}, $h_0^m\,RM\subset Q$ for some $m$. Put $h_1:=h\,h_0^m$; then $h_1RM\subset hQ$, which gives the first chain of inclusions. By Lemma~\ref{lem::Artin::Rees}, applied to the Noetherian ring $R$, the finitely generated $R$-module $A$ and its submodule $RM$, there exists $c\geq0$ such that $RM\cap h_1^{c+1}A\subset h_1RM$. Hence
	$$M\cap h_1^{c+1}\cT^r_{\bar\lambda,\bar\mu}\subset RM\cap h_1^{c+1}A\subset h_1RM\subset\calF^{p_0}_{PBW}U(\sL_1)\cdot\{x_1,\ldots,x_s\}\subset M,$$
	so that $RM\cap h_1^{c+1}A\subset M\cap h_1^{c+1}\cT^r_{\bar\lambda,\bar\mu}$, and the second assertion follows. Finally, $h$ and $h_0$ are products of linear forms, hence homogeneous; if $M$ is graded, then so are $RM$ and $RM\cap h_1^{c+1}A$.
\end{proof}

\noindent\emph{Proof of Theorem~\ref{thm::Noether::L_1}.} We show by induction on $r$ that $\cV_r$ has property $\calP_{2r}$, which is the assertion of Theorem~\ref{thm::Noether::L_1}. For $r=0$ the only submodules of $\cT^0=\Bbbk$ are $0$ and $\Bbbk$, generated by $\emptyset$ and $\{1\}$. Let $r\geq1$, and assume that $\cV_{r-1}$ has property $\calP_{2r-2}$, and hence also property $\calP_{2r}$ (Lemma~\ref{lem::closure}(a)). By Lemma~\ref{lem::normalization}, it suffices to prove property $\calP_{2r}$ for a normalized module $\cT\in\cV_r$. Let $M\subset \cT$ be a submodule.

First, let us show that the deep part of $M$ is generated in bounded PBW degree. Let $x_1,\ldots,x_s\in M$, $h_1$ and $c$ be as in Proposition~\ref{prp::deep}. Then
\begin{equation}
	\label{eq::deep::part}
	M\cap h_1^{c+1}\cT\subset\calF^{p_0}_{PBW}U(\sL_1)\cdot\{x_1,\ldots,x_s\}.
\end{equation}

Second, let us show that the rest of $M$ lives on the degenerate hyperplanes. The polynomial $h_1^{c+1}$ is a product of degenerate forms $\ell_1\cdots\ell_k$, because $h$ and $h_0$ are. By Lemma~\ref{lem::strata}, all successive subquotients of the filtration
$$\cT\supset\ell_1\cT\supset\ell_1\ell_2\cT\supset\ldots\supset\ell_1\cdots\ell_k\cT=h_1^{c+1}\cT$$
are isomorphic to submodules of modules of $\cV_{r-1}$, which have property $\calP_{2r-2}$. By Lemma~\ref{lem::closure}, the module $\cT/h_1^{c+1}\cT$ has property $\calP_{2r-2}$. Hence its submodule $M/(M\cap h_1^{c+1}\cT)$ is equal to $\calF^{2r-2}_{PBW}U(\sL_1)\cdot\bar S$ for a finite set $\bar S$. Let $S'\subset M$ be a finite set of representatives of $\bar S$; then
$$M=\calF^{2r-2}_{PBW}U(\sL_1)\cdot S'+M\cap h_1^{c+1}\cT .$$

Finally, let us combine these two facts. By~\eqref{eq::deep::part} and $p_0\leq 2r$,
$$M\subset\calF^{2r-2}_{PBW}U(\sL_1)\cdot S'+\calF^{p_0}_{PBW}U(\sL_1)\cdot\{x_1,\ldots,x_s\}\subset\calF^{2r}_{PBW}U(\sL_1)\cdot\bigl(S'\cup\{x_1,\ldots,x_s\}\bigr),$$
and the last space is contained in $M$.
Hence $\cT$ has property $\calP_{2r}$, which completes the induction step and the proof of Theorem~\ref{thm::Noether::L_1}.
\hfill$\Box$

\begin{remark}
	\label{rmk::Laurent}
	The proof of Theorem~\ref{thm::Noether::L_1} can be transferred to modules of Laurent polynomials over the Witt algebra, with the point $z=0$ replaced by the point $z=1$. The necessary changes are explained in Subsection~\ref{sec::Witt::Noether}.
\end{remark}

\subsection{Application to the Hilbert series}
\label{sec::L_1::Hilb}

For a graded vector space $M=\bigoplus_{d\in\bZ}M_d$ with finite-dimensional components, such that $M_d=0$ for $d\ll0$, we denote by $H_M(t):=\sum_d\dim M_d\,t^d\in\bZ((t))$ its Hilbert series. We consider the modules of Theorem~\ref{thm::Noether::L_1} with the grading of Remark~\ref{rmk::grading} (the total degree in the variables $z_i$); the elements $e_k$ act by homogeneous operators of degree $k$. Put
$$Q_r(t):=\prod_{k=1}^{r}(1-t^k).$$

\begin{theorem}
	\label{thm::Hilb}
	\begin{enumerate}[label=(\alph*), itemsep=0pt, topsep=0pt]
		\item Let $\bar\lambda,\bar\mu\in\Bbbk^r$, and let $M\subset\cT^r_{\bar\lambda,\bar\mu}$ be a graded $\sL_1$-submodule. Then $Q_r(t)\cdot H_M(t)\in\bZ[t]$, that is, the Hilbert series of $M$ is of the form
		$$H_M(t)=\frac{P(t)}{\prod_{k=1}^{r}(1-t^k)}\qquad\text{with}\qquad P(t)\in\bZ[t].$$
		\item For every such $M$, the dimension $\dim M_d$ is, for $d\gg0$, a quasi-polynomial in $d$ of degree at most $r-1$, whose period divides $\mathrm{lcm}(1,2,\ldots,r)$.
	\end{enumerate}
\end{theorem}
\begin{proof}
	We use three elementary observations.
	(i) For graded submodules $M'\subset M$ and $N\subset M$ we have $H_N=H_{N\cap M'}+H_{N/(N\cap M')}$, and $N/(N\cap M')$ is a graded submodule of $M/M'$.
	(ii) If $s_1+\ldots+s_m\leq s$, then $\prod_{j=1}^m\prod_{k=1}^{s_j}(1-t^k)$ divides $\prod_{k=1}^s(1-t^k)$. Indeed, both sides are products of cyclotomic polynomials, and the cyclotomic polynomial $\Phi_e(t)$ enters them with the multiplicities $\sum_j\lfloor s_j/e\rfloor\leq\lfloor s/e\rfloor$.
	(iii) If $M_d=0$ for $d<0$ and $Q(t)H_M(t)\in\bZ[t,t^{-1}]$ for a polynomial $Q$, then $Q(t)H_M(t)\in\bZ[t]$, because $H_M(t)\in\bZ[[t]]$.
	By (iii), it suffices to prove that $Q_r\cdot H_M\in\bZ[t,t^{-1}]$; this allows us to shift the gradings by constants. We say that a graded module has property $\caH(Q)$ if $Q\cdot H_N\in\bZ[t,t^{-1}]$ for each of its graded submodules $N$. By (i), the property $\caH(Q)$ is inherited by graded submodules and graded quotients, it is preserved by shifts of the grading, and it is closed under extensions; moreover, $\caH(Q)$ implies $\caH(QQ')$ for every polynomial $Q'$.
	
	We prove (a) by induction on $r$; for $r=0$ there is nothing to prove. Let $\cT\in\cV_r$.
	
	First, let us show that the quotients $\cT/D\cT$ by products $D$ of degenerate forms have property $\caH(Q_r)$. All maps in the proof of Lemma~\ref{lem::strata} (the multiplication by $D$, the restriction to a hyperplane, and the multiplication by $\prod_{k<l}(z_k-z_l)^{\nu'_{kl}}$) are homogeneous. Hence $D\cT/D\ell \cT$ is isomorphic, up to a shift of the grading, to a graded submodule of a module of $\cV_{r-1}$, and $Q_{r-1}$ divides $Q_r$. By the induction hypothesis and (i), $\cT/D\cT$ has property $\caH(Q_r)$ for every product $D$ of degenerate forms.
	
	Second, let us show that we may assume $\cT$ to be normalized. The moves in the proof of Lemma~\ref{lem::normalization} respect the gradings: the submodule $z_i^k\cT$ is isomorphic to the modified module up to the shift of the grading by $k$, the quotient $\cT/z_i^k\cT$ has property $\caH(Q_r)$ by the previous paragraph, and the de Rham differential $\Phi$ is an isomorphism of degree $0$, since $\Phi(z^m)=(m+\mu)z^m$. Hence we may assume that $\cT$ is normalized.
	
	Finally, let us treat a normalized module $\cT$. Let $M\subset \cT$ be a graded submodule, and let $h_1$ and $c$ be as in Proposition~\ref{prp::deep}. The graded $R$-module $N:=M\cap h_1^{c+1}\cT=RM\cap h_1^{c+1}A$ is finitely generated, and $R$ is the polynomial ring in the elementary symmetric polynomials of the sets of variables $G$ and $B$, whose degrees are $1,\ldots,|G|$ and $1,\ldots,|B|$. By the Hilbert--Serre theorem (see, e.g.,~\cite[Theorem~11.1]{AM}),
	$$H_N(t)\cdot\prod_{k=1}^{|G|}(1-t^k)\prod_{k=1}^{|B|}(1-t^k)\in\bZ[t],$$
	and by (ii), applied to $|G|+|B|=r$, we get $Q_r\cdot H_N\in\bZ[t]$. On the other hand, $M/N$ is a graded submodule of $\cT/h_1^{c+1}\cT$, which has property $\caH(Q_r)$ as shown above, since $h_1^{c+1}$ is a product of degenerate forms. Hence $Q_r\cdot H_M=Q_r\cdot H_N+Q_r\cdot H_{M/N}\in\bZ[t,t^{-1}]$. This proves (a).
	
	Part (b): by (a), $H_M(t)=P(t)/\prod_{k=1}^r(1-t^k)$ with $P\in\bZ[t]$. The poles of this rational function are roots of unity of orders at most $r$, and the order of the pole at $t=1$ is at most $r$; the partial fraction decomposition gives the claim.
\end{proof}

\begin{remark}
	\label{rmk::Hilb}
	The denominator in Theorem~\ref{thm::Hilb} cannot be improved in general: for the submodule of symmetric tensors $M=S^r(\cT_{\lambda,\mu})\subset\cT_{\lambda,\mu}^{\otimes r}$ we get the Hilbert series of the symmetric polynomials, $H_M(t)=\prod_{k=1}^r(1-t^k)^{-1}$. For $\bar\lambda=\bar\mu=0$ the element $e_0=z\partial_z$ of the Lie algebra $\cW_1^+=\sL_0$ of vector fields of nonnegative degree acts on $\cT_0^{\otimes r}$ by the total degree, so that the graded $\sL_1$-submodules of $\cT_0^{\otimes r}$ are exactly its $\cW_1^+$-submodules; the rationality of their Hilbert series also follows from the results of Sam, Snowden and Tosteson~\cite[Theorem~1.4]{SST}, which were obtained by different methods for the polynomial representations of $\cW_1^{\pol}$ and of $\cW_1^+$.
\end{remark}

For non-integral twists $\bar\nu$ the modules $\cT^r_{\bar\lambda,\bar\mu;\bar\nu}$ are not obtained from the modules of the class $\cV_r$ by taking submodules and subquotients, and the proof of Theorem~\ref{thm::Noether::L_1} does not apply directly. The main obstacle is the weight $1$: the de Rham differential twisted by non-integral powers of the forms $z_i-z_j$ is, in general, not surjective even after inverting all degenerate forms, so that the weight $1$ cannot be replaced by the weight $0$ as in Lemma~\ref{lem::normalization}.
We close this section with the following conjecture:

\begin{conjecture}
\label{item::Noetherian}
For any collections of {\bf{complex}} numbers $\bar\lambda,\bar\mu\in\bC^r$ and $\bar\nu\in\bC^{\binom{r}{2}}$, the $\sL_1$-module $\cT^r_{\bar{\lambda},\bar{\mu};\bar{\nu}}$ is Noetherian, and the Hilbert series of its graded submodules are rational functions whose denominators are products of factors of the form $(1-t^i)$. For integral $\nu_{ij}$ this follows from Theorems~\ref{thm::Noether::L_1} and~\ref{thm::Hilb}, by~\eqref{eq::strata::twist} and Lemma~\ref{lem::strata}.
\end{conjecture}

\section{Noetherianity of $\sL_d(n)$-modules for general $d$ and $n$}
\label{sec::Noether::Wn}

This section is devoted to proving the main result regarding the Noetherian property of tensor products of tensor fields, viewed as modules over a Lie subalgebra $\fg$ of finite codimension in the Lie algebra $\cW_n^{\pol}$.
We will deduce it from the results of Section~\ref{sec::Noether::T_lambda} in several steps, proceeding by cases based on different values of $d$ and $n$. Throughout this section, $\cT_\lambda$ is the $\cW_n^{\pol}$-module of polynomial tensor fields~\eqref{eq::T_lambda::pol}; accordingly, all Lie algebras acting on it are subalgebras of $\cW_n^{\pol}$, and we write $\sL_d(n)$ for $\sL_d^{\pol}(n)$.

\subsection{Noetherianity for all $d$ with $n=1$}
Note that Noetherianity ascends along inclusions of Lie algebras. So we have to examine them for $\sL_d$ for different $d$.
Indeed, for every $d\geq 1$ we have the following embedding of Lie algebras:
\begin{equation}
\label{eq::d:L1:to:L1}
	\imath_{d}:\sL_1{\hookrightarrow} \sL_d, \qquad e_k\mapsto \frac{e_{dk}}{d}. 
\end{equation}
\begin{proposition}
	\label{prp::Ld-L1}	
	For any $\lambda,\mu\in\Bbbk$ and any $d\ge1$, the $\sL_1$-module $\cT_{\lambda,\mu}$, viewed as a module over the Lie subalgebra $\imath_{d}(\sL_1)$, is isomorphic to a direct sum of $\sL_1$-modules of the same kind:
	$$
	\cT_{\lambda,\mu}\Big|_{\imath_d(\sL_1)} \ \simeq\ \bigoplus_{s=0}^{d-1} \cT_{\lambda,\,\nu_s}, \qquad \nu_s:=\frac{\mu+s}{d} - \lambda\,\frac{d-1}{d}.
	$$	
\end{proposition}
\begin{proof}
	By \eqref{eq::L_1::action}, for $m,k\ge0$,
	$$
	\imath_d(e_k)\cdot z^{m+\mu}\partial^{-\lambda} = \frac{1}{d}\,e_{dk}\cdot z^{m+\mu}\partial^{-\lambda} = \frac{m+\mu+(dk+1)\lambda}{d}\, z^{m+dk+\mu}\partial^{-\lambda}.
	$$
	Fix $s\in\{0,1,\ldots,d-1\}$ and write $m=dj+s$ with $j\ge0$. Then
	$$
	\imath_d(e_k)\cdot z^{dj+s+\mu}\partial^{-\lambda} = \left(j + k\lambda + \frac{s+\mu+\lambda}{d}\right) z^{d(j+k)+s+\mu}\partial^{-\lambda}
	= \bigl(j+\nu_s+(k+1)\lambda\bigr)\, z^{d(j+k)+s+\mu}\partial^{-\lambda},
	$$
	where $\nu_s = \frac{\mu+s}{d}-\lambda\frac{d-1}{d}$, as one sees by comparing the two constant terms $\frac{s+\mu+\lambda}{d}$ and $\nu_s+\lambda$. Comparing with \eqref{eq::L_1::action} for the action of $e_k$ on $\cT_{\lambda,\nu_s}$, the linear map
	$$
	\Bbbk[z^d]\,z^{s+\mu}\partial^{-\lambda} \ \xrightarrow{\ \sim\ }\ \cT_{\lambda,\nu_s}, \qquad z^{dj+s+\mu}\partial^{-\lambda}\longmapsto y^{j}(y^{\nu_s}\partial^{-\lambda}),
	$$
	intertwines the action of $\imath_d(e_k)$ with that of $e_k$. Since $\Bbbk[z]z^\mu\partial^{-\lambda}=\bigoplus_{s=0}^{d-1}\Bbbk[z^d]z^{s+\mu}\partial^{-\lambda}$, this proves the proposition.
\end{proof}

\begin{corollary}
	\label{cor::Ld::noether}
	For any collection $\bar\lambda=(\lambda_1,\ldots,\lambda_r)\in\Bbbk^r$, the tensor module $\cT^{r}_{\bar\lambda}$ has property $\calP_{2r}$ as an $\imath_d(\sL_1)$-module, and hence also as an $\sL_d$-module. In particular, it is a Noetherian $\sL_d$-module.
\end{corollary}
\begin{proof}
	By Lemma~\ref{lem::closure}(d) and~(a), it suffices to show that $\cT^{r}_{\bar\lambda}$, restricted to the subalgebra $\imath_d(\sL_1)\subseteq \sL_d$, has property $\calP_{2r}$. By Proposition~\ref{prp::Ld-L1} applied to each factor $\cT_{\lambda_i}=\cT_{\lambda_i,0}$,
	$$
	\cT^{r}_{\bar\lambda}\big|_{\imath_d(\sL_1)} \ \simeq\ \bigoplus_{\bar s\in\{0,\ldots,d-1\}^r}\ \bigotimes_{i=1}^{r} \cT_{\lambda_i,\,\nu_{s_i}}, \qquad \nu_{s_i}=\frac{s_i-(d-1)\lambda_i}{d},
	$$
	a finite direct sum of $d^r$ modules of the form $\cT^r_{\bar\lambda,\bar\nu(\bar s)}$. By Theorem~\ref{thm::Noether::L_1}, which holds for arbitrary parameters $\bar\mu$, and by Lemma~\ref{lem::closure}(c), this finite direct sum has property $\calP_{2r}$.
\end{proof}

\subsection{Noetherianity for arbitrary $n$}
\label{sec::Noether::n}
There are $n$ pairwise commuting coordinate embeddings $\alpha_m:\sL_1\hookrightarrow\sL_1(n)$, $m=1,\ldots,n$:
$$
\alpha_m: e_k \mapsto x_m^{k+1}\partial_m, \qquad \alpha_m(\sL_1)\subset \sL_1(n)\subset \cW_n=\Bigl\{\textstyle\sum_m f_m(x_1,\ldots,x_n)\partial_m\Bigr\}.
$$
Recall from the previous subsection the embedding $\imath_d:\sL_1\hookrightarrow\sL_d$, $e_k\mapsto e_{dk}/d$, defined for every $d\geq 1$. Composing $\alpha_m$ with $\imath_d$ in each variable, we obtain, for every $d\geq 1$, the following chain of embeddings of Lie algebras:
\begin{equation}
	\label{eq::L1_to_Wn}
	\sL_1^{\oplus n}:=\sL_1^{(x_1)}\oplus \ldots \oplus\sL_1^{(x_n)} \hookrightarrow 
	\sL_d^{(x_1)}\oplus \ldots \oplus \sL_d^{(x_n)}  \hookrightarrow \sL_d(n) \hookrightarrow \cW_n.
\end{equation}

To deduce Theorem~\ref{thm::Noether} below, we need the following version of Theorem~\ref{thm::Noether::L_1} for several copies of $\sL_1$.
\begin{lemma}
	\label{lem::Noether::parts}
	Let $\{1,\ldots,r\}=I_1\sqcup\ldots\sqcup I_n$ be a decomposition into $n\geq1$ parts (some of them may be empty), let $\bar\lambda,\bar\mu\in\Bbbk^r$, and consider
	$$
	\cT^r_{\bar\lambda,\bar\mu}=\bigotimes_{q=1}^{n}\Bigl(\bigotimes_{i\in I_q}\cT_{\lambda_i,\mu_i}\Bigr)
	$$
	as a module over the Lie algebra $\sL_1^{\oplus n}=\sL_1^{(1)}\oplus\ldots\oplus\sL_1^{(n)}$, whose $q$-th summand $\sL_1^{(q)}\simeq\sL_1$ acts diagonally on the $q$-th tensor factor $\bigotimes_{i\in I_q}\cT_{\lambda_i,\mu_i}$.
	\begin{enumerate}[label=(\alph*), itemsep=0pt, topsep=0pt]
		\item The $\sL_1^{\oplus n}$-module $\cT^r_{\bar\lambda,\bar\mu}$ has property $\calP_{2r}$.
		\item If $M\subset\cT^r_{\bar\lambda,\bar\mu}$ is a graded $\sL_1^{\oplus n}$-submodule, then $H_M(t)\cdot\prod_{q=1}^{n}\prod_{k=1}^{|I_q|}(1-t^k)\in\bZ[t]$.
	\end{enumerate}
\end{lemma}
\begin{proof}
	Let $\cV'_r$ be the class of all modules of the lemma, for all $n\geq1$, all decompositions and all $\bar\lambda,\bar\mu\in\Bbbk^r$; for $n=1$ these are the modules of the class $\cV_r$ of Definition~\ref{def::class::V}. The element $e_u\in\sL_1^{(q)}$ acts by the operator $L^{(q)}_u:=\sum_{i\in I_q}\ell^{(i)}_u=X^{(q)}_u+m_{\omega^{(q)}(u)}$, where $X^{(q)}_u$ and $\omega^{(q)}(u)$ are the sums of~\eqref{eq::L_u::Leibniz} restricted to $i\in I_q$. The proofs of Theorems~\ref{thm::Noether::L_1} and~\ref{thm::Hilb} apply to the class $\cV'_r$, with $U(\sL_1)$ replaced by $U(\sL_1^{\oplus n})$, with the following changes.
	\begin{enumerate}[label=(\arabic*), itemsep=0pt, topsep=0pt]
		\item The degenerate forms are $z_i$ and $z_i-z_j$, where $i\neq j$ belong to the same part. Formula~\eqref{eq::strata::twist} holds for $L^{(q)}_u$, with $c_\ell(u)$ replaced by $0$ if the variables of $\ell$ do not belong to $I_q$. In Lemma~\ref{lem::strata} the restriction to $\{z_i=0\}$ removes $i$ from its part, and the restriction to $\{z_i=z_j\}$ replaces the part $I_q\ni i,j$ by $I_q\setminus\{j\}$; hence the strata are isomorphic to submodules of modules of $\cV'_{r-1}$.
		\item The group $\Gamma$ of Lemma~\ref{lem::key::good} is replaced by $\prod_qS_{I_q}$, and the group $\Gamma$ of a normalized module by $\prod_q(S_{G_q}\times S_{B_q})$, where $G_q$ and $B_q$ are the sets of good and bad variables of the part $I_q$. Then $R=A^\Gamma=R_1\otimes\ldots\otimes R_n$, where $R_q$ is the ring of invariants in the variables of $I_q$.
		\item In Lemmas~\ref{lem::key::good} and~\ref{lem::key::mult} the construction is performed in every part separately: for $f_q\in R_q$ it gives an element $\xi_q(f_q)\in\calF^{2|G_q|+|B_q|}_{PBW}U(\sL_1^{(q)})$ acting by $m_{h_qf_q}\circ\Phi_{B_q}$, where $h_q$ is the polynomial $h$ of Lemma~\ref{lem::key::mult} for the variables of $I_q$ and $\Phi_{B_q}:=\prod_{b\in B_q}\Phi_b$. The elements $\xi_q(f_q)$ for different $q$ commute, their product belongs to $\calF^{p_0}_{PBW}U(\sL_1^{\oplus n})$ with $p_0:=\sum_q(2|G_q|+|B_q|)\leq2r$, and it acts by $m_{hf_1\cdots f_n}\circ\Phi_B$ with $h:=h_1\cdots h_n$. Since $R$ is spanned by the products $f_1\cdots f_n$, this gives Lemma~\ref{lem::key::mult} for the class $\cV'_r$.
		\item In the proof of Lemma~\ref{lem::support} the partial derivatives $\partial_i$, $i\in I_q$, are expressed through the derivations $X^{(q)}_{u_k}$, $u_k=z^{k+2}$, $0\leq k<|I_q|$, in every part separately.
		\item In the proof of Theorem~\ref{thm::Hilb} the polynomial $Q_r$ is replaced by
		$$Q(t):=\prod_{q=1}^n\prod_{k=1}^{|I_q|}(1-t^k).$$
		The strata are submodules of modules of $\cV'_{r-1}$ whose parts are obtained by removing one element from one part, so that their polynomials divide $Q(t)$, and $R$ is the polynomial ring in the elementary symmetric polynomials of the sets $G_q$ and $B_q$, so that observation~(ii) of that proof, applied to $|G_q|+|B_q|=|I_q|$, gives $Q(t)\cdot H_N(t)\in\bZ[t]$.
	\end{enumerate}
	All other arguments of Section~\ref{sec::Noether::T_lambda} apply verbatim.
\end{proof}

\begin{theorem}
	\label{thm::Noether}	
	For every $d\in\bN$ and every collection of weights $\lambda_1,\ldots,\lambda_r\in\bZ^n$, the tensor product of $\cW_n^{\pol}$-modules $\cT_{\lambda_1,\ldots,\lambda_r}:=\cT_{\lambda_1}\otimes \ldots\otimes \cT_{\lambda_r}$ has property $\calP_{2nr}$ as an $\sL_1^{\oplus n}$-module, where $\sL_1^{\oplus n}\hookrightarrow \cW_n$ is the embedding of~\eqref{eq::L1_to_Wn}, and hence also as an $\sL_d(n)$-module. In particular, $\cT_{\lambda_1,\ldots,\lambda_r}$ is a Noetherian $\sL_d(n)$-module.
\end{theorem}
\begin{proof}
	Let us first treat the case of a single tensor factor $\cT_{\lambda}$, for a dominant weight $\lambda=(\lambda_1\geq\lambda_2\geq \ldots\geq\lambda_n)\in\bZ^n$.
	Let $V_{\lambda}$ be the irreducible $\GL_n$-module of highest weight $\lambda\in\bZ^{n}$. By the basic representation theory of $\fgl_n$, $V_\lambda$ decomposes into weight spaces for the diagonal (Cartan) subalgebra $\fh\subset\fgl_n$:
	$$V_{\lambda}\simeq \bigoplus_{\alpha\leq \lambda} (V_{\lambda})_{\alpha}, \qquad h\cdot v = \alpha(h)\,v \ \text{ for all } h\in\fh,\ v\in(V_{\lambda})_{\alpha}. $$
	Fix a weight basis $\{v^i_{\alpha}\mid \alpha\in \bZ^{n} , \ i=1,\ldots,\dim (V_{\lambda})_{\alpha}\}$ of $V_\lambda$, adapted to this decomposition. For example, when $\lambda$ is a partition (i.e.\ $\lambda_n\geq 0$), such an $\fh$-eigenbasis may be indexed by semistandard Young tableaux of shape $\lambda$.
	Therefore, as a module over the Lie subalgebra $\sL_1^{(x_1)}\oplus\ldots\oplus\sL_1^{(x_n)}$, the module $\cT_{\lambda}$ is isomorphic to the following direct sum of tensor products of modules:
	\begin{equation}
		\label{eq::T::m}	
		\cT_{\lambda}\simeq\bigoplus_{\alpha}\bigoplus_{i=1}^{\dim (V_{\lambda})_{\alpha}}  \Bbbk[x_1,\ldots,x_n] \prod_{s=1}^{n}\partial_{s}^{-\alpha_s} \simeq
		\bigoplus_{\alpha}\bigoplus_{i=1}^{\dim (V_{\lambda})_{\alpha}}
		(\Bbbk[x_1]\partial_1^{-\alpha_1})\otimes \ldots\otimes (\Bbbk[x_n]\partial_n^{-\alpha_n}).
	\end{equation}
	Each summand on the right-hand side of~\eqref{eq::T::m} is a tensor product of $n$ modules of densities in the $n$ different variables $x_1,\ldots,x_n$, that is, a module of Lemma~\ref{lem::Noether::parts} with $n$ variables and $n$ parts of cardinality one.
	
	Similarly, for any collection of $\fgl_n$-weights $(\lambda^{1},\ldots,\lambda^{r})$ with $\lambda^{s}=(\lambda_1^{s}\geq \lambda_2^{s}\geq \ldots \geq \lambda_n^{s})\in \bZ^n$, the tensor product $\cT_{\lambda^1}\otimes \ldots \otimes \cT_{\lambda^r}$, viewed as a module over the Lie subalgebra
	$\sL_1^{(x_1)}\oplus\ldots\oplus\sL_1^{(x_n)}$, decomposes as a direct sum indexed by collections of weights $\alpha^s = (\alpha_1^s,\ldots,\alpha_n^s)$ of $V_{\lambda^s}$, $s=1,\ldots,r$:
	$$
	\bigotimes_{s=1}^r \bigotimes_{i=1}^{n}\cT_{\alpha_i^s}^{(x_i)} \simeq  \bigotimes_{i=1}^n
	\cT_{(\alpha_i^1,\ldots,\alpha_i^r)}^{(x_i)}.
	$$
	Here $\cT_{(\alpha_i^1,\ldots,\alpha_i^r)}^{(x_i)}=\cT_{\alpha_i^1}\otimes\ldots\otimes\cT_{\alpha_i^r}$ is the tensor product of $r$ modules of densities in the variables $x_i^{(1)},\ldots,x_i^{(r)}$ (the $i$-th coordinates of the $r$ points), on which $\sL_1^{(x_i)}$ acts diagonally. Therefore every summand is a module of Lemma~\ref{lem::Noether::parts} with $nr$ variables $x_i^{(s)}$ and $n$ parts $\{x_i^{(1)},\ldots,x_i^{(r)}\}$, $i=1,\ldots,n$. For $d>1$ the embedding~\eqref{eq::L1_to_Wn} involves the maps $\imath_d$; restricting to $\imath_d(\sL_1)$ in every variable and using Proposition~\ref{prp::Ld-L1}, we see that every summand becomes a finite direct sum of modules of the same kind, now with nonzero parameters $\bar\mu$, which are still covered by Lemma~\ref{lem::Noether::parts}.
	By Lemma~\ref{lem::Noether::parts}(a) and Lemma~\ref{lem::closure}(c), the finite direct sum $\cT_{\lambda^1}\otimes\ldots\otimes\cT_{\lambda^r}$ has property $\calP_{2nr}$ as an $\sL_1^{\oplus n}$-module. By Lemma~\ref{lem::closure}(d) and~(a), it has property $\calP_{2nr}$ also as an $\sL_d(n)$-module, and it is a Noetherian $\sL_d(n)$-module.
\end{proof}

\subsection{Noetherianity for subalgebras of $\cW_n$}

\begin{corollary}
	\label{cor::graded::noether}
	Let $\fg\subset\cW_1^{\pol}$ be a \emph{graded} infinite-dimensional Lie subalgebra,
	i.e.\ $\fg=\bigoplus_{i\ge -1}(\fg\cap\bC e_i)$. Then for any $\bar\lambda\in\Bbbk^r$,
	the $\fg$-module $\cT^r_{\bar\lambda}$ has property $\calP_{2r}$; in particular, it is Noetherian.
\end{corollary}
\begin{proof}
	Let $\Sigma:=\{i\ge 1 : e_i\in\fg\}$; this set is infinite, since $\fg$ is graded and infinite-dimensional. Choose $a<b$ in $\Sigma$. Since $[e_x,e_y]=(y-x)e_{x+y}$ and $\fg$ is graded, $x+y\in\Sigma$ whenever $x\neq y$ belong to $\Sigma$. By induction, $ma+nb\in\Sigma$ for all $m,n\geq1$; in particular, $kab=(k-1)b\cdot a+a\cdot b\in\Sigma$ for all $k\geq2$. Hence $\Sigma$ contains all positive multiples of $d:=2ab$, that is, $\imath_d(\sL_1)\subset\fg$, see~\eqref{eq::d:L1:to:L1}. By Corollary~\ref{cor::Ld::noether}, the $\imath_d(\sL_1)$-module $\cT^r_{\bar\lambda}$ has property $\calP_{2r}$, and hence, by Lemma~\ref{lem::closure}(d), so does the $\fg$-module $\cT^r_{\bar\lambda}$.
\end{proof}

\begin{corollary}
	\label{cor::closed::noether}
	Let $\fg\subset\cW_n^{\pol}$ be a Lie subalgebra of finite codimension which is
	\emph{closed} in the topology induced from $\cW_n$ — equivalently,
	$\fg\supseteq\sL_d(n)$ for some $d>>0$. Then for any collection of $\fgl_n$-weights $\bar\lambda\in(\bZ^n)^r$, the
	$\fg$-module $\cT^r_{\bar\lambda}$ has property $\calP_{2nr}$; in particular, it is Noetherian.
\end{corollary}
\begin{proof}
    Follows from Theorem~\ref{thm::Noether} and Lemma~\ref{lem::closure}(d), since $\sL_d(n)\subset\fg$.
\end{proof}

Let us now derive the case of the Lie subalgebra of finite codimension in $\cW_n^{\pol}$ that is not necessarily graded or closed.

For $0\ne v\in\cW_n^{\pol}$, write $v=\sum_{k\ge -1} v_k$ (finite sum, $v_k\in(\cW_n)_k$ in the
grading~\eqref{eq::Wn::decomposition}) and set
$$
\mathrm{pdeg}(v):=\max\{k : v_k\ne0\},\qquad pr(v):=v_{\mathrm{pdeg}(v)}\ \ (\text{the \emph{principal part}}).
$$
Likewise for $0\ne m\in\cT^r_{\bar\lambda}=\bigoplus_{k\ge k_0}(\cT^r_{\bar\lambda})_k$ define
$\mathrm{pdeg}(m)$, $pr(m)$ using the polynomial-degree grading. For a subspace
$Y\subseteq\cW_n^{\pol}$ and $k\ge-1$, write $Y_{\le k}:=Y\cap\bigoplus_{j\le k}(\cW_n)_j$ (finite-dimensional),
and let $pr(Y)_k\subseteq(\cW_n)_k$ be the image of $Y_{\le k}$ under the projection
$\bigoplus_{j\le k}(\cW_n)_j\twoheadrightarrow(\cW_n)_k$; set $pr(Y):=\bigoplus_k pr(Y)_k$. Define
$Y_{\le k}$, $pr(Y)_k$, $pr(Y)$ analogously for $Y\subseteq\cT^r_{\bar\lambda}$.

\begin{lemma}\label{lem::symbol}
	For nonzero $v,w\in\cW_n^{\pol}$: if $[v,w]\ne0$ then $\mathrm{pdeg}([v,w])\le \mathrm{pdeg}(v)+\mathrm{pdeg}(w)$,
	with equality and $pr([v,w])=[pr(v),pr(w)]$ whenever $[pr(v),pr(w)]\ne0$. The same holds for
	the action of $\cW_n^{\pol}$ on $\cT^r_{\bar\lambda}$: for nonzero $v\in\cW_n^{\pol}$, $m\in\cT^r_{\bar\lambda}$,
	if $v\cdot m\ne 0$ then $\mathrm{pdeg}(v\cdot m)\le\mathrm{pdeg}(v)+\mathrm{pdeg}(m)$, with equality
	and $pr(v\cdot m)=pr(v)\cdot pr(m)$ whenever $pr(v)\cdot pr(m)\ne0$.
\end{lemma}
\begin{proof}
	Write $v=\sum_{i\le K}v_i$ ($K=\mathrm{pdeg}(v)$), $w=\sum_{l\le L}w_l$ ($L=\mathrm{pdeg}(w)$); then
	$[v,w]=\sum_{i\le K,\,l\le L}[v_i,w_l]$, and $[v_i,w_l]\in(\cW_n)_{i+l}$ with $i+l\le K+L$,
	equality forcing $i=K,l=L$. The degree-$(K+L)$ component of $[v,w]$ is therefore exactly
	$[v_K,w_L]=[pr(v),pr(w)]$. The action case is identical.
\end{proof}

\begin{corollary}\label{cor::pr::algebra}
	For any Lie subalgebra $\fg\subset\cW_n^{\pol}$, $pr(\fg)$ is a graded Lie subalgebra of $\cW_n^{\pol}$.
	For any $\fg$-submodule $N\subset\cT^r_{\bar\lambda}$, $pr(N)$ is a graded subspace of
	$\cT^r_{\bar\lambda}$ stable under the $pr(\fg)$-action.
\end{corollary}
\begin{proof}
	Immediate from Lemma~\ref{lem::symbol}: if $\xi=pr(v)\in pr(\fg)_K$, $\eta=pr(w)\in pr(\fg)_L$
	with $v,w\in\fg$, then $[v,w]\in\fg$ has degree-$(K+L)$ part equal to $[\xi,\eta]$ (or $[v,w]=0$,
	trivially fine), so $[\xi,\eta]\in pr(\fg)_{K+L}$. Same argument for the action on $N$.
\end{proof}

\begin{lemma}\label{lem::codim}
	For any subspace $\fg\subseteq\cW_n^{\pol}$, $\operatorname{codim}_{\cW_n^{\pol}}pr(\fg)\le\operatorname{codim}_{\cW_n^{\pol}}\fg$.
\end{lemma}
\begin{proof}
	For each $k$ there is a short exact sequence
	$$
	0\to (\cW_n)_{\le k-1}/\fg_{\le k-1}\to (\cW_n)_{\le k}/\fg_{\le k}\to (\cW_n)_k/pr(\fg)_k\to0,
	$$
	so $\dim\big((\cW_n)_{\le k}/\fg_{\le k}\big)=\dim\big((\cW_n)_{\le k-1}/\fg_{\le k-1}\big)+\dim\big((\cW_n)_k/pr(\fg)_k\big)$.
	Since $(\cW_n)_{\le k}/\fg_{\le k}$ embeds into $\cW_n^{\pol}/\fg$, the left side is $\le c:=\operatorname{codim}\fg$
	for every $k$. Telescoping, 
    $$\sum_{j\le k}\dim\big((\cW_n)_j/pr(\fg)_j\big)\le c \text{ for every }k,$$
	hence (letting $k\to\infty$) $\operatorname{codim}\bigl(pr(\fg)\bigr):=\sum_j \dim\big((\cW_n)_j/pr(\fg)_j\big)\le c$.
\end{proof}

\begin{corollary}
	\label{cor::g::noether}
	For any Lie subalgebra $\fg\subset\cW_n^{\pol}$ of finite codimension and any collection of $\fgl_n$-weights
	$\bar\lambda\in(\bZ^n)^r$, the $\fg$-module $\cT^r_{\bar\lambda}$ has property $\calP_{2nr}$; in particular, it is Noetherian.
\end{corollary}
\begin{proof}
	By Corollary~\ref{cor::pr::algebra} and Lemma~\ref{lem::codim}, $pr(\fg)$ is a graded Lie
	subalgebra of finite codimension $\le\operatorname{codim}\fg$. As in the original argument, finite
	codimension of a \emph{graded} subalgebra forces $pr(\fg)_k=(\cW_n)_k$ for all $k$ beyond some
	bound $D$, i.e.\ $\sL_D(n)\subseteq pr(\fg)$.

	Let $N\subset\cT^r_{\bar\lambda}$ be a $\fg$-submodule. By Corollary~\ref{cor::pr::algebra},
	$pr(N)$ is stable under $pr(\fg)$, hence (since $\sL_D(n)\subseteq pr(\fg)$) is an
	$\sL_D(n)$-submodule of $\cT^r_{\bar\lambda}$. By Theorem~\ref{thm::Noether}, the $\sL_D(n)$-module $\cT^r_{\bar\lambda}$
	has property $\calP_{2nr}$, so $pr(N)=\calF^{2nr}_{PBW}U(\sL_D(n))\cdot S$ for a finite subset $S\subset pr(N)$. Replacing the elements of $S$ by their homogeneous components, which belong to $pr(N)$, since $pr(N)$ is graded, we may assume that $S=\{\bar f_1,\dots,\bar f_k\}$, where $\bar f_i$ is homogeneous of degree $P_i$. We lift each $\bar f_i$ to
	$f_i\in N$ with $\mathrm{pdeg}(f_i)=P_i$, $pr(f_i)=\bar f_i$. Let $N'':=U(\fg)\langle f_1,\dots,f_k\rangle\subseteq N$.

	Let us explain why $N$ and $N''$ coincide. We show $N_{\le P}\subseteq N''$ for every $P$, by ordinary induction on $P$.
	The base case ($P$ below the grading floor of $\cT^r_{\bar\lambda}$) is trivial. Assume
	$N_{\le P-1}\subseteq N''$. Take $m\in N_{\le P}$; if $\mathrm{pdeg}(m)<P$ we are done by the induction
	hypothesis, so assume $\mathrm{pdeg}(m)=P$. Then $pr(m)\in pr(N)_P$, so write
	$pr(m)=\sum_i \xi_i\cdot\bar f_i$ with $\xi_i\in U(\sL_D(n))$ a sum of PBW monomials of degree
	exactly $P-P_i$ and of PBW length at most $2nr$. Replacing each factor $\eta\in\sL_D(n)$ occurring in $\xi_i$ by a fixed
	linear lift $v(\eta)\in\fg$ with $\mathrm{pdeg}(v(\eta))=k$, $pr(v(\eta))=\eta$ (possible since
	$\sL_D(n)\subseteq pr(\fg)$) produces $\tilde\xi_i\in U(\fg)$ of the same PBW length. By repeated
	application of the action half of Lemma~\ref{lem::symbol}, $\mathrm{pdeg}(\tilde\xi_i\cdot f_i)\le
	(P-P_i)+P_i=P$, with degree-$P$ part exactly $\xi_i\cdot\bar f_i$. Hence
	$m-\sum_i\tilde\xi_i\cdot f_i\in N$ has $\mathrm{pdeg}<P$, i.e.\ lies in $N_{\le P-1}\subseteq N''$
	by the induction hypothesis; and $\sum_i\tilde\xi_i\cdot f_i\in N''$. So $m\in N''$.

	Thus $N=N''=U(\fg)\langle f_1,\dots,f_k\rangle$ is finitely generated. Since the elements $\tilde\xi_i$ have PBW length at most $2nr$, the same induction on $P$ shows that $N_{\le P}\subseteq\calF^{2nr}_{PBW}U(\fg)\langle f_1,\dots,f_k\rangle$ for all $P$, that is, $N=\calF^{2nr}_{PBW}U(\fg)\langle f_1,\dots,f_k\rangle$.
\end{proof}

\begin{corollary}
	\label{cor::Hilb::Wn}
	Let $\fg\subset\cW_n^{\pol}$ be a graded Lie subalgebra of finite codimension, and let $D\geq1$ be such that $\sL_D(n)\subset\fg$ (see the proof of Corollary~\ref{cor::g::noether}). Let $M=M'/M''$, where $M''\subset M'$ are graded $\fg$-submodules of a finite direct sum of tensor products $\cT_{\lambda^1}\otimes\ldots\otimes\cT_{\lambda^s}$ of $s\leq r$ modules of tensor fields, with the induced grading. Then
	$$H_M(t)\cdot\prod_{k=1}^{r}(1-t^{Dk})^n\in\bZ[t].$$
	In particular, the Hilbert series $H_M(t)$ is a rational function whose denominator is a product of factors of the form $(1-t^m)$, and $\dim M_d$ is a quasi-polynomial in $d$ for $d\gg0$.
\end{corollary}
\begin{proof}
	By the additivity of Hilbert series (see (i) in the proof of Theorem~\ref{thm::Hilb}), it suffices to consider a graded $\fg$-submodule $M$ of $\cT:=\cT_{\lambda^1}\otimes\ldots\otimes\cT_{\lambda^s}$. Since $\imath_D(\sL_1)^{\oplus n}\subset\sL_D(n)\subset\fg$, see~\eqref{eq::L1_to_Wn}, $M$ is a graded $\imath_D(\sL_1)^{\oplus n}$-submodule of $\cT$. By the proof of Theorem~\ref{thm::Noether}, $\cT$ is a finite direct sum of graded $\imath_D(\sL_1)^{\oplus n}$-submodules $W_j$, and each $W_j$ is isomorphic to a module of Lemma~\ref{lem::Noether::parts} with $ns$ variables and $n$ parts of cardinality $s$, for which $\prod_q\prod_{k=1}^{|I_q|}(1-t^k)=\prod_{k=1}^s(1-t^k)^n$. By Proposition~\ref{prp::Ld-L1}, this isomorphism multiplies the degrees by $D$ and shifts them by a constant $c_j\geq0$. Hence, by Lemma~\ref{lem::Noether::parts}(b), every graded submodule $N\subset W_j$ satisfies $H_N(t)\prod_{k=1}^s(1-t^{Dk})^n\in t^{c_j}\bZ[t]$. Filtering $\cT$ by the partial sums $W_1\oplus\ldots\oplus W_j$ and using the additivity of Hilbert series, we obtain the claim, since $s\leq r$.
\end{proof}

\section{The Witt algebra}
\label{sec::Witt}

In this section we replace the positive Witt algebra $\sL_1$ by the whole polynomial Witt algebra $\witt$ introduced in Section~\ref{sec::Witt::def}.
We first explain how the proof of Noetherianity from Section~\ref{sec::Noether::T_lambda} is adapted to $\witt$ (Subsection~\ref{sec::Witt::Noether}): the proof of Theorem~\ref{thm::Noether::L_1} is performed at the point $z=1$ instead of the point $z=0$, and we list the changes which have to be made.
Since $\witt$ contains $e_k$ for all integers $k$, and in particular both $e_1$ and $e_{-1}$, modules of densities are much more rigid over $\witt$ than over $\sL_1$: for $\lambda\neq0,1$ every nonzero submodule of $S^r(\cT_\lambda)$ is large (Theorem~\ref{thm::Witt::submod}).
To prove this we need to know which operators arise as images of elements of the universal enveloping algebra $U(\witt)$; this is explained in Subsection~\ref{sec::Witt::U}.
The theorem itself is proved in Subsection~\ref{sec::Witt::submod}, and in Subsection~\ref{sec::Witt::conj} we discuss its relation to the conjectures of Petukhov--Sierra and Etingof.

\subsection{Noetherianity of tensor products of densities}
\label{sec::Witt::Noether}

\begin{theorem}
	\label{thm::Noether::Witt}
	Let $\bar\lambda=(\lambda_1,\ldots,\lambda_r)\in\Bbbk^r$ be arbitrary. Then the $\witt$-module $\cT^r_{\bar\lambda}=\cT_{\lambda_1}\otimes\ldots\otimes\cT_{\lambda_r}$ has property $\calP_{2r}$ (Definition~\ref{def::property::P}). In particular, $\cT^r_{\bar\lambda}$ and all its submodules (for instance, the modules of symmetric tensors $S^r(\cT_\lambda)$) are Noetherian $\witt$-modules.
\end{theorem}

The proof of Theorem~\ref{thm::Noether::Witt} is the proof of Theorem~\ref{thm::Noether::L_1}, performed at the point $z=1$ instead of the point $z=0$. Below we explain why the point has to be changed, and which changes have to be made in the proof.

\smallskip
\noindent\emph{Why the point $z=0$ cannot be used.} In the proof of Theorem~\ref{thm::Noether::L_1} the point $z=0$ plays three roles. First, $\sL_1$ consists of the vector fields vanishing at $z=0$ to the second order; this guarantees that $u(z)/z$ is a polynomial for $e_u\in\sL_1$, which is used in the factorization $\ell_u=m_{u/z}\circ\Phi$ of Lemma~\ref{lem::Euler} and in the definition of the elements $\Psi(F)$. Second, the submodules $z_i^k\cT^r_{\bar\lambda,\bar\mu}$ are used in Lemma~\ref{lem::normalization} to raise the parameters $\mu_i$; this makes the de Rham differentials of the variables of weight $0$ injective and allows us to replace the weight $1$ by the weight $0$. Third, the restrictions to the hyperplanes $\{z_i=0\}$ are steps of the induction (Lemma~\ref{lem::strata}). None of this works at the point $z=0$ for the modules of Laurent polynomials: the multiplication by $z_i$ is invertible on $\Bbbk[z_i^{\pm1}]$, so that $z_i^k\cT=\cT$ for every such module $\cT$, and for the weight $\lambda=0$ and $\mu\in\bZ$ the operator $z\partial+\mu$ has the nonzero kernel $\Bbbk z^{-\mu}$ on $\Bbbk[z^{\pm1}]$, which cannot be removed. On the other hand, the Laurent polynomials are regular at the point $z=1$, and this point can play the role of the point $z=0$.

\smallskip
\noindent\emph{The Lie algebra and the modules.} Put $t:=z-1$ and consider the Lie subalgebra
$$\fg_1:=\{e_u=u\,\partial_z\mid u\in t^2\Bbbk[z^{\pm1}]\}\subset\witt$$
of the vector fields vanishing at the point $z=1$ to the second order. For $\lambda,\mu,a\in\Bbbk$ let $\cT^\pm_{\lambda,\mu,a}$ be the space $\Bbbk[z^{\pm1}]$ on which $e_u\in\fg_1$ acts by the operator
\begin{equation}
	\label{eq::Laurent::action}
	\ell_u:=u\partial+\lambda u'+\mu\frac uz+a\frac ut ;
\end{equation}
this is a differential operator with Laurent polynomial coefficients, since $t^2$ divides $u$. The operator~\eqref{eq::Laurent::action} is the Lie derivative on the densities $z^\mu t^af(z)(dz)^\lambda$; in particular, $\cT^\pm_{\lambda,0,0}$ is the restriction of the $\witt$-module $\cT_\lambda$ to $\fg_1$. The parameter $a$ plays the role of the parameter $\mu$ of Subsection~\ref{sec::L_1}: for $k\in\bZ_{\geq0}$ the multiplication by $t^k$ identifies $\cT^\pm_{\lambda,\mu,a+k}$ with the submodule $t^k\cT^\pm_{\lambda,\mu,a}\subset\cT^\pm_{\lambda,\mu,a}$. The parameter $\mu$ can be changed by arbitrary integers: the multiplication by $z^n$, $n\in\bZ$, is an isomorphism $\cT^\pm_{\lambda,\mu+n,a}\simeq\cT^\pm_{\lambda,\mu,a}$.
Let $\cV_r^\pm$ be the class of the $\fg_1$-modules
$$\cT=\bigotimes_{i=1}^{r}\cT^\pm_{\lambda_i,\mu_i,a_i}\simeq A:=\Bbbk[z_1^{\pm1},\ldots,z_r^{\pm1}],\qquad\lambda_i,\mu_i,a_i\in\Bbbk,$$
on which $\fg_1$ acts diagonally, as in Definition~\ref{def::class::V}. The degenerate forms are now $t_i:=z_i-1$ and $z_i-z_j$ with $i\neq j$.

\begin{theorem}
	\label{thm::Noether::Laurent}
	The class $\cV_r^\pm$ has property $\calP_{2r}$: every $\fg_1$-submodule $M$ of a module $\cT\in\cV_r^\pm$ is of the form $M=\calF^{2r}_{PBW}U(\fg_1)\cdot S$ for a finite subset $S\subset M$.
\end{theorem}

\begin{proof}[Proof of Theorem~\ref{thm::Noether::Witt}]
	The restriction of $\cT^r_{\bar\lambda}$ to $\fg_1$ is the module $\bigotimes_{i=1}^r\cT^\pm_{\lambda_i,0,0}$ of the class $\cV^\pm_r$. Let $M\subset\cT^r_{\bar\lambda}$ be a $\witt$-submodule. Then $M$ is a $\fg_1$-submodule, and by Theorem~\ref{thm::Noether::Laurent} there is a finite subset $S\subset M$ such that
	$$M=\calF^{2r}_{PBW}U(\fg_1)\cdot S\subset\calF^{2r}_{PBW}U(\witt)\cdot S\subset M .$$
	The last assertion follows, since $S^r(\cT_\lambda)$ is a submodule of $\cT_\lambda^{\otimes r}$.
\end{proof}

\smallskip
\noindent\emph{Proof of Theorem~\ref{thm::Noether::Laurent}: the changes in the proof of Theorem~\ref{thm::Noether::L_1}.} The proof follows Subsections~\ref{sec::L_1}--\ref{sec::L_1::proof} step by step, with $z_i$ replaced by $t_i$ wherever the point $z=0$ was used. Let us list the changes.
\begin{enumerate}[label=(\arabic*)]
	\item \emph{Lemma~\ref{lem::closure}} is a general statement and does not change.
	\item \emph{The strata (Lemma~\ref{lem::strata}).} For the degenerate form $t_i$ put $c_{t_i}(u):=u(z_i)/t_i$; formula~\eqref{eq::strata::twist} holds verbatim. The restriction $g\mapsto g|_{z_i=1}$ identifies $A/t_iA$ with the Laurent polynomials in the remaining variables. For $u\in t^2\Bbbk[z^{\pm1}]$ the functions $u$, $u'$, $u/z$ and $u/t$ vanish at $z=1$, so $\ell^{(i)}_u$ induces zero on $A/t_iA$; moreover, $c_{z_i-z_k}(u)|_{z_i=1}=c_{t_k}(u)$. Hence the restriction to $\{t_i=0\}$ removes the variable $z_i$ and shifts the parameters $a_k$ (instead of $\mu_k$) by nonnegative integers. The restriction to $\{z_i=z_j\}$ adds the parameters $(\lambda,\mu,a)$ of the two variables and shifts the merged weight $\lambda$ by $\nu_{ij}$, as before.
	\item \emph{The de Rham differential (Lemma~\ref{lem::Euler}).} For $\lambda=0$ one has $\ell_u=m_{u/t}\circ\Phi$ with
	$$\Phi:=t\partial+\mu\frac tz+a ;$$
	this is the de Rham differential $d(z^\mu t^ag)=z^\mu t^{a-1}\Phi(g)\,dz$. A direct computation shows that $\Phi\circ m_{u/t}$ is the operator of $e_u$ on $\cT^\pm_{1,\mu,a-1}$, so that $\Phi:\cT^\pm_{0,\mu,a}\to\cT^\pm_{1,\mu,a-1}$ is a homomorphism. The main difference with Lemma~\ref{lem::Euler} is that $\Phi$ is, in general, not bijective. A nonzero element $g$ of its kernel satisfies $g'/g=-\mu/z-a/t$; hence $\ker\Phi=\Bbbk z^{-\mu}t^{-a}\cap\Bbbk[z^{\pm1}]$, which vanishes if $a\notin\bZ_{\leq0}$. Since $\Phi(z^m)=(m+\mu+a)z^m-(m+\mu)z^{m-1}$, the cokernel of $\Phi$ is spanned by the classes $v_m$ of the monomials $z^m$, $m\in\bZ$, subject to the relations $(m+\mu+a)v_m=(m+\mu)v_{m-1}$. Choose $p\leq q$ such that $m+\mu+a\neq0$ for $m>q$ and $m+\mu\neq0$ for $m<p$, and let $V_{p,q}\subset\Bbbk[z^{\pm1}]$ be the span of $z^p,\ldots,z^q$. Comparing the highest and the lowest terms, we see that $\Phi(g)\in V_{p-1,q}$ implies $g\in V_{p,q}$; in particular, $\ker\Phi\subset V_{p,q}$. On the other hand, the relations show that the classes $v_m$ with $p-1\leq m\leq q$ span $\mathrm{coker}\,\Phi$. Hence $\mathrm{coker}\,\Phi\simeq V_{p-1,q}/\Phi(V_{p,q})$, so that $\dim\mathrm{coker}\,\Phi=1+\dim\ker\Phi$; in particular, $\mathrm{coker}\,\Phi$ is one-dimensional if $a\notin\bZ_{\leq0}$. Moreover, $\fg_1$ acts on $\mathrm{coker}\,\Phi$ trivially, because the operators of $\fg_1$ on $\cT^\pm_{1,\mu,a-1}$ have the form $\Phi\circ m_{u/t}$.
	\item \emph{Normalization (Definition~\ref{def::normalized} and Lemma~\ref{lem::normalization}).} A variable is good if $\lambda_i\notin\{0,1\}$, and bad if $\lambda_i=0$ and $a_i\notin\bZ_{\leq0}$; by~(3), the de Rham differentials of the bad variables are injective. The parameters $a_i$ are raised by passing to the submodules $t_i^k\cT$ (the quotients are filtered by strata, as before). The replacement of the weight $1$ by the weight $0$ changes as follows. After raising $a_i$ so that $a_i+1\notin\bZ_{\leq0}$, the map $\Phi_i:\cT^\natural\to \cT$, where $\cT^\natural$ is obtained from $\cT$ by replacing $(\lambda_i,\mu_i,a_i)$ with $(0,\mu_i,a_i+1)$, is an injective homomorphism, but not necessarily surjective. Its cokernel is $\mathrm{coker}(\Phi_i)\otimes(\text{the other tensor factors})$; since $\fg_1$ acts trivially on $\mathrm{coker}(\Phi_i)$, this is a finite direct sum of copies of a module of $\cV^\pm_{r-1}$. By Lemma~\ref{lem::closure}, $\cT$ has property $\calP_p$ if $\cT^\natural$ and all modules of $\cV^\pm_{r-1}$ have it.
	\item \emph{Operators of bounded PBW degree (Subsections~\ref{sec::L_1::operators} and~\ref{sec::L_1::deRham}).} The map $\Psi$ is defined by the same formula on the space $t(x_1)^2\cdots t(x_n)^2\,\Bbbk[x_1^{\pm1},\ldots,x_n^{\pm1}]$, and Lemma~\ref{lem::Psi} holds with $g$ a monomial in the variables $z_i^{-1}$ and $t_i^{-1}$. In Lemma~\ref{lem::onevar} the additional terms $a\,u/t$ carry no derivatives, so that a pair over a variable of weight $\lambda_i$ still contributes $2\lambda_i(1-\lambda_i)$; in Lemma~\ref{lem::singletons} one has $C=F|_{x_a=z_{\iota(a)}}/\prod_at_{\iota(a)}$. In Lemmas~\ref{lem::key::good} and~\ref{lem::key::mult} one takes
	$$F_H:=\prod_{i\in G}(x_i-y_i)^2\cdot\prod_{v}t(v)^2\cdot\prod_{\{v,v'\}}(v-v')^q\cdot H$$
	with a Laurent polynomial $H$ in the representatives, and obtains
	$$h=\prod_{i\in G}t_i^4\prod_{b\in B}t_b\prod_{i<i'}(z_i-z_{i'})^{4q}\prod_{i,b}(z_i-z_b)^{2q}\prod_{b<b'}(z_b-z_{b'})^q.$$
	The ring $R=A^\Gamma$ of $\Gamma$-invariant Laurent polynomials is generated by the elementary symmetric polynomials of the sets of variables $\{z_i\mid i\in G\}$ and $\{z_b\mid b\in B\}$ and by the inverses of the products of the variables of these sets; it is Noetherian, and $A$ is a finitely generated $R$-module.
	\item \emph{Lemma~\ref{lem::support}.} One takes
	$$\delta:=\prod_it_i\cdot\prod_{i<j}(z_i-z_j),\qquad h_0:=\prod_{\gamma\in\Gamma}\gamma(\delta),$$
	and $u_k:=t^2z^k$, $k=0,\ldots,r-1$. The determinant of the corresponding matrix $\bigl(t_i^2z_i^k\bigr)_{k,i}$ is $\prod_it_i^2$ times a Vandermonde determinant, and it is invertible in $A^\circ$. To obtain a nonzero polynomial in $JA^\circ$, one multiplies by a power of $h_0$ and by a monomial in the variables $z_i$.
	\item \emph{Lemma~\ref{cor::generic::rank}.} The functions $\omega(u)$ contain the additional terms $a_i\,u(z_i)/t_i$, and $c_u=\sum_{b\in B}\bigl(u'(z_b)-u(z_b)/t_b\bigr)\in R$. The operator $\Phi_B=\prod_{b\in B}\Phi_b$ is injective on $A$ by~(3), but it is neither surjective nor diagonalizable, so both steps of the proof change. In the first step, let $D$ be the operator of $\xi$ on $\cT'$. Since $\Phi_B$ is a homomorphism, $D\circ\Phi_B=\Phi_B\circ m_h\circ\Phi_B$ on $A$, and hence as differential operators. The ring of differential operators with Laurent polynomial coefficients has no zero divisors, and $\Phi_B\neq0$; hence $D=\Phi_B\circ m_h$, and the rest of the first step does not change. In the second step, let $I_N\subset A$ be the ideal generated by the elements $t_b^N$, $b\in B$. Since $\Phi_b(t_b^Ng)=t_b^N(\Phi_b+N)g$, the operator $\Phi_B$ preserves $I_N$. The quotient $A/I_N$ is the tensor product of the finite-dimensional spaces $\Bbbk[z_b^{\pm1}]/(t_b^N)\simeq\Bbbk[t_b]/(t_b^N)$, $b\in B$, and of the Laurent polynomials in the good variables, and $\Phi_B$ acts on the first factor only. Since $\Phi_b(t_b^n)=(n+a_b)t_b^n+\mu_bt_b^{n+1}/z_b$, this action is triangular in the basis of monomials in the variables $t_b$, with the diagonal entries $\prod_{b\in B}(n_b+a_b)\neq0$, $0\leq n_b<N$. By the Cayley--Hamilton theorem, there is a polynomial $P_N$ with $P_N(0)=0$ such that $P_N(\Phi_B)$ induces the identity on $A/I_N$. Write $P_N(t)=t\,P_{N,1}(t)$. The $R$-module $\bar Q:=A\cap\mathbb{K}Q$ is stable under $\Phi_B$ by the first step, and for $y\in M$ we get $P_N(\Phi_B)y=P_{N,1}(\Phi_B)\bigl(\Phi_B(y)\bigr)\in\bar Q$; hence $y\in\bar Q+I_N$ for all $N$. Let $J\subset R$ be the ideal generated by the elementary symmetric polynomials in the variables $t_b$, $b\in B$. Then $t_b^{|B|}\in JA$, since $\prod_{b'\in B}(t_b-t_{b'})=0$; hence $I_{|B|k}\subset J^kA$, and the image of $y$ in the finitely generated $R$-module $A/\bar Q$ lies in $\bigcap_{k\geq1}J^k(A/\bar Q)$. By Krull's intersection theorem~\cite[Theorem~10.17]{AM}, there is $j\in J$ with $(1+j)y\in\bar Q$. Since $J$ is contained in the ideal of $A$ generated by the elements $t_b$, we have $1+j\neq0$, and hence $y\in\mathbb{K}Q$.
	\item \emph{The final step (Subsection~\ref{sec::L_1::proof})} does not change: one uses the Artin--Rees lemma for the Noetherian ring $R$ and the finitely generated $R$-module $A$, and argues by induction on $r$ within the class $\cV_r^\pm$.
\end{enumerate}
This completes the proof of Theorem~\ref{thm::Noether::Laurent}, and hence of Theorem~\ref{thm::Noether::Witt}.

\subsection{The image of the universal enveloping algebra in differential operators}
\label{sec::Witt::U}

Let $\bar\lambda\in\Bbbk^r$. We identify $\cT^r_{\bar\lambda}$ with the ring $A:=\Bbbk[z_1^{\pm1},\ldots,z_r^{\pm1}]$ of Laurent polynomials, as in Subsection~\ref{sec::Witt::Noether}; then $e_u\in\witt$, $u\in\Bbbk[z^{\pm1}]$, acts by $L_u=\sum_{i=1}^r\ell_u^{(i)}$ with $\ell_u^{(i)}:=u(z_i)\partial_i+\lambda_iu'(z_i)$, and every element of $U(\witt)$ acts by a differential operator with Laurent polynomial coefficients. For a Laurent polynomial $F\in\Bbbk[x_1^{\pm1},\ldots,x_n^{\pm1}]$ let $\Psi(F)\in U(\witt)$ be the linear extension of $u_1(x_1)\cdots u_n(x_n)\mapsto e_{u_1}\cdots e_{u_n}$, $u_a\in\Bbbk[z^{\pm1}]$; here the full Witt algebra is used, since all exponents are allowed. Lemmas~\ref{lem::Psi} and~\ref{lem::onevar} and their proofs hold verbatim in this setting (with $\bar\mu=0$).

From now on we fix a weight $\lambda\in\Bbbk$ and restrict to the weights $\bar\lambda=(\lambda,\ldots,\lambda)$. Then $\cT^r_{\bar\lambda}=\cT_\lambda^{\otimes r}\simeq A\,\partial^{-\bar\lambda}$ and $\ell^{(i)}_u=u(z_i)\partial_i+\lambda u'(z_i)$. The symmetric group permutes the variables and commutes with the action of $U(\witt)$. Since the characteristic is zero, symmetrization identifies the module of symmetric tensors $S^r(\cT_\lambda)$ with the invariants $R\,\partial^{-\bar\lambda}$, where $R:=A^{S_r}$ is the ring of symmetric Laurent polynomials; this is the Laurent analogue of the ring $R$ of Lemma~\ref{lem::key::good}. We omit the symbol $\partial^{-\bar\lambda}$ and regard the elements of $S^r(\cT_\lambda)$ as symmetric Laurent polynomials, on which $e_u$ acts by the operator $L_u$. For $\lambda=-1$ we obtain the symmetric powers $S^r(\witt)$ of the adjoint representation, and $L_u=\sum_i(u(z_i)\partial_i-u'(z_i))$. Let $\cA_{r,\lambda}\subset\mathrm{End}_\Bbbk(R)$ be the image of $U(\witt)$ acting on $S^r(\cT_\lambda)$. As in Section~\ref{sec::Noether::T_lambda}, we use the following notation:
\begin{itemize}[itemsep=0pt, topsep=0pt]
	\item $\Delta:=\prod_{i<j}(z_i-z_j)$; note that $\Delta^2\in R$.
	\item $L_u=X_u+m_{\omega(u)}$, where $X_u:=\sum_{i=1}^ru(z_i)\partial_i$ and $\omega(u):=\lambda\sum_{i=1}^ru'(z_i)\in R$, see~\eqref{eq::L_u::Leibniz}; here $m_f$ is the operator of multiplication by $f$.
	\item $c_\Delta(u):=\sum_{i<j}c_{z_i-z_j}(u)=\sum_{i<j}\dfrac{u(z_i)-u(z_j)}{z_i-z_j}\in R$, in the notation of~\eqref{eq::c_ell}. Then $X_u(\Delta^2)=2c_\Delta(u)\Delta^2$.
\end{itemize}

\begin{proposition}
	\label{prp::Witt::mult}
	Let $\lambda\neq0,1$, and put $K:=4(r+1)$ and $h:=\Delta^{2K}$. For every $f\in R$ the operator $m_{hf}$ of multiplication by $hf$ on $S^r(\cT_\lambda)$ is the action of an element of $\calF^{2r}_{PBW}U(\witt)$; in particular, $m_{hf}\in\cA_{r,\lambda}$.
\end{proposition}
\begin{proof}
	This is the proof of Lemma~\ref{lem::key::good} with all weights equal to $\lambda$, in which the factor $\prod_vv^2$ is omitted: it was needed only to obtain elements of $U(\sL_1)$, while now $\Psi$ is defined for all Laurent polynomials. Namely, let $q:=2(r+1)$, consider the clusters $\{x_i,y_i\}$, $i=1,\ldots,r$, where $x_i$ precedes $y_i$, and for $f\in R$ put
	$$F_H:=\prod_{i=1}^r(x_i-y_i)^2\cdot\prod_{\{v,v'\}}(v-v')^q\cdot H,\qquad H:=\frac{1}{r!}\,f(x_1,\ldots,x_r),$$
	where $\{v,v'\}$ runs over the pairs of auxiliary variables from different clusters. Since $q>2r$, by Lemma~\ref{lem::Psi}(c) only the maps $\iota_\gamma:=\gamma\circ\iota_0$, $\gamma\in S_r$, contribute, where $\iota_0(x_i)=\iota_0(y_i)=i$. By Lemma~\ref{lem::onevar}, each of them contributes $\kappa\cdot m_{hf/r!}$, where $\kappa:=\bigl(2\lambda(1-\lambda)\bigr)^r$: indeed, $f$ is symmetric, the four pairs of auxiliary variables between the clusters of $i\neq i'$ contribute $(z_{\gamma(i)}-z_{\gamma(i')})^{4q}$, and $\Delta^{4q}=\Delta^{2K}=h$. Hence $\Psi(F_H)\in\calF^{2r}_{PBW}U(\witt)$ acts on $S^r(\cT_\lambda)$ by $\kappa\cdot m_{hf}$, and $\kappa\neq0$, since $\lambda\neq0,1$.
\end{proof}

\subsection{Submodules of $S^r(\cT_\lambda)$}
\label{sec::Witt::submod}

In this subsection, and hence in the next one, we assume that $\Bbbk$ is algebraically closed (the general case of characteristic zero follows by extension of scalars). We use the notation of Subsection~\ref{sec::Witt::U}; in particular, $S^r(\cT_\lambda)$ is identified with the ring $R$ of symmetric Laurent polynomials, so that $\Delta^{2d}S^r(\cT_\lambda)=\Delta^{2d}R$. Our main result about the Witt algebra is the following theorem.

\begin{theorem}
	\label{thm::Witt::submod}
	Let $\lambda\in\Bbbk$, $\lambda\neq0,1$, and let $M\subset S^r(\cT_\lambda)$ be a nonzero $\witt$-submodule. Then there exists $d$ (depending on $M$) such that $\Delta^{2d}S^r(\cT_\lambda)\subset M$. In particular, for $\lambda=-1$ every nonzero submodule of $S^r(\witt)$ contains $\Delta^{2d}S^r(\witt)$ for some $d$.
\end{theorem}

\begin{remark}
	\label{rmk::Witt::Delta}
	As in the proof of Lemma~\ref{lem::strata} (iterate~\eqref{eq::strata::twist} over the factors of $D=\Delta^{2d}$), each $\Delta^{2d}S^r(\cT_\lambda)=\Delta^{2d}R$ is itself a submodule of $S^r(\cT_\lambda)$ (for any $\lambda$): for $f\in R$,
	$$
	L_u(\Delta^{2d}f)=\Delta^{2d}\bigl(L_uf+2d\,c_\Delta(u)f\bigr).
	$$
	Thus $\Delta^{2d}S^r(\cT_\lambda)$ is isomorphic to the module of symmetric elements of the Laurent version of the twisted module $\cT^r_{\bar\lambda,0;\bar\nu}$ of~\eqref{eq::L::action::3} with all $\nu_{ij}=2d$.
	Hence Theorem~\ref{thm::Witt::submod} says that every nonzero submodule of $S^r(\cT_\lambda)$ contains one of the submodules $\Delta^{2d}S^r(\cT_\lambda)$.
\end{remark}

To go to the proof we need the following technical definition and Lemma~\ref{lem::Witt::ideal}. Let $R_\Delta:=R[\Delta^{-2}]$. To each submodule $M\subset S^r(\cT_\lambda)=R$ we associate the subset
$$
I(M):=\{f\in R_\Delta\ \mid\ \exists d_0\ \forall d\geq d_0:\ \Delta^{2d}f\in M\}
$$
(the membership $\Delta^{2d}f\in M$ includes the requirement $\Delta^{2d}f\in R$).

\begin{lemma}
	\label{lem::Witt::ideal}
	Let $\lambda\neq0,1$ and let $M\subset S^r(\cT_\lambda)$ be a $\witt$-submodule. Then
	\begin{enumerate}[label=(\alph*), itemsep=0pt, topsep=0pt]
		\item $M\subset I(M)$; in particular $I(M)\neq0$ if $M\neq0$;
		\item $I(M)$ is an ideal of $R_\Delta$;
		\item $X_u(I(M))\subset I(M)$ for every $u\in\Bbbk[z^{\pm1}]$.
	\end{enumerate}
\end{lemma}
\begin{proof}
	Write $I=I(M)$ and let $K=4(r+1)$ and $h=\Delta^{2K}$ be as in Proposition~\ref{prp::Witt::mult}.

	(a) Let $g\in M$. For $d\geq K$ we have $\Delta^{2d}g=m_{h\Delta^{2(d-K)}}(g)\in M$ by Proposition~\ref{prp::Witt::mult} (this is the place where $\lambda\neq0,1$ is used), applied to $f=\Delta^{2(d-K)}\in R$.

	(b) Closure under addition is clear: take the larger of two bounds $d_0$. Now let $f\in I$ with a bound $d_0$ and let $g=g_0\Delta^{-2a}$ with $g_0\in R$ be an arbitrary element of $R_\Delta$. For $d\geq d_0+K+a$,
	$$
	\Delta^{2d}gf=m_{hg_0}\bigl(\Delta^{2(d-K-a)}f\bigr)\in M,
	$$
	since $\Delta^{2(d-K-a)}f\in M$ and $m_{hg_0}\in\cA_{r,\lambda}$ by Proposition~\ref{prp::Witt::mult}.

	(c) The derivation $X_u$ extends to $R_\Delta$ and preserves $S_r$-invariants. Let $f\in I$. Using $X_u(\Delta^{2d})=2d\,c_\Delta(u)\Delta^{2d}$ and $L_u=X_u+m_{\omega(u)}$ we get
	$$
	\Delta^{2d}X_uf=L_u\bigl(\Delta^{2d}f\bigr)-2d\,\Delta^{2d}c_\Delta(u)f-\Delta^{2d}\omega(u)f .
	$$
	By (b), the elements $c_\Delta(u)f$ and $\omega(u)f$ belong to $I$. Choose $d_1$ larger than the bounds for $f$, $c_\Delta(u)f$ and $\omega(u)f$. For $d\geq d_1$ all three terms on the right-hand side lie in $M$, the first one because $M$ is a submodule. Hence $X_uf\in I$.
\end{proof}

\begin{proof}[Proof of Theorem~\ref{thm::Witt::submod}]
	Let $M\neq0$ and $I=I(M)$. By Lemma~\ref{lem::Witt::ideal}, $I$ is a nonzero ideal of $R_\Delta$ stable under all derivations $X_u$. We show that $I=R_\Delta$ in three steps; the argument is the same as in the proof of Lemma~\ref{lem::support}.

	First, let us pass to the configuration space. Let $\Conf_r:=\{z\in(\Bbbk^*)^r\mid z_i\neq z_j\text{ for }i\neq j\}$ with the algebra of functions $\calO(\Conf_r)=A[\Delta^{-1}]$. Then $\calO(\Conf_r)^{S_r}=R_\Delta$: an invariant element can be written as $f/\Delta^{2a}$, and then $f$ is invariant. Let $\tilde I:=I\cdot\calO(\Conf_r)$, a nonzero ideal of $\calO(\Conf_r)$. The matrix $W=(z_i^k)_{0\leq k\leq r-1,\,1\leq i\leq r}$ has determinant $\pm\Delta$, which is invertible in $\calO(\Conf_r)$. Since $X_{z^k}=\sum_iW_{ki}\partial_i$, we get
	$$
	\partial_i=\sum_{k=0}^{r-1}(W^{-1})_{ik}X_{z^k},\qquad (W^{-1})_{ik}\in\calO(\Conf_r).
	$$
	Hence for $a_j\in\calO(\Conf_r)$ and $f_j\in I$
	$$
	\partial_i\Bigl(\sum_ja_jf_j\Bigr)=\sum_j(\partial_ia_j)f_j+\sum_{j,k}a_j(W^{-1})_{ik}X_{z^k}(f_j)\in\tilde I,
	$$
	that is, $\tilde I$ is stable under all partial derivatives $\partial_i$.

	Second, let us show that $\tilde I=\calO(\Conf_r)$. This is the statement that $\calO(\Conf_r)$ is a simple $\sDiff$-module. Suppose that $\tilde I$ is proper. Since $\calO(\Conf_r)$ is a finitely generated algebra over the algebraically closed field $\Bbbk$, by the Nullstellensatz there is a point $p\in\Conf_r$ at which all elements of $\tilde I$ vanish. Every nonzero function $f\in\calO(\Conf_r)$ has a nonzero Taylor expansion at $p$ in the local coordinates $z_i-p_i$; denote by $\mathrm{ord}_pf$ the degree of its lowest nonzero homogeneous component. Choose $f\in\tilde I\setminus\{0\}$ with the minimal value $m:=\mathrm{ord}_pf$. Then $m\geq1$, and if $f_m$ is the lowest homogeneous component of $f$, then Euler's formula $\sum_i(z_i-p_i)\partial_if_m=mf_m\neq0$ shows that $\partial_if_m\neq0$ for some $i$. Then $\partial_if\in\tilde I$ has order $m-1$, a contradiction.

	Third, let us return to $R_\Delta$. Since $\tilde I=\calO(\Conf_r)$, we can write $1=\sum_ja_jf_j$ with $a_j\in\calO(\Conf_r)$ and $f_j\in I$. Apply the Reynolds operator $\rho:=\frac{1}{r!}\sum_{\gamma\in S_r}\gamma$. Since the $f_j$ are $S_r$-invariant,
	$$
	1=\sum_j\rho(a_j)f_j,\qquad\rho(a_j)\in\calO(\Conf_r)^{S_r}=R_\Delta.
	$$
	Since $I$ is an ideal of $R_\Delta$, we get $1\in I$.

	Finally, since $1\in I$, by definition of $I$ there exists $d_0$ with $\Delta^{2d_0}\in M$. Now for any $f\in R$, Proposition~\ref{prp::Witt::mult} gives
	$$
	\Delta^{2(d_0+K)}f=m_{hf}\bigl(\Delta^{2d_0}\bigr)\in M,
	$$
	where $K=4(r+1)$ and $h=\Delta^{2K}$ are as in Proposition~\ref{prp::Witt::mult}.
	Hence $\Delta^{2d}S^r(\cT_\lambda)=\Delta^{2d}R\subset M$ for $d=d_0+K$.
\end{proof}

\section{Applications}
\label{sec::applications}

\subsection{Gelfand's conjecture}
\label{sec::Gelfand}
In this section, we prove the conjecture of I.\,M. Gelfand on finite-dimensionality of the (co)homology of the subalgebras $\sL_d(n)$ of $\cW_n$, stated in~\cite{Gelfand_Nice}. We state this conjecture in a more general setting: instead of $\sL_d(n)$ one may take any Lie subalgebra of finite codimension, and the coefficients may be taken in a fairly general class of modules as well.

\begin{theorem}
	\label{thm::finite::homology}	
	For any Lie subalgebra $\fg\subset \cW_n^{\pol}$ of finite codimension and any $\fg$-submodule $M\subset\cT_{\bar\lambda}^{r}$, the Lie algebra homology of $\fg$ with coefficients in $M$ is finite-dimensional in every homological degree:
	$$
	\forall p\geq 0 \ \ \dim H_p(\fg; M) <\infty. 
	$$
\end{theorem}
\begin{proof}
	Let us examine more closely the homological Chevalley--Eilenberg complex 
	$C_{\ldot}^{CE}(\fg;M)$, which computes the Lie algebra homology.
	The space of $p$-chains of this complex 
	is the following $\fg$-module:
	$$C_p(\fg;M) := \Lambda^p \fg \otimes M \subset \Lambda^p \cW_n^{\pol}\otimes M \subset \otimes^p \cW_n^{\pol}\otimes M\subset \otimes^p\cW_n^{\pol}\otimes\cT^r_{\bar\lambda}.$$
	Recall that $\cW_n^{\pol}\simeq\cT_{\epsilon}$ for an appropriate weight $\epsilon\in\bZ^n$.
	Thus, thanks to Corollary~\ref{cor::g::noether}, for every $p$ the space of $p$-chains $C_p(\fg;M)$ is a Noetherian $\fg$-module (being a submodule of $\cT_\epsilon^{\otimes p}\otimes\cT^r_{\bar\lambda}$, itself Noetherian since $\fg$ has finite codimension).
	The Chevalley--Eilenberg differential is $\fg$-equivariant, and it is a standard fact that the resulting $\fg$-module structure on the homology $H_p(\fg;M)$ is trivial (see e.g.~\cite{Fuchs}).
	As a subquotient of the Noetherian module $C_p(\fg;M)$, $H_p(\fg;M)$ is itself a Noetherian $\fg$-module; being trivial, it is thus a finitely generated trivial $\fg$-module.
	The number of generators needed for a trivial module coincides with its dimension as a vector space, since the higher-order action of $U(\fg)$ on it vanishes; hence a finitely generated trivial module is automatically finite-dimensional.
\end{proof}

Let us return to the Lie algebra of formal vector fields $\cW_n$.
This is an infinite-dimensional Lie algebra equipped with a natural topology, in which a basis of open neighborhoods of $0$ is given by the subalgebras $\sL_d(n)$.
The 1970s and 1980s were the golden age of Gelfand--Fuchs cohomology, during which the Lie algebra cohomology of many infinite-dimensional Lie algebras was computed.
It was postulated that the correct cohomology theory for topological Lie algebras of this kind is \emph{continuous} cohomology (see e.g.~\cite{Guillemin}, \cite{Guichardet}).
Nowadays there are several conceptual explanations of why one should work with the continuous Chevalley--Eilenberg complex rather than the ordinary one (see e.g.~\cite{Beilinson}, \cite{Positselski}); in particular, these explanations make it seem more natural to instead compute directly the homology of the corresponding discrete Lie algebra.
It is worth mentioning that the ordinary (i.e.\ non-continuous) Chevalley--Eilenberg cohomology of the topological Lie algebras $\cW_n$, $\sL_d(n)$, and of the Lie algebra of smooth vector fields on a manifold, can be infinite-dimensional, and we are still far from understanding it.
We will not give a detailed definition of continuous cohomology here, and refer the interested reader to the papers cited above. Roughly speaking, the space of continuous $k$-chains $C^k_{c}(\fg;M)$ consists of continuous (with respect to the given topology) skew-symmetric $k$-multilinear maps $\fg\times\ldots\times\fg\to M$.
In particular, the complex $C^{\udot}_c(\sL_d(n);\bR)$ is graded linear dual to the homological Chevalley--Eilenberg complex $C_{\ldot}(\sL_d^{\pol}(n);\bR)$.\footnote{By the graded linear dual of a graded vector space with finite-dimensional graded components, we mean the graded vector space whose graded components are the linear duals of the original ones.}
Therefore, finite-dimensionality of the continuous cohomology of $\sL_d(n)$ is equivalent to finite-dimensionality of the (ordinary) homology of $\sL_d^{\pol}(n)$.

The following Corollary~\ref{cor::Hk::Ld} was stated as a conjecture by I.\,M. Gelfand in his plenary lecture at the ICM in 1970 (see p.~106 of~\cite{Gelfand_Nice}). 
\begin{corollary}
	\label{cor::Hk::Ld}	
	For all $k$, $d$, and $n$, the continuous cohomology $H_c^k(\sL_d(n);\Bbbk)$ is finite-dimensional.
\end{corollary}
\begin{proof}	
	The complex computing the continuous cohomology of $\sL_d(n)$ is the graded linear dual of the (ordinary) Chevalley--Eilenberg complex of $\sL_d^{\pol}(n)$, so that $H_c^k(\sL_d(n);\Bbbk)$ is the linear dual of $H_k(\sL_d^{\pol}(n);\Bbbk)$. By Theorem~\ref{thm::finite::homology} (applied with $M=\Bbbk$), the latter is finite-dimensional for every $k$, and hence so is its dual $H_c^k(\sL_d(n);\Bbbk)$.
\end{proof}

There are not many results on the cohomology of $\sL_d(n)$. The homology of $\sL_1(1)$ was first computed by Goncharova in~\cite{Goncharova}. Later, another proof was given by Gelfand, Feigin, and Fuchs in~\cite{Gelfand_Feigin_Fuks::Laplace}.
In~\cite{Feigin_Fuchs}, Feigin and Fuchs showed how to compute the (co)homology of $\sL_1(n)$ up to degree $2n-1$, and, in particular, found a finite set of generators and relations for $\sL_1(n)$, for every $n$.

\begin{corollary}
	For any graded Lie subalgebra of finite codimension of either of the following two infinite-dimensional Lie subalgebras of polynomial vector fields, the homology with trivial coefficients is finite-dimensional in each homological degree:
	\begin{itemize}
		\item 
		$\cW_n^{\pol}\ltimes\bigl(\fg\otimes \Bbbk[z_1,\ldots,z_n]\bigr)$ -- the subalgebra preserving a fibration with structure group $G$, whose Lie algebra is $\fg$;
		\item $\cW_n^{\pol}\ltimes\bigl(\cW_m^{\pol}\otimes\Bbbk[z_1,\ldots,z_n]\bigr)$ -- the subalgebra preserving a foliation of codimension $n$ (with $m$-dimensional leaves). 	
	\end{itemize}
	The same finite-dimensionality holds for the continuous cohomology of any graded subalgebra of finite codimension of the corresponding Lie subalgebras of formal vector fields:
	$$\cW_n\ltimes\bigl(\fg\otimes \Bbbk[[z_1,\ldots,z_n]]\bigr) \text{ or }\cW_n\ltimes\bigl(\cW_m\otimes\Bbbk[[z_1,\ldots,z_n]]\bigr).$$
\end{corollary}
\begin{proof}
	The proof is entirely parallel to that of Theorem~\ref{thm::finite::homology}.
	For the first family, the Lie algebra $L:=\cW_n^{\pol}\ltimes\bigl(\fg\otimes \Bbbk[z_1,\ldots,z_n]\bigr)$ contains the Lie subalgebra $\cW_n^{\pol}$, and $L$ is itself a $\cW_n^{\pol}$-module isomorphic to a finite direct sum of tensor field modules, $\cT_\epsilon\oplus \cT_0^{\oplus\dim\fg}$, which is Noetherian as an $\sL_d^{\pol}(n)$-module by Corollary~\ref{cor::g::noether}.
	A graded subalgebra of finite codimension in $L$ contains $\sL_d^{\pol}(n)$ for $d\gg0$. Therefore, the Chevalley--Eilenberg complex of any graded finite-codimension subalgebra of $L$ consists of Noetherian $\sL_d^{\pol}(n)$-modules, and thus its homology is finite-dimensional in each degree, exactly as in the proof of Theorem~\ref{thm::finite::homology}.

	For the second family this argument does not apply directly: since $\cW_m^{\pol}$ is infinite-dimensional, the Lie algebra $L:=\cW_n^{\pol}\ltimes\bigl(\cW_m^{\pol}\otimes\Bbbk[z_1,\ldots,z_n]\bigr)$ is an infinite direct sum of copies of $\cT_0$ as a $\cW_n^{\pol}$-module. Instead, we use the proof of Theorem~\ref{thm::Noether} in all $n+m$ variables. Let $y_1,\ldots,y_m$ be the coordinates along the leaves, so that $L$ consists of the vector fields $\sum_if_i(z)\partial_{z_i}+\sum_jg_j(z,y)\partial_{y_j}$ with polynomial coefficients. The coordinate copies
	$$\sL_1^{(z_i)}:=\{u(z_i)\partial_{z_i}\mid u\in t^2\Bbbk[t]\},\qquad \sL_1^{(y_j)}:=\{v(y_j)\partial_{y_j}\mid v\in t^2\Bbbk[t]\}$$
	of $\sL_1$ pairwise commute, so that $L$ contains the diagonal Lie subalgebra
	$$\fh:=\sL_1^{(z_1)}\oplus\ldots\oplus\sL_1^{(z_n)}\oplus\sL_1^{(y_1)}\oplus\ldots\oplus\sL_1^{(y_m)}\simeq\sL_1^{\oplus(n+m)}.$$
	Since $[v(y_j)\partial_{y_j},f(z)\partial_{z_i}]=0$, as an $\fh$-module
	$$L\simeq\cT_\epsilon(z)\oplus\cT_{\epsilon'}(y)\otimes\cT_0(z),$$
	where $\cT_\epsilon(z)=\cW_n^{\pol}$ and $\cT_{\epsilon'}(y)=\cW_m^{\pol}$ are the modules of vector fields in the variables $z$ and $y$, respectively, and $\cT_0(z)=\Bbbk[z_1,\ldots,z_n]$; the copies $\sL_1^{(y_j)}$ act trivially on $\cT_\epsilon(z)$ and on $\cT_0(z)$, and the copies $\sL_1^{(z_i)}$ act trivially on $\cT_{\epsilon'}(y)$. By~\eqref{eq::T::m}, both summands, and hence all tensor powers $L^{\otimes p}$, are finite direct sums of modules of Lemma~\ref{lem::Noether::parts} with $n+m$ parts, one for each coordinate (some of the parts may be empty). Let $\fg'\subset L$ be a graded subalgebra of finite codimension. It contains all homogeneous components of $L$ of sufficiently large degree, and hence, for $D\gg0$, the Lie algebra $\imath_D(\fh)$ obtained by applying $\imath_D$ to every copy of $\sL_1$. As in the proof of Theorem~\ref{thm::Noether}, by Proposition~\ref{prp::Ld-L1}, Lemma~\ref{lem::Noether::parts}(a) and Lemma~\ref{lem::closure}(c), the module $L^{\otimes p}$ has property $\calP_{2p(n+m)}$ as an $\imath_D(\fh)$-module. By Lemma~\ref{lem::closure}(b), (d) and~(a), its $\fg'$-submodule $\Lambda^p\fg'$ of $p$-chains is a Noetherian $\fg'$-module. Thus the Chevalley--Eilenberg complex of $\fg'$ consists of Noetherian $\fg'$-modules, and its homology is finite-dimensional in each degree, exactly as in the proof of Theorem~\ref{thm::finite::homology}.

	Finally, the complex of continuous cochains of a graded subalgebra of finite codimension in the formal versions of these Lie algebras is the graded dual of the Chevalley--Eilenberg complex of its polynomial part, as in the proof of Corollary~\ref{cor::Hk::Ld}.
\end{proof}
Note that the Lie algebra homology of the current algebra $\fg[z_1,\ldots,z_n]:=\fg\otimes\Bbbk[z_1,\ldots,z_n]$ need not be finite-dimensional, even when $\fg$ is semisimple.

\subsection{Around Specht's problem}
\label{sec::Specht}
In this section we consider a finiteness problem of Specht type for commutative polynomials. Recall that Specht's problem asks whether every $T$-ideal of the free associative algebra, that is, every ideal invariant under all endomorphisms of the free algebra, is finitely generated as a $T$-ideal. Over a field of characteristic zero it was solved by Kemer~\cite{Kemer1987}; see also~\cite{Kanel-Kemer,belov-rowenbook,Kemer} and the survey~\cite{Grishin} of the related theory of $T$-spaces.

With each polynomial $p(x)\in\Bbbk[x]$ satisfying $p(0)=0$ we associate a substitution operator on the ring of polynomials in $r$ variables:
$$
\circ_p: f(x_1,\ldots,x_r) \mapsto f(p(x_1),\ldots,p(x_r)). 
$$
The following Corollary~\ref{cor::Specht} is a commutative analogue of the Specht property: it gives the finite basis property for the subspaces of $\Bbbk[x_1,\ldots,x_r]$ invariant under all substitutions $\circ_p$.
\begin{corollary}
	\label{cor::Specht}	
	Suppose a vector space $M\subset\Bbbk[x_1,\ldots,x_r]$ is closed under the substitutions $\circ_p$ for all $p\in\Bbbk[x]$ with $p(0)=0$. Then:
	\begin{enumerate}[label=(\arabic*), itemsep=0pt, topsep=0pt]
		\item The vector space $M$ is graded, $M=\oplus_{d=0}^{\infty} M_d$, by the degree of polynomials;
		\item $M$ is generated by a finite subset $S\subset M$ under the operations $\circ_p$ (for $p(0)=0$) and linear combinations;
		\item The Hilbert series $F_M(t):=\sum_{d=0}^{\infty} \dim M_d \, t^d$ is a rational function whose denominator is of the form $\prod_{i=1}^m(1-t^i)$ for some $m$; in fact, one can take $m=r$.
	\end{enumerate} 
\end{corollary} 
\begin{proof}
	Let $f\in M$ be a (possibly nonhomogeneous) polynomial, and write $f=f_{n_1}+\ldots+f_{n_k}$ for its decomposition into homogeneous components of degrees $n_1,\ldots,n_k$, respectively.
	Applying the substitution $\circ_{\alpha x}$, for $\alpha\in\Bbbk$, gives
	$$
	f(\alpha x_1,\ldots,\alpha x_r)=\alpha^{n_1} f_{n_1}(x_1,\ldots,x_r)+ \ldots + \alpha^{n_k} f_{n_k}(x_1,\ldots,x_r).$$
	Choosing pairwise distinct scalars $\alpha_1,\ldots,\alpha_k\in\Bbbk$ generically, so that the associated (generalized) Vandermonde matrix $(\alpha_i^{n_j})_{i,j=1}^{k}$ is invertible — possible since $\Bbbk$ is infinite (having characteristic $0$) and non-vanishing of $\det(\alpha_i^{n_j})$ is a nontrivial polynomial condition on the $\alpha_i$ — we conclude that every homogeneous component $f_{n_i}$ belongs to $M$.
	
	Similarly, consider the one-parameter family of substitutions $\circ_{p_\varepsilon}$, where $p_\varepsilon(x):=x+\varepsilon q(x)$ for a fixed $q\in\Bbbk[x]$ with $q(0)=0$. For $f\in M$ the polynomial $f(p_\varepsilon(x_1),\ldots,p_\varepsilon(x_r))=\sum_{k\geq0}\varepsilon^kf_k$ depends polynomially on $\varepsilon$ and belongs to $M$ for every $\varepsilon\in\Bbbk$. By the Vandermonde argument above, applied to finitely many values of $\varepsilon$, all its coefficients $f_k$ belong to $M$; in particular, $M$ contains $f_1=\bigl(q(x_1)\partial_1+\ldots+q(x_r)\partial_r\bigr)f$. As $q$ ranges over the polynomials with $q(0)=0$, the diagonal vector fields $q(x_1)\partial_1+\ldots+q(x_r)\partial_r$ form the diagonal copy of $\sL_0\supset\sL_1$ in $\cW_r^{\pol}$ (in the notation of the earlier sections, the image of the diagonal embedding $\alpha_1+\ldots+\alpha_r$). Therefore, $M$ is an $\sL_0$-submodule, in particular an $\sL_1$-submodule, of $\Bbbk[x_1,\ldots,x_r]=\cT_0^{\otimes r}$, graded compatibly with the total-degree grading of Remark~\ref{rmk::grading}.
	
	The module $\cT_0^{\otimes r}$ is Noetherian by Theorem~\ref{thm::Noether::L_1}, which gives (1) and (2). Part (3), with $m=r$, follows from Theorem~\ref{thm::Hilb}.
\end{proof}

\subsection{The conjectures of Petukhov--Sierra and Etingof on $U(\vir)/(C-c)$}
\label{sec::Witt::conj}

For $c\in\Bbbk$ consider the quotient $U_c:=U(\vir)/(C-c)$ of the universal enveloping algebra of the Virasoro algebra. It has the PBW filtration $\calF_{\bullet}U_c$, where $\calF_rU_c$ is spanned by the products of at most $r$ elements of $\witt$ (recall that $C$ acts by the scalar $c$), and $\gr U_c\simeq S(\witt)$.
The algebras $U(\witt)$ and $U(\vir)$ are not Noetherian (see~\cite{SierraWalton}). Nevertheless, Petukhov and Sierra conjectured that the quotients $U_c$ with nonzero central charge are simple, see~\cite[slides~26, 29 and~33]{Sierra::Madrid2023}:
\begin{conjecture}[Petukhov--Sierra]
	\label{conj::PS}
	For every $c\neq0$ the algebra $U_c=U(\vir)/(C-c)$ is simple.
\end{conjecture}
For $c=0$ this is false, since $U_0=U(\witt)$ has the augmentation ideal.

There is a Poisson analogue of Conjecture~\ref{conj::PS}. Let $P_c:=S(\vir)/(C-c)$ be the symmetric algebra $S(\witt)$ with the Poisson bracket determined by $\{x,y\}=[x,y]+c\,\psi(x,y)$ for $x,y\in\witt$ (the term $c\,\psi(x,y)$ is a constant, so the bracket is not homogeneous). Petukhov and Sierra proved in~\cite{PetukhovSierraPoisson} that $P_c$ is Poisson simple if and only if $c\neq0$. Conjecture~\ref{conj::PS} is the quantum version of this statement.

Both statements can be analyzed through the adjoint action of $\witt$ on $U_c$ (resp. on $P_c$) and through the following computation in PBW coordinates. For $x\in\witt$ denote by $\iota_x:S^r(\witt)\to S^{r-1}(\witt)$ the derivation determined by $\iota_x(w)=\psi(x,w)$ for $w\in\witt$. Let $\sym:S(\witt)\to U_c$ be the symmetrization map, which is a linear isomorphism by the PBW theorem.

\begin{lemma}
	\label{lem::Witt::ad}
	For $x\in\witt$ and $a\in S^r(\witt)$ we have in $U_c$
	$$
	[x,\sym(a)]=\sym(x\cdot a)+c\,\sym(\iota_xa).
	$$
	The same formula describes the action $\{x,\cdot\}$ on $P_c\simeq S(\witt)$.
\end{lemma}
\begin{proof}
	Let $a=w_1\cdots w_r$. Then $[x,\sym(a)]=\frac{1}{r!}\sum_{\sigma}\sum_{i}w_{\sigma(1)}\cdots[x,w_{\sigma(i)}]\cdots w_{\sigma(r)}$. In $U_c$ we have $[x,w]=[x,w]_{\witt}+c\,\psi(x,w)$. The summands $[x,w]_{\witt}$ add up to $\sym(x\cdot a)$. In the scalar summands, averaging over $\sigma$ of the remaining ordered products gives $c\sum_j\psi(x,w_j)\sym\bigl(\prod_{k\neq j}w_k\bigr)=c\,\sym(\iota_xa)$. For $P_c$ the same computation applies, since $\{x,\cdot\}$ is a derivation.
\end{proof}

\begin{corollary}
	\label{cor::Witt::ext}
	The adjoint action of $\witt$ on $\calF_rU_c/\calF_{r-2}U_c$ (and the action $\{x,\cdot\}$ on the corresponding quotient of the filtered Poisson algebra $P_c$) is isomorphic to the module
	$$
	\cE_r^{c}:=S^r(\witt)\oplus S^{r-1}(\witt),\qquad x\cdot(a,b)=(x\cdot a,\ x\cdot b+c\,\iota_xa).
	$$
	It is an extension $0\to S^{r-1}(\witt)\to\cE_r^c\xrightarrow{\ \pi\ }S^r(\witt)\to0$, where $S^{r-1}(\witt)\simeq\calF_{r-1}U_c/\calF_{r-2}U_c$ and $\pi$ is induced by the projection $\calF_rU_c/\calF_{r-2}U_c\to\calF_rU_c/\calF_{r-1}U_c=S^r(\witt)$. For $c\neq0$ the rescaling $b\mapsto b/c$ identifies $\cE_r^c$ with $\cE_r:=\cE^1_r$, and for $c=0$ the extension splits.
\end{corollary}

In the function model $S^r(\witt)\simeq R$ of Subsection~\ref{sec::Witt::U}, in which the monomial $w_1\cdots w_r$ with $w_i=f_i(z)\partial_z$ corresponds to the symmetric Laurent polynomial
$$\sum_{\sigma\in S_r}f_1(z_{\sigma(1)})\cdots f_r(z_{\sigma(r)}),$$ the map $\iota_{e_u}$ becomes the contraction $a\mapsto\frac{1}{12}\Res_{w=0}u'''(w)\,a(z_1,\ldots,z_{r-1},w)\,dw$ with the Gelfand--Fuchs cocycle $u\mapsto u'''dz^2$; for the symmetrization normalized by the factor $\frac1{r!}$ the contraction has to be multiplied by $r$, which does not affect the splitting questions below. So the extension class of $\cE_r$ is the image of the Gelfand--Fuchs class under this contraction. For $r=1$, $\cE_1$ is the Virasoro algebra $\vir$ with the adjoint action of $\witt$, and it is non-split: a $\witt$-equivariant splitting would be given by a linear form $\phi$ on $\witt$ with $\psi(x,a)=\phi([x,a])$, that is, $\psi$ would be a coboundary, which is not the case.

By Remark~\ref{rmk::Witt::Delta} the module $\Delta^{2d}S^r(\witt)\subset S^r(\witt)$ is a submodule. Denote by $\cE_r^{(d)}:=\pi^{-1}\bigl(\Delta^{2d}S^r(\witt)\bigr)\subset\cE_r$ its preimage; it is the consecutive extension
$$
0\longrightarrow S^{r-1}(\witt)\longrightarrow\cE_r^{(d)}\stackrel{\pi}{\longrightarrow}\Delta^{2d}S^r(\witt)\longrightarrow0.
$$
Since $\Delta^{2d'}S^r(\witt)\subset\Delta^{2d}S^r(\witt)$ for $d'\geq d$, a splitting over $\Delta^{2d}S^r(\witt)$ restricts to a splitting over $\Delta^{2d'}S^r(\witt)$. Hence non-splitness of the extension for $d'$ implies non-splitness for all $d\leq d'$.

\begin{conjecture}[Etingof, private communication]
	\label{conj::Etingof}
	For all $r\geq1$ and $d\geq0$ the extension $0\to S^{r-1}(\witt)\to\cE_r^{(d)}\to\Delta^{2d}S^r(\witt)\to0$ of $\witt$-modules is non-split. By the previous remark, it suffices to prove this for all $d\gg0$.
\end{conjecture}

Conjecture~\ref{conj::Etingof} implies Conjecture~\ref{conj::PS}, and even a stronger statement.

\begin{theorem}
	\label{thm::Etingof::PS}
	Assume Conjecture~\ref{conj::Etingof} and let $c\neq0$. Then every nonzero $\ad(\witt)$-stable subspace $M\subset U_c$ contains $1$. In particular, $U_c$ is simple, that is, Conjecture~\ref{conj::PS} holds. The same holds for the Poisson algebra $P_c$ and the action $\{x,\cdot\}$; in particular, we recover Poisson simplicity of $P_c$.
\end{theorem}
\begin{proof}
	Any nonzero two-sided ideal $J\subset U_c$ is $\ad(\witt)$-stable, so simplicity follows from the first statement. Let $M\neq0$ be an $\ad(\witt)$-stable subspace, and let $r$ be minimal such that $M_r:=M\cap\calF_rU_c\neq0$. If $r=0$, then $M\supset\calF_0U_c=\Bbbk\cdot1$, and we are done. Suppose $r\geq1$.
	
	By minimality of $r$ we have $M_r\cap\calF_{r-1}U_c=0$, and therefore $M_r\cap\calF_{r-2}U_c=0$. Thus the two maps
	$$
	M_r\hookrightarrow\calF_rU_c/\calF_{r-1}U_c=S^r(\witt),\qquad M_r\hookrightarrow\calF_rU_c/\calF_{r-2}U_c\simeq\cE_r
	$$
	are injective and $\witt$-equivariant. Let $N_r\subset S^r(\witt)$ be the image of the first map; it is a nonzero submodule. By Theorem~\ref{thm::Witt::submod} applied to $\lambda=-1$, $N_r\supset\Delta^{2d}S^r(\witt)$ for some $d$. Let $\tilde N_r\subset\cE_r$ be the image of the second map. It is a submodule, and $\pi$ maps it isomorphically onto $N_r$. Therefore the preimage of $\Delta^{2d}S^r(\witt)$ in $\tilde N_r$ is a submodule of $\cE_r^{(d)}$ that is mapped by $\pi$ isomorphically onto $\Delta^{2d}S^r(\witt)$. The inverse of this isomorphism is a $\witt$-equivariant section of $\pi$, that is, a splitting of the extension $\cE_r^{(d)}$. This contradicts Conjecture~\ref{conj::Etingof}. Hence $r=0$.
	
	For $P_c$ the argument is the same, using the filtration by degree and the second part of Lemma~\ref{lem::Witt::ad}.
\end{proof}

\begin{remark}
	For $c=0$ the extension $\cE_r^0$ splits, in accordance with the fact that the algebra $U_0=U(\witt)$ is not simple. In the cohomological language, Conjecture~\ref{conj::Etingof} asserts that the class of $\cE_r$ in $\mathrm{Ext}^1_{\witt}(S^r(\witt),S^{r-1}(\witt))=H^1\bigl(\witt;\mathrm{Hom}(S^r(\witt),S^{r-1}(\witt))\bigr)$ remains nonzero after the restriction to $\Delta^{2d}S^r(\witt)$.
\end{remark}

\end{document}